\documentclass[a4paper, 10pt, reqno, final]{article}

\usepackage[T1]{fontenc}		     
\usepackage[utf8]{inputenc}			 
\usepackage[american]{babel}
\usepackage{float}
\usepackage{caption}
\usepackage{subcaption}

\usepackage{pgfplots}
\usepgfplotslibrary{groupplots}
\usepackage{amsmath}
\pgfplotsset{compat=1.18}

\usetikzlibrary{decorations.pathmorphing}

\usepackage{amsmath}
\usepackage{amssymb}
\usepackage{amsthm}
\usepackage{tikz-cd}
\usepackage{multicol}

	\theoremstyle{plain}
		\newtheorem{mainthm}{\textsc{Theorem}}		
				
				\newtheorem{thm}{Theorem}[section]	
						\newtheorem{cor}[thm]{Corollary}	
				
				\newtheorem{lem}[thm]{Lemma}

\newtheorem{addlem}{Supplementary lemma}

						\newtheorem{prop}[thm]{Proposition}

	\theoremstyle{definition}
		\newtheorem{defn}[thm]{Definition}	
				\newtheorem{ex}[thm]{Example}		
					\theoremstyle{remark}
		\newtheorem{rem}[thm]{Remark}		
				\newtheorem{note}[thm]{Notation}		
				\numberwithin{equation}{section}

				\usepackage{mathtools}

				\usepackage{mathrsfs}		
				
\usepackage{eucal}			

\usepackage{braket}	

\usepackage[all,pdf]{xy}		

\usepackage{mparhack}		
\usepackage{relsize}			

\usepackage{a4wide}		

\usepackage{booktabs}		
\usepackage{multirow}		
\usepackage{caption}		
\usepackage{rotating}

\usepackage{footmisc}

\usepackage{graphicx}		
\usepackage{epsfig}

\usepackage{enumerate}		
\usepackage[colorlinks=true, linkcolor=black, citecolor=black,
urlcolor=blue]{hyperref}		
\mathtoolsset{showonlyrefs}

\usepackage{varioref}		

\newcommand{\nullity}{\mathrm{n}_0\,}

\newcommand{\noo}[1]{\overset {\mbox{
\lower1pt\hbox{${\scriptstyle o}$}}}{\mathrm n}_{\mbox{
\lower2pt\hbox{$\scriptstyle#1$}}}}

\newcommand{\coindex}{\mathrm{n}_+\,}
\newcommand{\itriple}{\iota}
		
\newcommand{\Find}{\mathrm{ind}\,}
\newcommand{\trasp}[1]{{#1}^\mathsf{T}}	
\newcommand{\iMor}{\mathrm{n_-}}		
\newcommand{\R}{\mathbb{R}}		
\newcommand{\ZZ}{\mathbf{Z}}
\newcommand{\C}{\mathbf{C}}

\newcommand{\Sym}{\mathrm{S}}
\newcommand{\Mat}{\mathrm{M}}

\newcommand{\N}{\mathbb{N}}		
\newcommand{\iCLM}{\mu^{\scriptscriptstyle{\mathrm{CLM}}}}

\DeclareMathOperator{\sgn}{sgn}

\newcommand{\coiMor}{\mathrm{n_+}}
\newcommand{\Id}{I}
		
\DeclareMathOperator{\spfl}{sf}			
	
\DeclareMathOperator{\CFs}{\mathcal{CF}^{s}}	
\DeclareMathOperator{\CFsa}{\mathcal{CF}^{sa}}	
\DeclareMathOperator{\BFsa}{\mathcal{BF}^{sa}}
\DeclareMathOperator{\BF}{\mathcal{BF}}

\newcommand{\cfsa}{\mathcal{CF}^{sa}}
\newcommand{\Bou}{\mathcal{B}}
\newcommand{\Bsa}{\mathcal{B}^{sa}}

\newcommand{\irel}{I}
\newcommand{\iindex}[1]{\mu_{\scriptscriptstyle{\mathrm{Mor}}}\left[#1\right]}
\newcommand{\coiindex}[1]{\mathrm{n_+}\left[#1\right]}
\newcommand{\Real}{\mathrm{Re}}
\newcommand{\ind}{\mathrm{ind\,}}   
\newcommand{\Imm}{\mathrm{Im}}

\newcommand{\norm}[1]{\lVert#1\rVert}

\def\sq(H3){\sqrt{-1}}

\DeclareMathOperator{\rk}{rank}	
\def\={:=}
\def\>{\supset}
\def\<{\subset}

\def\12{\dfrac{1}{2}}
\def\0{^{\circ}}

\def\CC{{\mathbb C}}

\def\NN{{\mathbb N}}

\def\RR{{\mathbb R}}

\def\ZZ{{\mathbb Z}}

\def\C{\CC}

\def\f{\varphi}

\def\N{\NN}

\def\R{\RR}

\def\V{\sqrt}

\def\z{\zeta}
\def\Z{\ZZ}

\def\Cl{\mbox{\rm C$\ell$}}

\DeclareMathOperator{\codim}{codim} \DeclareMathOperator{\coker}{coker}

 \DeclareMathOperator{\dist}{dist}

\newcommand{\dom}{D}

\DeclareMathOperator{\Graph}{gr}
\newcommand{\Grn}[1]{\mathcal G({#1})}
\DeclareMathOperator{\Grass}{\mathcal S}

\DeclareMathOperator{\image}{im}

\DeclareMathOperator{\Lag}{\mathrm{Lag}}

 \DeclareMathOperator{\rank}{rank}

\renewcommand{\Cl}{\mathcal{Cl}} 

\DeclareMathOperator{\Span}{span}

\DeclareMathOperator{\supp}{supp}

\DeclareMathOperator{\Tr}{\mathcal T}

\makeatletter
\def\namedlabel#1#2{\begingroup
    #2
    \def\@currentlabel{#2}
    \phantomsection\label{#1}\endgroup
}
\makeatother

\title{Index theory for singular Lagrangian systems and Bessel-type differential operators}
\author{Xijun Hu\thanks{The author is partially supported by  supported by the National Key R\&D Program of China(2020YFA0713303).} , Alessandro Portaluri
\thanks{The author is partially supported by Progetto di Ricerca GNAMPA - INdAM, codice CUP\_E55F22000270001 “Dinamica simbolica e soluzioni periodiche per problemi singolari della Meccanica Celeste”.}, Li Wu\thanks{The author is partially supported by  supported by the  National Natural Science Foundation of China NSFC N.12171281.}}
\date{\today}

\renewcommand{\Cl}{\mathcal{C}\ell}
\begin{document}
 \maketitle

\begin{abstract}
The aim of the present manuscript is to develop an index theory for singular Lagrangian systems, with a particular focus on the important class of singular operators given by Bessel-type differential operators. The main motivation is to address several challenges posed by singular operators, which appear in a wide range of applications: celestial mechanics (for instance, perturbations in planetary motion), oscillatory systems with time-dependent forcing, electromagnetism (such as wave equations in nonuniform media), and quantum mechanics (notably certain Schrodinger equations with periodic potentials).

We pursue two principal objectives. First, we establish a spectral flow formula and a Morse Index Theorem for gap-continuous paths of singular Sturm–Liouville operators. By means of these index formulas, we construct a Morse index theory for a broad class of Bessel-type differential operators and apply it to a family of asymptotic solutions of the gravitational $N$-body problem.

Finally, our new index theory provides new insight into a phenomenon first observed by Rellich concerning the spectrum of one-parameter families of Sturm–Liouville operators with varying domains.

\end{abstract}

\tableofcontents

\section{Introduction}

In its classical form, the Morse Index Theorem asserts that the number of negative eigenvalues of the Hessian of the geodesic action functional at a critical point equals the number of conjugate points, counted with multiplicity, along the corresponding geodesic. Through standard spectral flow techniques, this integer can be expressed as the sum of the dimension of the intersections between the graph of the fundamental solution with the Lagrangian subspace corresponding to the boundary condition.  For \emph{singular} Sturm–Liouville operators, however, this classical framework fails, and the usual spectral flow identities no longer apply.

The present paper has two main objectives:
\begin{enumerate}
  \item[(O1)] To construct a Lagrangian intersection theory that provides the natural framework for both regular and (one-sided) singular Sturm–Liouville (SL) boundary value problems, and within this setting to establish a Spectral Flow Formula and a Morse Index Theorem with general Lagrangian boundary conditions.
  \item[(O2)] To develop a Morse Index Theory for Bessel-type differential operators and to apply to some classes of motions in the gravitational $n$-body problem.
\end{enumerate}

As a direct consequence of this index theory, we rigorously capture the Rellich  phenomenon—the “disappearance” of eigenvalues at \(-\infty\)—thereby refining the relation between spectral flow and the difference of Morse indices in this broader functional-analytic framework. This spectral effect arises when considering unbounded perturbations of closed, self-adjoint, lower-bounded Fredholm operators, which frequently appear in applications, especially in families of operators with varying domains. Finally, we compute the Morse index for a class of unbounded motions in the gravitational $N$-body problem.

The motivations for this theory are twofold:
\begin{itemize}
\item[(M1)] To classify, via the Morse index, all possible motions in the gravitational $N$-body problem—whether partially or totally colliding, or unbounded.
\item[(M2)] To use this classification to \emph{a priori} exclude collisions along solutions obtained through topological methods in the Calculus of Variations.
\end{itemize}
To the best of authors knowledge the index theory in this singular context has been investigated in \cite{HS22}. In this paper the authors investigated  a general class of linear Hamiltonian systems on intervals with at least one singular endpoint which can be limit-point, limit-circle, or limit-intermediate. In particular, they proved some 
renormalized oscillation results by using  a  Lagrangian intersection theory in the classical $\C^{2n}$ symplectic space.

\subsection{Highlights and Main Results}

After setting the stage in Section~\ref{sec:Abstract-sf}, we establish in Section~\ref{sec:sf-SL-operators} a spectral flow formula for a family \( s \mapsto l_s \) of one-sided singular Sturm–Liouville operators defined by the differential expression
\begin{equation}\label{eq:sturm-liouville-operator-intro}
l_s := -\frac{d}{dt}\Bigl(P(t)\frac{d}{dt} + Q(t)\Bigr) + Q(t)^{T}\frac{d}{dt} + R(t) + C_s(t), \qquad t \in (a,b),
\end{equation}
with \( a, b \in \overline{\mathbb{R}} \) and \( s \in [0,1] \), where
\begin{multline}\label{eq:assumptions-coeff-intro}
	P^{-1} \in \mathscr{C}^1\big((a,b), \Sym_n(\mathbb{R})\big), \qquad
	Q \in \mathscr{C}^1\big((a,b), \Mat_n(\mathbb{R})\big), \qquad R \in \mathscr{C}^0\big((a,b), \Sym_n(\mathbb{R})\big), \\[3pt]
	\text{and } \ s \mapsto C_s \in \mathscr{C}^0_b\big((a,b), \Sym_n(\mathbb{R})\big) \ \text{is continuous},
\end{multline}
where \( \mathscr{C}^0_b \) denotes the set of bounded continuous functions, and \( C_0(t) = 0 \) for all \( t \in (a,b) \).

Under the regularity assumptions in Equation~\eqref{eq:assumptions-coeff-intro}, the formal differential expression \( l_s \) in Equation~\eqref{eq:sturm-liouville-operator-intro} defines an eventually singular second-order Sturm–Liouville differential operator (abbreviated as {\sc SL-operator}).

In the Hilbert space \( H = L^2((a,b), \mathbb{R}^n) \), let \( L_s \) (respectively \( L_s^* \)) denote the minimal (respectively maximal) operator with domains \( \dom(L_s) \) and \( \dom(L_s^*) \) induced by \( l_s \), and introduce the following hypotheses:
\begin{itemize}
	\item[{\sc (H1)}] The endpoint   $b$ is finite. The coefficients $P,Q,R,C_s$ extend continuously to $b$, and $\int_c^b\|P(t)^{-1}\|\,dt<\infty$ for some interior $c$. Thus $b$ is regular in the integrable-coefficient sense; $P(b)$ need not be invertible .
	\item[{\sc (H2)}] The path \( [0,1] \ni s \mapsto L_s \in \CFs(H) \) is gap-continuous and consists of symmetric Fredholm operators.
\end{itemize}

We note that \( \dom(L_s) \) is independent of \( s \) and coincides with \( \dom(L) \), where \( L \) denotes the unperturbed Sturm–Liouville operator. On the space \( \dom(L^*) \), there is a well-defined skew-symmetric bilinear form given by
\begin{equation}\label{eq:symplectic-form-SL}
\omega(x, y) = \langle x, L^* y \rangle - \langle L^* x, y \rangle,
\end{equation}
which, in general, is not symplectic since it is typically degenerate.

One of the main contributions of this first part of the paper is the {\em explicit construction of a symplectic subspace} \( (W, \omega|_W) \) through the \( \ker L^* \), that is, the space of solutions to the Sturm–Liouville equation (see Section~\ref{sec:coordinate-map} for further details), and of a {\em trace map} that we now describe.

Under assumptions (H1) and (H2), the following crucial {\em decomposition formula} holds:
\[
\dom(L^*) = \dom(L) \oplus W,
\]
where \( (W, \omega|_W) \) is a \( (2k) \)-dimensional symplectic space with \( n \le k \le 2n \).

\begin{ex}
In the special case   $P,Q\in C^1([a,b])$, $P$ invertible on $[a,b]$, and continuous zero-order coefficients,  we have
\begin{multline}
\dom(L_s^*) = W^{2,2}([a,b], \mathbb{R}^n), \qquad
\dom(L_s) = W^{2,2}_0([a,b], \mathbb{R}^n), \\[3pt]
\text{and} \qquad W \cong T^*\mathbb{R}^n \oplus T^*\mathbb{R}^n,
\end{multline}
where the isomorphism between \( W \) and \( T^*\mathbb{R}^n \oplus T^*\mathbb{R}^n \) is defined by
\begin{multline}
W \ni x \longmapsto \phi(x) = \big(x^{[1]}(a), x(a), x^{[1]}(b), x(b)\big), \\[3pt]
\text{where } \quad x^{[1]}(c) = P(c)\, x'(c) + Q(c)\, x(c), \quad \text{for } c = a,b.
\end{multline}
The vector \( x^{[1]}(c) \) is referred to as the {\sc quasi-derivative at \( c \)}.

Integrating by parts in Equation~\eqref{eq:symplectic-form-SL}, we obtain
\begin{multline}
\omega(f,g) = [f,g](b) - [f,g](a), \qquad \text{where} \\[3pt]
[f,g](t) := \langle P(t) f'(t) + Q(t)f(t), g(t) \rangle - \langle f(t), P(t) g'(t) + Q(t)g(t) \rangle.
\end{multline}
\end{ex}

Let \( Z = \Span\{z_i \mid 1 \le i \le 2n\} \), where each   $z_i$ is a maximal-domain cutoff of a homogeneous solution near $b$, vanishing on $(a,c_0]$ for an interior $c_0$  and such that the family \( \{(z_i^{[1]}(b), z_i(b))\}_{i=1}^{2n} \) forms a basis of \( \mathbb{R}^{2n} \).

\begin{itemize}
\item If \( k = n \), on the symplectic space \( (W, \omega|_W) \) we define the {\bf trace map} \( \Tr \) by
\[
W \ni f \longmapsto \Tr(f) := \big(\omega(f,z_1), \ldots, \omega(f,z_{2n})\big)
= \big([f,z_1](b), \ldots, [f,z_{2n}](b)\big) \in \mathbb{R}^{2n}.
\]
This case corresponds to the classical {\sc limit point case} in the Weyl classification of the boundary conditions  where no boundary condition is required at $a$.
\item If \( k \neq n \), since \( \omega|_Z \) is nondegenerate, we obtain the following symplectic decomposition:
\begin{equation}\label{eq:decomp-W-intro}
W = Z^{\omega|_W} \oplus Z.
\end{equation}
Setting
\[
[f,g](a^+) := \lim_{t \to a^+} [f,g](t),
\]
we define the (linear) {\bf trace map} \( \Tr : W \to \mathbb{R}^{2k-2n} \oplus \mathbb{R}^{2n} \) by
\[
\Tr(f) := \big(-[f,y_1](a^+), \ldots, -[f,y_{2k-2n}](a^+), [f,z_1](b), \ldots, [f,z_{2n}](b)\big),
\]
where   $y_i\in\ker L^*$ are selected as in  Lemma~\ref{thm:decomp-factor-space}   and  Remark~\ref{rmk:Selection} .  The adjusted vectors $\widetilde y_i\in\dom(L^*)$ have zero right Cauchy data and need not be homogeneous solutions. The case $k=2n$ is limit circle (including regular endpoints); $n<k<2n$ is limit intermediate.
 \end{itemize}

Since \( \omega|_W \) is nondegenerate, it follows that
\( \Tr : Z^{\omega|_W} \oplus Z \to \mathbb{R}^{2k-2n} \oplus \mathbb{R}^{2n} \)
is a bijection, and the decomposition is \(\Tr\)-invariant. Moreover, the space
\[
\mathbb{R}^{2k-2n} \oplus \mathbb{R}^{2n}
\]
is the {\sc space of boundary data}, where \( \mathbb{R}^{2k-2n} \) (respectively \( \mathbb{R}^{2n} \)) contains the boundary data at \( a \) (respectively \( b \)).

With all these ingredients, we are now in a position to construct a proper Lagrangian intersection theory and to prove a {\em general spectral flow formula for Sturm–Liouville operators}.

Let \( \Lag(W, \omega|_W) \) denote the Lagrangian Grassmannian of \( (W, \omega|_W) \). For each \( s \), consider the self-adjoint extension of \( L_s \) defined by
\[
L_{s,\Lambda_s} := L_s^*|_{\dom(L) \oplus \Lambda_s}.
\]
We observe that if \( s \mapsto \Lambda_s \in \Lag(W, \omega|_W) \) is gap-continuous, then
\( s \mapsto L_{s,\Lambda_s} := L_s^*|_{\dom(L) \oplus \Lambda_s} \)
is a gap-continuous path of operators in \( \CFsa(H) \). Setting \( V_s = \ker L_s^* \), the main result of Section~\ref{sec:sf-SL-operators} reads as follows.

\begin{mainthm}\label{thm:Sturm_Sf_formula-intro}
Under conditions {\rm (H1)}–{\rm (H2)} and assuming that
\( s \mapsto \Lambda_s \in \Lag(W, \omega|_W) \)
is gap-continuous, the following spectral flow formula holds:
\[
\spfl(L_{s,\Lambda_s}, s \in [0,1])
= -\iCLM(\Tr(\Lambda_s), \Tr(V_s), \rho^{\Tr}, s \in [0,1]),
\]
where \( \rho^{\Tr} \) denotes the push-forward symplectic form.
\end{mainthm}
In the spectral flow formula stated in Theorem~\ref{thm:Sturm_Sf_formula-intro}, the decomposition $\dom(L^*) = \dom(L) \oplus W$ plays an essential role. It is worth mentioning that there exists a more abstract approach to proving a spectral flow formula. In a fairly general setting, this was developed by Furutani in~\cite{Fur04}, by Booss-Bavnbek, Lesch, and Phillips in~\cite{BLP05}, and, in the very general context of symplectic Banach spaces, by Booss-Bavnbek and Zhu in~\cite{BZ18}. In the aforementioned works, the intersection theory is established in the {\em Gelfand–Robbin quotient space} 
\[
\beta(L) := \dom(L^*) / \dom(L).
\]

By comparing the method developed in this manuscript with those presented in the above papers, our approach offers several advantages:
\begin{enumerate}
    \item The definition of the Maslov index is concrete and expressed in terms of the paths $s \mapsto V_s$ and $s \mapsto \Lambda_s$, both of which are contained in $\dom(L^*)$.  
    \item This approach can be extended to the case of varying domains $s \mapsto \dom(L_s)$, a generalization that will be addressed in a forthcoming paper. 
    \item The space $W$ is defined explicitly by using  a basis of $\ker L^*$. Moreover, since the entire construction depends only on the domains of $L^*$ and $L$, the definition of $W$ remains valid if $L$ is replaced by another operator $\widetilde{L}$ (not necessarily belonging to the same family) provided that $\dom(\widetilde{L}) = \dom(L)$.
\end{enumerate}

The third remark is particularly useful, since in most cases it is extremely difficult—if not impossible—to determine an explicit basis of $\ker L^*$. Replacing $L$ by an operator $\widetilde{L}$ having the same domain resolves this difficulty: if $\widetilde{L}^*$ is suitably chosen, then $\ker \widetilde{L}^*$ and hence $W$ can be computed explicitly. This idea has been applied, for instance, in Subsection~\ref{subsec:properties-bessel}.

\medskip

In Subsection~\ref{subsec:examples}, we prove that the spectral flow formula stated in Theorem~\ref{thm:Sturm_Sf_formula-intro} generalizes several classical results available in the literature. In particular, it encompasses the following cases:
\begin{enumerate}
	\item  {\sc Regular Sturm–Liouville operators} on bounded domains; 
	\item  {\sc Sturm–Liouville operators} on $\R^+$
    \item {\sc Bessel-type operators} on $(0,1]$ and $[1,\infty)$.
\end{enumerate}
The regular case is well understood, and several precise and comprehensive results are available in the literature; see, for instance, \cite{CLM94, HP17, LZ00, RS95, Zhu06} and the references therein. In applications, this first case typically arises in the study of trajectories of Lagrangian systems defined on bounded time intervals. 

The second case, however, appears naturally in the study of trajectories of Lagrangian systems defined on unbounded time intervals. This is the case, for example, of half-clinic solutions of Lagrangian systems, that is, solutions parametrized by the half-line. Recently, the first two authors established a general spectral flow formula in~\cite{HP17}, which became a key ingredient in several subsequent results concerning the Morse index theorem in the gravitational $N$-body problem (cf.~\cite{HP17, HO16, HOY21} and references therein).

\medskip

Using the above spectral flow formula, we prove in Section~\ref{sec:Morse-dirichlet} the {\bf Morse index theorem for singular Sturm–Liouville operators.} For the regular case or the heteroclinic and half-clinic cases,  we refer the interested reader to \cite{HWY20} , \cite{HPWX20}, \cite{BCJLMS18} and  \cite{HP17}.

We start with a single Sturm–Liouville operator \( l \) defined on the function space   \( \mathscr{C}_0^\infty((a,b), \mathbb{R}^n) \) , where the endpoint \( t = b \) is regular according to hypothesis (H1). In this context, by restricting the operator \( l \) to the interval \( (a, \sigma) \) with \( \sigma \in (a, b] \), we obtain a one-parameter family of operators \( \sigma \mapsto l_{(a,\sigma)} \). This construction is necessary for the application of the spectral flow formula.

For $\sigma \in (a,b)$, we consider $l_{(a,\sigma)}$ to be the one-sided singular Sturm–Liouville differential operator obtained by restricting $l$ to $ \mathscr C_0^\infty((a,\sigma), \R^n)$, and we denote by $L_{(a,\sigma)}$ the minimal operator associated with $l_{(a,\sigma)}$. We observe that $L_{(a,\sigma)}$ is singular only at the initial point $t=a$ (hence the final endpoint $\sigma$ is regular). We consider the Lagrangian subspaces 
\[
\Lambda_0 \in \Lag(\R^{2k-2n}), \qquad 
\Lambda_D := \R^{n} \times (0), \qquad 
\Lambda = \Lambda_0 \oplus \Lambda_D,
\]
where $\Lambda_D$ is the standard {\sc Dirichlet Lagrangian}. 

We introduce the following conditions:
\begin{itemize}
\item[-]{(H3)} The quadratic form on $\mathscr C^\infty_0((a,b), \R^n)$ associated with the SL-operator $l$ is {\sc bounded from below} in $L^2$.
\item[-]{(H4)} The minimal operator $L_b \in \CFs(L^2((a,b), \R^n))$.
\end{itemize}

\begin{rem}
It is worth observing that, under (H3), for every $\sigma \in (a,b]$, the quadratic form induced by $l_{(a,\sigma)}$ is bounded from below. In particular, there exists a well-defined self-adjoint Friedrichs extension $L_{(a,\sigma)}$ (cf. Section~\ref{appendix:friedrichs} for further details).
\end{rem}

Given $\Lambda \in \Lag(\R^{2k-2n} \oplus \R^{2n}, -\Omega \oplus \Omega)$, we set 
\[
L_{(a,\sigma),\Lambda} = L_{(a,\sigma)}^*|_{\dom(L_{(a,\sigma)}) \oplus \Tr_\sigma^{-1}(\Lambda)},
\]
where $\Tr_\sigma : W_\sigma \to \R^{2k-2n} \oplus \R^{2n}$ is the trace map defined on the subspace $W_\sigma$ associated with the operator $l_{(a,\sigma)}$ and constructed above.

\begin{note}
In what follows we denote by 
\begin{itemize}
\item $L_{(a,\sigma),F}$ the self-adjoint Friedrichs extension associated with $L_{(a,\sigma)}$; 
\item $L_{(a,\sigma),\Lambda}$ the self-adjoint extension of $L_{(a,\sigma)}$ with boundary condition $\Lambda := \Lambda_0 \oplus \Lambda_D$.
\end{itemize}
We denote by $\iMor(\#)$ the {\sc Morse index} of the self-adjoint operator $\#$, namely, the maximal dimension of the negative spectral subspace of $\#$.
\end{note}

Thus, the Morse index theorem reads as follows.

\begin{mainthm}{\bf [Singular Morse Index Theorem]}\label{thm:main-oneside-intro}
Assume that the SL-operator $l$ satisfies conditions (H1), (H3), and (H4), and that Dirichlet boundary conditions are imposed at $b$.  Assume $\iMor(L_{b,\Lambda})<\infty$.  Then the following equality holds:
\[
\iMor(L_{b,\Lambda}) = \sum_{\sigma \in (a,b)} \dim \ker L_{(a,\sigma),\Lambda}.
\]
\end{mainthm}
The Theorem~\ref{thm:main-oneside-intro} works even if $\iMor(L_b)=+\infty$ and it is obtained by combining Theorem~\ref{thm:main-oneside} and  Proposition~\ref{thm:morse-infinito}.
It is worth noting that, even in this more general setting, Theorem~\ref{thm:main-oneside-intro} retains, at least formally, a structure closely resembling the classical result in the regular case (cf.~\cite{BOPW21, HP17} and the references therein). However, the traditional method of proving such a theorem—namely, by computing   local  crossing forms—  requires an additional localization and limiting argument at the singular endpoint .

\medskip

The next result relates the Morse index of the Friedrichs extension operator $L_F$ with the limit, as $\sigma$ tends to the endpoint, of the Friedrichs operators $L_{(a,\sigma),F}$ and $L_{(\sigma,b),F}$, under the sole assumption that the operator $L$ is bounded from below. In particular, in Theorem~\ref{thm:limit_morse_index-intro}, {\em neither Fredholmness nor the finiteness of the Morse index has been assumed}.

\begin{mainthm}\label{thm:limit_morse_index-intro}
Assume condition (H3). If $l$ is regular at $a$ and Dirichlet boundary conditions are imposed, then
\[
\iMor(L_F) = \lim_{\sigma \to b^-} \iMor(L_{(a,\sigma),F}) = \sum_{t \in (a,b)} \dim\big(\gamma(t)\Lambda_D \cap \Lambda_D\big),
\]
where $\gamma$ is the fundamental matrix solution satisfying $\gamma(a) = \Id$.

Analogously, if $l$ is regular at $b$ and Dirichlet boundary conditions are imposed, then
\[
\iMor(L_F) = \lim_{\sigma \to a^+} \iMor(L_{(\sigma,b),F}) = \sum_{t \in (a,b)} \dim\big(\gamma(t)\Lambda_D \cap \Lambda_D\big),
\]
where $\gamma$ is the fundamental matrix solution satisfying $\gamma(b) = \Id$.
\end{mainthm}
In \cite[Theorem 10.12.1]{Zet05} the author collect the spectral properties and well-known index results for singular scalar Sturm-Liouville operators. The singular Morse Index Theorem can be seen as a multi-dimensional generalization of the aforementioned result. 

The case of general Lagrangian boundary conditions is treated in detail in Section~\ref{sec:morse-general}. In particular, we prove that even in this singular framework, the Morse index with general boundary conditions differs from the Dirichlet case by a correction term computed via the triple index of the Cauchy data space and the projections of the domains of $L_\Lambda$ and of the Friedrichs extension $L_F$. 

Section~\ref{sec:Rellich } is devoted to study a rigorous quantification of a phenomenon first observed by Rellich in the case of  singular Sturm-Liouville equations. In the abstract terms the stage is the following. If \(T\) is a closed self-adjoint operator bounded from below, and \(A\) is a relatively bounded symmetric perturbation with sufficiently small relative bound, then
$S = T + A$ remains self-adjoint and bounded from below (see \cite[Chapter~5, Section~4, Remark~4.13]{Kat80}).

However, this property does not necessarily hold for unbounded perturbations. As Kato explains in~\cite[Chapter~5, Section~4, Remark~4.13]{Kat80}, based on Rellich’s example, when a sequence of self-adjoint operators $T_n$ converges in the gap topology to a self-adjoint operator bounded from below, the lower bounds of $T_n$ may diverge to $-\infty$, and eigenvalues may “disappear” at \(-\infty\). Rellich’s scalar counterexample in the context of Sturm–Liouville operators illustrates this effect,   and motivates the singular boundary count established below . Section~\ref{sec:Rellich } provides a precise quantification of the multiplicity of eigenvalues that escape to \(-\infty\).   In the regular case, this phenomenon has been already studied at \cite{HLWZ19, HLWZ22}. 

This phenomenon occurs, for instance, when the domains of the operators vary with the parameter. Consider $A \in \CFs(H)$ and a continuous path $s \mapsto \Lambda_s$ of Lagrangian subspaces of $(U,\rho)$. The corresponding path $s \mapsto A_{\Lambda_s}$ of self-adjoint extensions is gap-continuous, but the operators are no longer relatively bounded perturbations of a fixed operator. Such situations arise naturally in applications. The main result of Section~\ref{sec:Rellich } (Theorem~\ref{thm:Rellich }) computes the total number of such {\em limiting ghost eigenvalues} that disappear to $-\infty$ through the Morse Index Theorem for singular Sturm–Liouville operators with general boundary conditions.

\medskip

Section~\ref{sec:Bessel} is devoted to a detailed study of regular perturbations of the unperturbed Bessel operator
\[
l_q := -\frac{d^2}{dt^2} + \frac{q}{t^2},
\]
acting on   $C_0^\infty((0,1),\R^n)$ , which is particularly relevant in the analysis of the variational and index properties of unbounded motions in the gravitational $N$-body problem. Denoting by $L_q : \dom(L_q) \subset H \to H$ the corresponding minimal operator, where $H = L^2((0,1], \R^n)$, we investigate its functional-analytic and spectral properties, with particular emphasis on Fredholmness and the Morse index. The main results of this section are the following.

\begin{mainthm}\label{thm:Fredholmness-Bessel-perturbation-intro}
Let $R \in \mathscr C^0([0,1], \Sym(n))$, and let $L_R$ be the minimal operator induced by the Bessel-type operator
\[
l_R := -\frac{d^2}{dt^2} + \frac{R(t)}{t^2}, \qquad t \in (0,1].
\]
If   $R(0)>-\tfrac14 I$ , then every self-adjoint extension of $L_R$ is a Fredholm operator.
\end{mainthm}

The final main results of Section~\ref{sec:Bessel} concerns the Morse index properties of Bessel-type operators on   $\mathscr C_0^\infty((1,+\infty), \R^n)$ and on $\mathscr C_0^\infty((0,1),\R^n)$ , respectively. 

\begin{mainthm}\label{thm:bessel-morse-index-intro}
Let $l_R$ be the operator defined above, acting on   $\mathscr C_0^\infty((1,+\infty), \R^n)$ . 
\begin{itemize}
\item[-]\textbf{Case 1.} If $\liminf_{t \to +\infty} R(t) > -1/4$, then every self-adjoint extension of $l_R$ has finite Morse index.
\item[-]\textbf{Case 2.} If $\limsup_{t \to +\infty} R(t) < -1/4$, then every self-adjoint extension of $l_R$ has infinite Morse index.
\end{itemize}
\end{mainthm}

\begin{mainthm}\label{thm:bessel-morse-index-on-0-1-intro}
Let $l_R$ be the operator defined above on   $\mathscr C_0^\infty((0,1),\R^n)$ . 
\begin{itemize}
\item{\bf Case 1.} If  $\liminf_{t\to 0^+} R(t)>-1/4 $, then any self-adjoint extension  has finite Morse index.
\item{\bf Case 2.} If $\limsup_{t\to 0^+} R(t)<-1/4  $, then any self-adjoint extension has infinite Morse index.
\end{itemize}
\end{mainthm}

Using these properties together with the developed index theory, we derive new results on the Morse index of totally colliding and unbounded motions in the $N$-body problem (Section~\ref{sec:N-bp}). In particular, we provide both a new proof of the main results of \cite{BHPT20, HOY21} and a significant generalization to more general self-adjoint boundary conditions. The main result of Section~\ref{sec:N-bp} is the following.

\begin{mainthm}\label{thm:Morse-asymptotic-intro}
Let   $q\in C^2([0,T),\hat{\mathcal X})$ be a collision-free solution with a total collision at $T<\infty$  or a parabolic   end at $T=\infty$, having limiting collision-free  normalized central configuration $a$, and set
\[
\bar B(a) := \frac{2}{9}\,\frac{M^{-1} D^2 U(a)}{U(a)}.
\]
If
\[
\bar B(a) > -\dfrac{1}{4},
\]
then the Morse index $\iMor(q)$ of $q$ is finite, and the following identity holds:
  \[
\iMor(q) = \sum_{t \in (0,T)} \dim\big(\gamma(t)\Lambda_D \cap \Lambda_D\big),
\] 
where $\gamma$ is the fundamental matrix solution satisfying   $\gamma(0)=\Id$ .

If   \[
\lambda_{\min}(\bar B(a)) < -\dfrac{1}{4},
\] 
the Morse index is $\infty$.

 Here $\bar B(a)$ is self-adjoint for the mass inner product, and its eigenvalues are those of $M^{1/2}\bar B(a)M^{-1/2}$. The index is taken on compactly supported variations in $(0,T)$; in the finite case it is the Friedrichs/Dirichlet index.  If $q$ is hyperbolic  with regular initial point $0$ , then the Morse index $\iMor(q)$ of $q$ is always finite.
\end{mainthm}

\medskip

The Appendix collects the most technical material. We first gather definitions, notation, and key results on symplectic indices and spectral flow used throughout the paper. Section~\ref{appendix:friedrichs} establishes the fundamental properties of Friedrichs extensions, which play a central role in our analysis, while Section~\ref{sec:postponed-proofs} contains postponed technical proofs in order to streamline the main exposition.

\medskip

{  \footnotesize{
\subsubsection*{List of Notation}
{\footnotesize{
\begin{itemize}
\item $\overline \R:=\R \cup\{-\infty, +\infty\}$. $J=[a,b]\subset \overline \R$,  $I=[0,1]\subset \R$. 
\item $\Mat(n)$ and $ \Sym(n)$ the vector space of all (real) matrices and  symmetric matrices, respectively. 
\item $<$ denotes inclusion between linear subspaces to distinguish by the set-theoretical inclusion $\subset$
    \item $(\R^{2n}, \Omega)$ be the standard symplectic space. With abuse of notation we denote by the same symbol $\Omega$ the standard symplectic form of $\R^{2k-2n}$ for $n \le k\le 2n$.
\item \( H = L^2((a,b), \mathbb{R}^n) \).
\item $\Grass(H )$ the set of all closed subspaces of $H$
	\item $\Bou(H)$, $\Cl(H)$ (resp. $\Cl^s(H)$) be respectively the set of all linear bounded, closed (resp. closed and symmetric) densely defined operators in $H$. $\Bsa(H)$ denotes the set of all bounded and self-adjoint  operators on $H$ 
	\item $\dom(A)$ (resp. $\dom(A^*)$) denotes the domain of $A$ (resp. of $A^*$, the adjoint of $A$). $\langle \cdot,\cdot \rangle^G$ graph inner product of $A$.
    \item $p: \dom(A^*)\to U $ denotes the quotient projection. Let $\omega$ denote the symplectic form induced by $A$ on $\dom(A^*)$ 
	\item $\Lag$  denotes the space of all Lagrangian subspaces  
	\item $\CFs(H)$ denotes the set of all closed densely defined and symmetric Fredholm  operators on $H$. $\BFsa(H)$ denotes the set of all bounded and self-adjoint   Fredholm  operators on $H$
	\item For $s \in I$, let $l_s$  denotes a family of Sturm-Liouville operators. We denote by 
$L_s$ (resp. $L_s^*$) the {\sc minimal} (resp. {\sc maximal}) {\sc operator associated to $l_s$}. 
\item $\Lambda_D=\R^n \times (0)$ denotes the Dirichlet and $\Lambda_N=(0)\times \R^n$ the Neumann Lagrangian  subspaces of $(\R^{2n}, \Omega)$. With abuse of notation we denote by the same symbols $\Lambda_D$ (resp. $\Lambda_N$) the Dirichlet (resp. Neumann) Lagrangian of the symplectic space $(\R^{2k-2n}, \Omega)$, for $n \le k\le 2n$.
\item $\iMor(\#)$ denotes the Morse index (i.e. the maximal dimension of the negative spectral space) of the operator $\#$ 
    \item \( \spfl(\cdot) \): the spectral flow of a path of closed self-adjoint  Fredholm operators;
    \item \( \iCLM(\cdot) \): the  Maslov index as defined by Cappell--Lee--Miller
    .
    \item \( \itriple(\cdot,\cdot,\cdot) \): the triple index, quantifying the difference of Morse indices between different boundary conditions.
\end{itemize}
}}

 \subsubsection*{List of Assumptions}

\begin{itemize}
	\item[-]{\sc (H0)} Given $A \in \Cl^s(H)$,  we assume that $\ker A^* \cap \dom(A)=\{0\}$ {\sc (Unique Continuation Property)}
	\item[-]{\sc (H1)} The endpoint $b$ is finite. The coefficients $P,Q,R,C_s$ extend continuously to $b$, and $\int_c^b\|P(t)^{-1}\|\,dt<\infty$ for some interior $c$. Thus $b$ is regular in the integrable-coefficient sense; $P(b)$ need not be invertible.
 \item[-]{\sc (H2)} We assume that $[0,1]\ni s\mapsto L_s \in \CFs(H)$ is a gap-continuous path of symmetric Fredholm operators {\sc (gap-continuity  and Fredholmenss)}
\item[-]{ (H3)} We assume that the quadratic form  $k$ in $ H =L^2((a,b), \R^n)$ associated to the SL-operator $l$   is  {\sc bounded from below}
	\item[-]{(H4)} The minimal operator $L_b$ is a closed symmetric and Fredholm operator in $ H $,  so $ L _b \in \CFs( H )$.
	\end{itemize}
} }

\subsubsection*{Acknowledgements}

The second author is grateful to New York University Abu Dhabi for the opportunity to conduct his research in such an inspiring and supportive environment. He also extends his sincere thanks to the Institute’s administrative and technical staff for their invaluable assistance and the excellent working conditions they provide. 

\section{An abstract  spectral flow formula }\label{sec:Abstract-sf}

In this section, we introduce an abstract decomposition formula for an operator pencil of densely defined, closed self-adjoint Fredholm operators and we finally prove a spectral flow formula. In a different context of Cauchy data space, our main result of this section can be recast in the Gelfand-Robbin quotients. Our primary references are \cite{Fur04,BZ18,Kat80} and the sources cited therein. 

\subsection{Symplectic Hilbert spaces and a decomposition formula}

Let $(H, \langle \cdot, \cdot \rangle)$ be a real and separable Hilbert space with an inner product $\langle \cdot, \cdot \rangle$ and associated norm $\|\cdot\|$.
We denote by $\Grass(H)$ the set of all closed linear subspaces of $H$ and we set 
\[
\dist (u,V)\=\inf_{v \in V}\|u-v\|.
\]
 Given  $U,V \in \Grass(H)$, we denote by $S_U$ the unit sphere of $U$ and we set 
\begin{equation}
	\delta(U,V):=\begin{cases}
	\sup_{u \in S_U} \dist(u, V) & \textrm{ if } U \neq (0)\\[3pt]
	0 &  \textrm{ if } U = (0)
\end{cases}\quad \textrm{ and } \quad 
\widehat \delta(U,V):= \max\{\delta(U,V), \delta(V,U)\}.
\end{equation}
$\widehat \delta(U,V)$ is called the {\sc  gap} between $U$ and $V$ and defines a metric on  $\Grass(H )$ whose associated metric topology is called {\sc gap topology}. Given $U \in \Grass(H )$, there exists a unique orthogonal projection $P_U$ onto $U$ which is a bounded operator on $H$. Given $U, V \in \Grass(H)$, we consider the norm of the difference of the corresponding orthogonal projectors:
\begin{equation}\label{eq:DG}
d_G(U,V):= \norm{P_U- P_V}.
\end{equation}
It can be shown  that metric topology induced by $d_G$ is equivalent to the gap topology. 
\begin{note}
We denote by $\Cl(H)$ (resp. $\Cl^s(H)$)  the set of all {\sc  closed} (resp. {\sc  closed and symmetric}) and densely defined   operators\footnote{
Since  the adjoint of a densely defined linear operator is closed then  self-adjoint  operators on $H$ are contained in $\Cl(H)$.} and we denote by $\Grn{\#}$  the graph of the operator $\#$.
\end{note}
We observe that the gap metric induces in a natural way a metric on $\Cl(H)$ (resp. $\Cl^s(H)$) given by 
\[
d(T, S)\=\widehat \delta(\Grn{T} , \Grn{S}) \qquad S,T \in \Cl(H).
\]

Let $A \in \Cl^s(H)$,   and let $A^*\in \Cl(H)$  be its adjoint and we denote by  $\dom(A)$ and $\dom(A^*)$  the domain of $A$ and $A^*$, respectively. We  define the {\sc   graph inner product on $\dom(A^*)$}  as follows
\[
\langle x,y \rangle^G:=\langle x, y \rangle + \langle A^* x, A^* y\rangle \qquad x, y \in \dom(A^*)
\]
and we observe that  $\dom(A^*)$ equipped with the graph norm  becomes a Hilbert space  and $\dom(A)$ is a closed subspace. (Cf.  \cite[Example 2.2]{Fur04} for further details).  
Let us introduce on $D(A^*)$ the following skew-symmetric bilinear form $\omega$  given by 
\begin{equation}\label{eq:symplectic-form}
\omega\big(x, y\big)=  \langle x, A^* y\rangle -\langle A^*\, x, y\rangle
\end{equation}
and we observe that $\omega$ is degenerate being $\ker \omega= \dom(A)$.
Let $A \in \CFs(H)$ be a densely defined closed, symmetric and Fredholm operator on $H$. Under this assumption the following {\bf splitting  formula} holds 
\[
\dom(A^*)=\dom(A)\oplus U
\]
where $U$ is a finite dimensional (closed) subspace of $\dom(A^*)$. It is worth observing that the restriction  $\rho\=\omega|_U$ is non-degenerate. So the pair $(U, \rho)$ defines a symplectic space. We denote by $p:\dom(A^*) \to U$ the canonical projection on the second factor.

For example we can choose $U$ be  the orthogonal complement (wrt the graph inner product) $\dom(A)^{\perp_G}$ of $\dom(A)$ in $\dom(A^*)$ with respect to the graph-inner product; namely with 
\begin{multline}
\dom(A)^{\perp_G}=\Set{x \in \dom(A^*)|\langle x, y\rangle^G=0\ \textrm{ for all } y \in \dom(A)}\\[5pt]= \Set{x \in \dom(A^*)|A^*(x) \in \dom(A^*)   \textrm{ and }  A^*(A^*(x))=-x}.
\end{multline}
From this characterization we get that $-A^*$ restricted to $\dom(A)^{\perp_G}$ is an orthogonal transformation into itself and defines a complex  structure on $\dom(A)^{\perp_G}$ since for every $x,y \in \dom(A)^{\perp_G}$, we have
\[
\omega(x,y)=\langle x, A^*(y)\rangle- \langle A^*(x), y\rangle =\langle -A^*(A^*(x)), A^*(y)\rangle- \langle A^*(x), y\rangle=\langle -A^*(x), y\rangle^G.
\]
This equality shows that the symplectic space $U$ together with  $\rho$ defined above and together with the almost complex structure $-A^*$ after the identification  with the orthogonal complement $\dom(A)^{\perp_G}$ is a symplectic  space with a compatible symplectic form, inner product and the almost complex structure. 

\subsection{An abstract spectral flow formula}

Let   $A\in\Cl^s(H)$ and let $A^*$ be its adjoint,  with domains respectively given by $\dom(A), \dom(A^*)$, and we consider the  splitting  
\[
\dom(A^*)=\dom(A)\oplus U
\]
Let $p: \dom(A^*) \to U $ be  projection map  and we observe that the subspace $p (\ker A^*)$ 
is an {\em isotropic subspace} of the symplectic space $U$. We refer to $p (\ker A^*)$ as the {\sc Cauchy data space of the operator $A$.}
The next well-known result gives a characterization of the self-adjointness  of an operator $A$ on $\dom$ in terms of the symplectic properties of $p(\dom)$. 
\begin{lem}\label{lem:abstract_fundamental_solution}
Let $A\in \Cl^s(H)$ and let $\dom $ be denote a linear subspace such that $\dom(A)<\dom < \dom(A^*)$. Then the restriction of the adjoint operator  $A^*$ to the subspace $\dom$ is self-adjoint, if and only if, the subspace $p(\dom)$ is a Lagrangian subspace in $U$.	
\end{lem}
\begin{proof}
We have $p(D)=D\cap U$.
Let $A_D=A^*|_D: D \to H$. Since $A$ is symmetric then $\dom (A) < \dom (A^*)$ and so $\dom (A_D^*) < \dom (A^*)$. Then we have 
  \[
\dom(A_D^*)=\set{x\in \dom(A^*)|\langle x,A_Dy\rangle-\langle A^*x,y\rangle=0,\ \forall y\in D}=D^\omega
\] 
since $A_D^*x=A^*x$ for every   $x \in \dom(A_D^*)$  and $A_Dy=Ay$ for every $y\in \dom(A)$. So, we get   $\dom(A_D^*)=D^\omega$ .
To conclude that $A_D^*=A_D$, by the above arguments, it's equivalent to prove that $D^\omega=D$.  Since $D(A) \subset D$ and $D(A)<D^\omega$, we have to show that 
\[
D\cap U= D^\omega \cap U
\]
and being $(D\cap U)^\rho=D^\omega \cap U$, we just need to show that  $(D\cap U)^\rho= D\cap U$. The conclusion follows by observing that  $D\cap U= p(D)$.

\end{proof}
\begin{lem}\label{thm:beta-finit-dimensional}
Let $A\in \CFs(H)$. Then 
\begin{itemize}
\item[(a)] $U$ is a finite dimensional symplectic vector space
	\item[(b)] $A$ has at least one self-adjoint  Fredholm extension; that is, there exists a subspace $ \dom $ (closed in the graph norm topology) such that $A_{\dom}:= A^*\big \vert_{\dom}$ is self-adjoint  and Fredholm.
	\end{itemize}
\end{lem}
\begin{proof}
We  prove (a). Since $A$ is symmetric and Fredholm, then  also $A^*$ is Fredholm (even if  not necessarily symmetric\footnote{  For example, on $L^2(0,1;\mathbb C)$ the closed  symmetric   operator $i\partial_x$  with domain   $H^1_0(0,1)$ has adjoint domain $H^1(0,1)$ .}). Let us consider the injection $j: \dom(A) \to \dom(A^*)$ and denoting by $\Find$ the Fredholm index, we get  the following formula (cf. \cite[Theorem 1.3.2]{Hor89})
  \[
\Find A= \Find A^*+\Find j.
\] 
Being $j$ an injection, it follows that $\Find j =-\dim \coker j= -\dim U$. So, we get  that 
\[
\dim U= \Find A^*-\Find A
\]
and since both integers on the (RHS) of the previous equality are finite, also their difference is so. This  concludes  the proof of (a).

For proving $(b)$ we start  by  observing that since by (a), the symplectic space $(U, \rho)$ is finite dimensional, then there exists a Lagrangian subspace $L<U$. We let $\dom  = p^{-1}(L)$ and we observe that $p(\dom )= L$  since   $p$  is surjective. The conclusion directly  follows by Lemma~\ref{lem:abstract_fundamental_solution}.  
  For completeness, the injection $j$ is Fredholm by the two-out-of-three property applied to $A=A^*j$, with both domains carrying their graph norms. Every intermediate closed domain is a finite-dimensional enlargement of $\dom(A)$; extending a Fredholm operator by finitely many domain directions preserves closed range and finite-dimensional kernel and cokernel. This proves Fredholmness of the self-adjoint extension in (b).
 \end{proof}
\begin{rem}
The minimal operators   considered here are closed and symmetric; Fredholmness is imposed when required and is not automatic for singular problems. Under that hypothesis the boundary quotient is  finite-dimensional.   For elliptic PDEs  the minimal operator   can fail to be Fredholm because its cokernel is  infinite-dimensional  , even when unique continuation makes its kernel trivial .
\end{rem}

We introduce the following {\sc Unique Continuation Property}.
\begin{itemize}
	\item {\sc (H0)} Let $A \in \Cl^s(H)$  and we assume that $\ker A^* \cap \dom(A)=(0)$. 
\end{itemize}
 \begin{rem}
In the case of regular linear ordinary differential operators, the unique‐continuation property holds automatically. However, for general elliptic partial differential operators this property can fail, a phenomenon that is intimately connected to Carleman estimates and the Fefferman–Phong inequality.
 \end{rem}
 The first basic consequence of (H0) is the following result. 
\begin{lem}\label{thm:nuovo}
	Let $A \in \CFs(H)$ and we assume that (H0) holds. Then 
	\[
	p\big(\ker A^*\big)
	\]
	is a Lagrangian subspace of $U$. Moreover this property doesn't depend on the chosen decomposition.
\end{lem}
\begin{proof}
    By (H0) we infer that  $\dim p(\ker A^*)=\dim \ker A^*$. We start showing that this property is independent on the chosen decomposition.
    Let $\dom(A^*)=\dom(A)\oplus U'$ be another decomposition.
    Since $\dom (A)=\ker \omega$, then the map $\Phi: (U',\rho')\mapsto (U,\rho)$ is a symplectic isomorphism. 
    So, by this argument we get that the property for $p(\ker A^*)$ to be Lagrangian is independent on the chosen decomposition.
    Without loss of generality, we can assume that $\ker A^*< U$ .
    We have
    \[
    (\ker A^*)^\rho=(\ker A^*)^\omega\cap U.
    \]
    We observe  that 
    \begin{multline}
    (\ker A^*)^\omega=\set{v \in \dom(A^*)|\langle u, A^*v \rangle=0,\forall u\in \ker A^*}=(A^*)^{-1}\left(\image (A)\right)\\=(A^*)^{-1}A^*(\dom (A))=\dom (A)+\ker A^*.
    \end{multline}

    It follows that $(\ker A^*)^\omega\cap U=(\dom (A)+\ker A^*)\cap U=\ker A^*+(\dom(A)\cap U)=\ker A^*$.
\end{proof}
At this point, we introduce two technical results that will be used frequently throughout the manuscript.
\begin{lem}\label{thm:lemma-same-domain-perturbation}
	Let $A\in \Cl^s(H)$ be a closed and symmetric operator on $H$ and let $B \in \Bsa(H)$ be a bounded self-adjoint  operator on $H$.  Then
    \begin{itemize}
    \item $A+B\in \Cl^s(\dom(A),H)$
    \item  $\dom (A^*)=\dom ((A+B)^*)$.
    \end{itemize}
\end{lem}
\begin{proof}
	By \cite[Chapter 4, Theorem 1.1]{Kat80} we get that $A+B$ is a closed operator.   Symmetry follows from $A\subset A^*$ and $B=B^*$. To identify the adjoint,  we start observing that  the operator-matrix  $\begin{pmatrix}0 &  C\\ D& 0\end{pmatrix}$   in  $H\times H$ , with domain $\dom(D)\times\dom(C)$ and densely defined closed $C$,  is self-adjoint if and only if $D=C^*$.
	By \cite[Chapter 5, Theorem 4.3]{Kat80}, we get that  $\begin{pmatrix}0 & A+B\\ A^*+B&0\end{pmatrix}$ is self-adjoint  since it is a bounded self-adjoint  perturbation of $\begin{pmatrix}0 & A\\ A^*&0\end{pmatrix}$.
	  The two matrices involving $A$ above have domain $\dom(A^*)\times\dom(A)$. Comparing their off-diagonal adjoint entries,  we have 
	\[
	(A+B)^*=A^*+B
	\]
	 and by this equality the result then follows.
	\end{proof}
Lemma~\ref{lem:conti_bound_pertub} which is frequently used along the paper asserts that the perturbation of a closed operator on \( H \) by a continuous family of bounded linear operators, parametrized by an open subset of \( \mathbb{R}^m \), defines a gap-continuous family in \( H \). In particular, a one-parameter family resulting from such a perturbation is gap-continuous.

	\begin{lem} \label{lem:conti_bound_pertub}
	Let $ L\in \Cl(H)$ be a closed operator. Let $\Omega\subset \R^m$ be an open subset  and we assume that the mapping 
    \[
    \Omega\ni v\mapsto K_v\in \Bou(H)
    \]
    is continuous. Then the map   $v\mapsto   L+K_v$ is gap-continuous  in $H$.
	\end{lem}
	\begin{proof}
Let $u\in\Omega$ and let $0<\varepsilon<1$. Since the mapping
$v\mapsto K_v$ is continuous in the operator norm, up to choosing
a smaller neighbourhood $U$ of $u$, we can assume that
$\|K_u-K_v\|<\varepsilon$ for every $v\in U$.

Each operator $L+K_v$, with domain $\dom(L)$, is closed because
$L$ is closed and $K_v$ is bounded. In the product space
$H\times H$, let us consider the norm
\[
\|(p,q)\|_{H\times H}=\|p\|_H+\|q\|_H.
\]
This norm is equivalent to the Hilbert product norm and therefore
induces the same gap topology.

Write $G_v:=\Graph(L+K_v)$. For closed subspaces $E,F$ of
$H\times H$, denote their directed gap by
\[
\delta(E,F)
:=
\sup_{\substack{x\in E\\ \|x\|=1}}
\operatorname{dist}(x,F).
\]
If $E=\{0\}$, set $\delta(E,F)=0$. The symmetric gap is
$\widehat\delta(E,F):=\max\{\delta(E,F),\delta(F,E)\}$.

Let
\[
x=(w,Lw+K_uw)\in G_u,
\qquad
\|w\|+\|Lw+K_uw\|=1,
\]
and consider the corresponding vector
\[
y=(w,Lw+K_vw)\in G_v.
\]
Then
\begin{align*}
\operatorname{dist}(x,G_v)
&\le \|x-y\|\\
&=\|(K_u-K_v)w\|\\
&\le \|K_u-K_v\|\,\|w\|
<\varepsilon.
\end{align*}
Taking the supremum over all unit vectors $x\in G_u$, we obtain
\[
\delta(G_u,G_v)\le\varepsilon.
\]

Moreover, we observe that
\[
\|y\|
\ge \|x\|-\|x-y\|
\ge 1-\varepsilon>0.
\]
Since $G_u$ is a linear subspace, $x/\|y\|\in G_u$. Hence
\[
\operatorname{dist}\left(\frac{y}{\|y\|},G_u\right)
\le
\left\|\frac{y}{\|y\|}-\frac{x}{\|y\|}\right\|
=
\frac{\|y-x\|}{\|y\|}
\le \frac{\varepsilon}{1-\varepsilon}.
\]

As $x$ ranges over the unit vectors of $G_u$, the normalized
vectors $y/\|y\|$ range over all unit vectors of $G_v$.
Indeed, the linear correspondence
\[
(w,Lw+K_uw)\longmapsto(w,Lw+K_vw)
\]
is a bijection from $G_u$ onto $G_v$, and normalization preserves
each nonzero direction. Taking the supremum therefore gives
\[
\delta(G_v,G_u)\le\frac{\varepsilon}{1-\varepsilon}.
\]

Combining the two estimates, we conclude that
\[
\widehat\delta(G_u,G_v)
\le \frac{\varepsilon}{1-\varepsilon}.
\]
Since $\varepsilon>0$ can be chosen arbitrarily small, the map
$v\mapsto L+K_v$ is gap-continuous at $u$. As $u\in\Omega$ was
arbitrary, the proof is complete.
\end{proof}
Let $[0,1]\ni s \longmapsto A_s:= A+s\Id \in \CFs(H)$ be a family of closed symmetric and Fredholm operators.  Then, by Lemma~\ref{thm:lemma-same-domain-perturbation}, we  get   that 
\begin{equation}\label{eq:dominio-fisso}
\dom(A_{s}) =\dom(A)  \quad \textrm{ and } \quad \dom(A^*_{s}) =\dom(A^*) \qquad s \in[0,1].
\end{equation} 
As a direct consequence of  Equation~\eqref{eq:dominio-fisso}, is that  both the factor space $U_s$ and the associated form $\rho_s$ induced by $A_s$  are  independent on $s$ and so, they  coincide with the factor space of $(U, \rho)$. 

 Moreover, by invoking Lemma~\ref{lem:conti_bound_pertub}, we get that  the path $s\mapsto A_s$ is gap-continuous.
 \begin{rem}
  This choice greatly simplifies the construction of the entire index theory by permitting us to work in a fixed symplectic space.  The situation becomes markedly more complicated if $\dom(A_s^*)$—and hence the quotient space $U$—depends on $s$, or if $\beta$ remains $s$‑independent while the symplectic form $\omega$ varies with $s$.    A bounded symmetric zero-order perturbation, however, cancels from the Green form and does not cause the latter scenario .
\end{rem}
\begin{lem}\label{lem:conti_ker}
We let $V_s:= \ker A^*_{s}$ and we assume that condition (H0) holds  for every $A_s$ .    Then the map  $s \mapsto p(V_s)$ is gap-continuous   in $\Lag(U,\rho)$. 
\end{lem}
\begin{proof}
 Since $\ker A_s=0$ and $A_s$ has closed range, $A_s^*$ is surjective. On the fixed maximal graph domain the maps $A_s^*=A^*+sI$ depend continuously in operator norm.  By Corollary~\ref{cor:continu_ker}, the map $s\mapsto V_s$ is gap continuous in $\Grass(\dom(A^*))$ with respect to the graph norm.
We have 
\[
p(V_s)=(V_s+\dom(A))\cap U.
\]
By assumption (H0),  we get that $V_s \cap  \dom(A)=(0)$. So, by Lemma~\ref{lem:continu_subspace}, $V_s+\dom(A)$ is gap continuous in $\Grass(\dom(A^*))$. Since $V_s+\dom(A)+U=\dom(A^*)$, by Lemma~\ref{lem:continu_subspace}, also $p(V_s)$ is gap-continuous in $\dom(A^*)$ and then it is gap-continuous in $\Lag(U,\rho)$.
 We also record the differentiability used below. A bounded right inverse $R_0$ for $A_{s_0}^*$ gives local analytic right inverses $R_s=R_0(A_s^*R_0)^{-1}$. The projections $I-R_sA_s^*$ have range $V_s$, so analytic local frames exist. By (H0), $p|_{V_s}$ is an isomorphism onto $p(V_s)$.
 \end{proof}
The next result shows that the Lagrangian path $s \mapsto V_s$ is a {\sc negative curve} meaning that each crossing instant is nondegenerate and gives a negative contribution  to the Maslov index.

\begin{lem}\label{thm:plus-curve}
The path $[0,1]\ni s \longmapsto p(V_s) \in \Lag(U,\rho)$ is negative.
\end{lem}
\begin{proof}
	We assume that $s_0 \in[0,1]$ is a crossing instant.   Choose a differentiable local section  $x(s)\in p(V_s)$ such that  $\displaystyle\lim_{s\to s_0}x(s)= x$ and we set 
	   $z(s)=(p|_{V_s})^{-1}(x(s))$ . By a direct calculation, it follows that 
	\begin{multline}
	\omega(x,x(s))= \langle z(s_0),A_{s_0}^*\,z(s)\rangle-\langle A_{s_0}^*\,z(s_0),z(s)\rangle = \langle z(s_0),A_{s_0}^*z(s)\rangle\\
    =\langle z(s_0),A_{s}^*z(s)\rangle +\langle z(s_0),(s_0-s)z(s)\rangle
    =(s_0-s)\langle z(s_0), z(s)\rangle
	\end{multline} 
	since $z(s)\in \ker A_s^*$ and $z(s_0)\in \ker A_{s_0}^*$.
    So, we get that 
	  \[
	\Gamma: p(V_{s_0})\ni x\longmapsto \left.\dfrac{d}{ds}\right|_{s=s_0}\omega(x,x(s))=-\langle z(s_0),z(s_0)\rangle =-\|z(s_0)\|^2
	\] 
	which is a  negative definite quadratic form. This concludes the proof. 	\end{proof}

Summing up all the previous arguments we get the following spectral flow formula for an operator pencil in $\CFs(H)$. 
\begin{prop}\label{thm:main1-abstract}
Let   $s\mapsto A_s=A+s\Id \in \CFs(H)$  and we assume that  (H0) holds  for every $A_s$, $s\in[0,1]$ .  Given $L \in \Lag(U,\rho)$, let us   consider the self-adjoint  extension $A_L$.  Then we get 
\begin{equation}\label{eq:equation-pencil}
 \spfl(A_L+sI, s \in[0,1])=-\iCLM(L,p(V_s), \rho, s \in[0,1]).
 \end{equation}

\end{prop}
\begin{rem}
Some remarks are in order. First, by Lemma~\ref{thm:lemma-same-domain-perturbation}, the symplectic space \( (U_s, \rho_s) = (U, \rho) \) is independent of \( s \). By Lemma~\ref{thm:beta-finit-dimensional}, since \( A \in \CFs(H) \), there exists a self-adjoint extension, denoted \( A_L \). By Lemma~\ref{lem:abstract_fundamental_solution}, this extension corresponds to a Lagrangian subspace \( L \). Moreover, by Lemma~\ref{lem:conti_bound_pertub}, the path \( s \mapsto A_L + s \Id \) is gap-continuous, and hence the spectral flow is well-defined. Finally, by the Unique Continuation Property, the path \( s \mapsto p(V_s) \) is also gap-continuous and so the $\iCLM$-index is well-defined, too. Thus, both sides of Equation~\eqref{eq:equation-pencil} are well-defined.
\end{rem}

\begin{proof}
  Each $A_L+sI$ is Fredholm, being a finite-dimensional domain extension of $A_s$. Its spectral-flow crossing form is $u\mapsto\|u\|^2$, so every crossing is regular and isolated (one-sided at the endpoints). The crossing set is closed by openness of invertibility; it is therefore finite  on the compact interval $I$ . To prove the theorem,  it is enough to show that the local contributions to the Maslov index and to the spectral flow are   opposite . For, we assume that $s_0 \in[0,1]$ is a crossing instant and we observe that 
\begin{equation}\label{eq:local-contribution}
\dim \ker(A_L+ s_0 \Id)= \dim \big(L\cap p(V_{s_0})\big).
\end{equation}
 Indeed, $u\mapsto p(u)$ is a bijection between these two spaces: injectivity is (H0), and surjectivity follows from $\dom(A_L)=p^{-1}(L)$.
 Moreover $s \mapsto A_L+s\Id$ is a positive curve,  $s\mapsto p(V_s)$ is  a negative curve as proved in Lemma~\ref{thm:plus-curve} and  the crossing instants for these two paths are in 1-1 correspondence.    At an interior crossing this  implies the following two facts: 
\begin{itemize}
	\item The local contribution to the spectral flow which is given by the signature of the crossing form at $s_0$ coincides with the (LHS) of Equation~\eqref{eq:local-contribution}
	\item The local contribution to the Maslov index which is given by the signature of the crossing form at $s_0$ coincides with the negative of the (RHS) of Equation~\eqref{eq:local-contribution}.
\end{itemize} 
 Under the endpoint conventions in Equation~\eqref{eq:iclm-crossings-2} and Proposition~\ref{thm:spectral-flow-formula}, an initial crossing contributes zero to both indices. A final crossing contributes its kernel dimension to spectral flow and its negative to the Maslov index.  The conclusion follows by summing   over all  crossings. 
\end{proof}
\begin{rem}
We briefly indicate a geometric interpretation of
Proposition~\ref{thm:main1-abstract} in the spirit of the
infinite-dimensional Maslov index, without developing that
theory in detail.

Consider the pair of Lagrangian subspaces
\[
H\times\{0\},\qquad \Graph(A_L+sI),
\]
and equip $H\oplus H$ with the standard symplectic form
\[
\widetilde\omega\big((u_1,v_1),(u_2,v_2)\big)
=
\langle u_1,v_2\rangle-\langle v_1,u_2\rangle.
\]
Set $A_{s,L}:=A_L+sI$. At a crossing instant $s_0$, take
$u\in\ker A_{s_0,L}$ and consider
\[
z(s):=(u,A_{s,L}u)\in\Graph(A_{s,L}).
\]
Then
\[
\left.\frac{d}{ds}\right|_{s=s_0}
\widetilde\omega\big(z(s_0),z(s)\big)
=
\left\langle u,
\left.\frac{d}{ds}\right|_{s=s_0}A_{s,L}u
\right\rangle
=
\|u\|^2.
\]
By Definition~\ref{def:crossing-form}, the crossing form
of $(\Graph(A_{s,L}),H\times\{0\})$ is therefore positive.

Fix $s_0$. By (H0) and the local differentiability of the
kernel spaces, we can choose a fixed finite-dimensional
subspace $V$ such that, for $s$ near $s_0$,
\[
\dom(A^*)=\dom(A)\oplus\bigl(\ker A_s^*\oplus V\bigr).
\]
Set
\[
U_s:=\ker A_s^*\oplus V.
\]
Since the radical of the Green form on $\dom(A^*)$ is
$\dom(A)$, its restriction to $U_s$ is nondegenerate.
Consequently,
\[
\Graph(A_s^*|_{U_s})
\]
is a symplectic subspace of $(H\oplus H,\widetilde\omega)$.
Notice that it is the graph over the complement $U_s$,
rather than the full maximal graph, that is symplectic.

Let
\[
\Lambda_s:=U_s\cap\dom(A_{s,L}).
\]
Inside $\Graph(A_s^*|_{U_s})$, the subspaces
\[
\ker A_s^*\times\{0\},
\qquad
\Graph(A_s^*|_{\Lambda_s})
\]
are Lagrangian. Indeed, the natural identifications
\[
\Graph(A_s^*|_{U_s})
\longrightarrow U_s
\xrightarrow{\;p\;} U
\]
are symplectic and send these subspaces to
$p(\ker A_s^*)$ and $L$, respectively. Here $p$ is the
fixed projection associated with
$\dom(A^*)=\dom(A)\oplus U$; in particular,
$p(\Lambda_s)=L$.

To check the induced crossing form, let
$u\in\ker A_{s_0,L}$, set $x=p(u)$, and choose a
differentiable section $v(s)\in\ker A_s^*$ with
$v(s_0)=u$. Since $A_s^*=A^*+sI$, we have
\[
\rho\bigl(x,p(v(s))\bigr)
=
\langle u,A_{s_0}^*v(s)\rangle
-
\langle A_{s_0}^*u,v(s)\rangle
=
(s_0-s)\langle u,v(s)\rangle.
\]
As $L$ is fixed, Definition~\ref{def:crossing-form} gives
\[
\Gamma\bigl(L,p(\ker A_s^*),s_0\bigr)[x]
=
-\left.\frac{d}{ds}\right|_{s=s_0}
\rho\bigl(x,p(v(s))\bigr)
=
\|u\|^2.
\]
Thus the positive graph crossing form corresponds to
the positive crossing form of the ordered boundary pair
$(L,p(\ker A_s^*))$. Reversing the order gives the negative
form of Lemma~\ref{thm:plus-curve}, consistently with the
sign convention in Proposition~\ref{thm:main1-abstract}.
\end{rem}

\section{An Abstract Spectral Flow Formula for SL-Operators}\label{sec:sf-SL-operators}

The primary objective of this section is to prove Theorem~\ref{thm:Sturm_Sf_formula}, which establishes a spectral flow formula for a class of one‑sided singular Sturm–Liouville operators.  We begin by deriving several preliminary results on gap-continuity  that are crucial for the proof of the main theorem.  For a comprehensive treatment of regular and singular Sturm–Liouville differential operators, see \cite{Zet05} and the references therein.

\subsection{A splitting formula for gap-continuous  paths}

Let $a,b \in \overline\R$. We  
define 
\begin{equation}\label{eq:sturm-liouville-operator}
		l_s:=-\dfrac{d}{dt}\left(P(t)\dfrac{d}{dt}+ Q(t)\right)+ \trasp{Q}(t)\dfrac{d}{dt}+ R(t) + C_s(t)\qquad \textrm{ for } \quad (t,s)\in (a,b) \times [0,1]
\end{equation}
where 
\begin{multline}\label{eq:assumptions-coeff}
		P^{-1} \in \mathscr C^1\big((a,b) , \Sym_n(\R)\big)  \qquad 
		Q \in\mathscr C^1\big((a,b) , \Mat_n(\R)\big) \qquad R \in \mathscr  C^0\big((a,b) ,\Sym_n(\R)\big) \\[3pt] \textrm{ and  }   s\mapsto C_s \in \mathscr  C^0_b\big((a,b) ,\Sym_n(\R)\big) \textrm{ is continuous}
\end{multline}
where $\mathscr C^0_b$ denotes the set of  bounded continuous functions and where  $C_0(t)=0$ for every $t \in (a,b)$. 

In the Hilbert space \( H = L^2((a,b), \mathbb{R}^n) \), we refer to the minimal operator $L_s$ defined by $l_s$ as the {\bf perturbed SL-operator} and to the minimal operator $L$ corresponding to $s=0$ as the  {\bf unperturbed SL-operator}. We denote by   $L_s^*$ and by $L^*$ the {\sc maximal} {\sc operator associated to $l_s$} and to {\sc $l$} respectively and we observe that both domains $\dom(L_s)$ and $\dom(L_s^*)$ are  independent on $s$.  

We now prove  that under the assumption (H1), the unique continuation property holds for the Sturm-Liouville operator $l$. 
 \subsubsection*{Regularity and quasi-derivative traces at the finite endpoint}
Condition (H1) is an integrable-coefficient regularity condition, not a nondegeneracy requirement on $P(b)$. Set $w=u^{[1]}=Pu'+Qu$ and $f=l_su$. Then, in the order $(u,w)$,
\[
 \frac{d}{dt}\binom{u}{w}
 =\begin{pmatrix}
 -P^{-1}Q&P^{-1}\\
 R+C_s-Q^TP^{-1}Q&Q^TP^{-1}
 \end{pmatrix}\binom{u}{w}+\binom{0}{-f}.
\]
The matrix belongs to $L^1((c,b))$ under (H1). For a maximal-domain element $f\in L^2$, hence $f\in L^1((c,b))$. Gronwall's inequality first bounds $(u,w)$ up to $b$, and the displayed equation then gives an integrable derivative. Thus $u,w$ extend absolutely continuously to $[c,b]$. Conversely, the Volterra integral equation with any prescribed $(u(b),w(b))$ has a unique solution on $[c,b]$; this follows by contraction on intervals on which the matrix has small $L^1$ norm, followed by concatenation. In particular zero endpoint data force a homogeneous solution to vanish. The same argument applies after any scalar spectral shift.

The endpoint trace is continuous in the maximal graph norm: the interior estimate bounds $(u(c),w(c))$ by that norm, and Gronwall bounds the endpoint data by these interior data and $\|f\|_{L^1(c,b)}$. To realize arbitrary endpoint data, solve the homogeneous system backwards from $b$ and multiply its $u$ component by a smooth cutoff equal to one near $b$ and zero to the left of an interior point. The commutator is supported in a compact interior interval, where the coefficients have the stated classical regularity, so the cutoff belongs to the maximal domain. This proves trace surjectivity and constructs the representatives $z_i$ below. They need not be smooth at $b$.

Throughout the general theory, $u^{[1]}(b)$ means the limit of the whole quasi-derivative, not the product $P(b)u'(b)+Q(b)u(b)$. No separate value $u'(b)$ is required. Green's bracket is evaluated using $(u^{[1]}(b),u(b))$; its continuity and approximation by the open-interval minimal core show that both traces vanish on the minimal domain. For example, $P(t)=\sqrt{b-t}\,\Id$, $Q=R=0$ satisfies (H1) with $P(b)=0$.

 \begin{lem} \label{lm:unique_extension}
We assume that condition (H1) holds and let $ L $ be  the minimal operator associated to the Sturm-Liouville operator $\ell$ defined at Equation~\eqref{eq:sturm-liouville-operator}. 
Then we get 	that the equation 
\[
	 L \,u=0
	\] 
	has no nontrivial solutions in $\dom( L )$. 
\end{lem}
\begin{proof}
Arguing by contradiction, we assume that there exists a nontrivial solution $0 \neq u$ in $H $ and by (H1) the operator $ L $ is only one-sided singular at $t=a$. We note that for each $u\in \dom ( L )$ it holds that   $u(b)= u^{[1]}(b)=0$ . For every  $d\in (a,b)$ the equation $ L \, u=0$ is a regular Sturm-Liouville equation on $[d,b]$ and by this we get that $u(t)=0$ for every  $t\in (d,b]$. This argument implies that   $u(t)=0$ for every  $t\in (a,b]$, concluding the proof. 
\end{proof}
Before proving the gap-continuity for a family of self-adjoint Fredholm extension operators, we start with the following well-known preliminary result. 
\begin{lem}\label{thm:collect}
Under  assumptions (H1) and  (H2), then we get: 
	\begin{enumerate}
		\item The  path of minimal operators $s \mapsto L_s$ is gap-continuous 

		\item  $s\mapsto p(V_s)$ is gap-continuous in $\Grass(U)$
		\item  $s\mapsto \Lambda_s \in \Lag(U, \rho)$ is a gap-continuous  Lagrangian path if and only if  
        \[ 
        s\mapsto L _{s,\Lambda_s}:= L_s^*|_{\dom( L ) \oplus \Lambda_s}
        \]
        is a gap-continuous path of operators in $\CFsa(H)$. 
		\end{enumerate}
\end{lem}
\begin{proof}
We start by observing that item~(1) is a direct consequence
of Lemma~\ref{lem:conti_bound_pertub}, since $s\mapsto L_s$
is a norm-continuous bounded perturbation of the closed
operator $L$. Item~(2) comes from Lemma~\ref{lem:conti_ker}
and Lemma~\ref{lm:unique_extension}.

We are now in a position to prove item~(3). Assume that
$s\mapsto\Lambda_s$ is a gap-continuous path of Lagrangian
subspaces in $U$. Then it also defines a gap-continuous
path in $\Grass(H)$.

Since $U$ is finite-dimensional and
$s\mapsto L_s^*|_U$ is continuous in the operator norm,
a continuous local basis of $\Lambda_s$ gives a continuous
local basis of
\[
\Graph(L_s^*|_{\Lambda_s}).
\]
Thus these graphs form a gap-continuous path in
$\Grass(H\times H)$. Moreover,
\[
\Graph(L_{s,\Lambda_s})
=
\Graph(L_s)\oplus\Graph(L_s^*|_{\Lambda_s}).
\]
The two summands have zero intersection, and their sum
is closed because the second summand is finite-dimensional.
By Lemma~\ref{lem:continu_subspace}, this sum is therefore
gap-continuous.

Conversely, suppose that
$s\mapsto\Graph(L_{s,\Lambda_s})$ is gap-continuous.
Set
\[
K_s:=\Graph(L_s^*|_U).
\]
As above, $s\mapsto K_s$ is gap-continuous. Since
$\dom(L_{s,\Lambda_s})=\dom(L)\oplus\Lambda_s$ and
$\dom(L^*)=\dom(L)\oplus U$, we have
\[
\Graph(L_{s,\Lambda_s})+K_s=\Graph(L_s^*).
\]
The right-hand side is closed and gap-continuous by
Lemma~\ref{lem:conti_bound_pertub}, applied to
$L_s^*=L^*+C_s$. Hence Lemma~\ref{lem:continu_subspace}
shows that
\[
\Graph(L_s^*|_{\Lambda_s})
=
\Graph(L_{s,\Lambda_s})\cap K_s
\]
is gap-continuous.

By invoking once again Lemma~\ref{lem:continu_subspace},
we obtain that
\[
\Lambda_s\times H
=
\Graph(L_s^*|_{\Lambda_s})+(\{0\}\times H)
\]
is gap-continuous: the two summands have zero intersection,
and their sum is closed because $\Lambda_s$ is
finite-dimensional. In conclusion, $s\mapsto\Lambda_s$
is gap-continuous in $\Grass(H)$, and hence also in
$\Grass(U)$.

Finally, each $L_{s,\Lambda_s}$ is self-adjoint by the
extension correspondence and Fredholm because it is a
finite-dimensional domain extension of the Fredholm
operator $L_s$.
\end{proof}

The main result of this section is the following spectral flow formula. 
\begin{thm}\label{thm:Sturm_Sf_formula} 
Under assumptions (H1) and (H2) and if  $s\mapsto \Lambda_s \in \Lag(U, \rho)$ is a gap-continuous  Lagrangian path, then we get:
  \[
  \spfl ( L _{s,\Lambda_s},s\in[0,1])= -\iCLM(\Lambda_s,  p(V_s) ,\rho, s\in [0,1]).
  \]
 \end{thm}

\begin{proof}
As a direct consequence of Lemma~\ref{thm:collect}, we obtain that the map \( s \mapsto L_{s,\Lambda_s} \) is gap-continuous. Moreover, by the previous discussion, the two Lagrangian maps \( s \mapsto \Lambda_s \) and \( s \mapsto p(V_s) \) are continuous. 

We also observe that \( L_s + r\Id = L + C_s + r\Id \), and by Lemma~\ref{lem:conti_bound_pertub}, it follows that the map \( (s,r) \mapsto L_s + r\Id \) is continuous with respect to the gap topology.

Using the localization properties of the Maslov index and the spectral flow, to conclude the proof, it suffices to verify that the formula holds in a sufficiently small neighborhood of any fixed instant \( s \in [0,1]\). Without loss of generality, we may localize near \( 0 \). 

 By openness of Fredholmness, first restrict $(s,r)$ to a rectangle where every minimal operator $L_s+rI$ is Fredholm; unique continuation holds there by the regular endpoint argument.  Let us choose \( \varepsilon > 0 \) such that \( L_{0,\Lambda_0} + \varepsilon \Id \) is invertible. Since invertibility is an open condition, there exists \( \delta > 0 \) such that \( L_{s,\Lambda_s} + \varepsilon \Id \) remains invertible for all \( s \in [0,\delta] \).

By the homotopy invariance properties of the spectral flow and the Maslov index, we obtain:
\begin{multline}
	\spfl( L_{s,\Lambda_s},\, s \in [0,\delta]) = \spfl( L_{0,\Lambda_0} + r\Id,\, r \in [0,\varepsilon]) - \spfl( L_{\delta,\Lambda_\delta} + r\Id,\, r \in [0,\varepsilon]) \quad \text{and} \\[5pt]
	\iCLM\big(\Lambda_s, p(V_s), \rho,\, s \in [0,\delta]\big) = \iCLM\big(\Lambda_0, p(V_{0,r}), \rho,\, r \in [0,\varepsilon]\big) \\
	- \iCLM\big(\Lambda_\delta, p(V_{\delta,r}), \rho,\, r \in [0,\varepsilon]\big).
\end{multline}

From this computation, and by invoking once again Proposition~\ref{thm:main1-abstract}, we immediately deduce that
\[
\spfl( L_{s,\Lambda_s},\, s \in [0,\delta]) = -\iCLM\big(\Lambda_s, p(V_s), \rho,\, s \in [0,\delta]\big).
\]
The conclusion then follows by summing up all  local contributions to both the Maslov index and the spectral flow.

\end{proof}

\section{An explicit decomposition and a trace map}\label{sec:coordinate-map}

The spectral flow formula for one‐sided singular Sturm–Liouville operators provided at Theorem~\ref{thm:Sturm_Sf_formula} is given, in terms of a splitting of the maximal domain and of a coordinate map \(O\).  In this section we implement \(U, O\) and $p$  concretely by means of a \emph{trace map}, built from the fundamental solutions of the Sturm–Liouville equation.  We then show that, in the regular (nonsingular) case, our abstract formula specializes exactly to the well‐known classical spectral flow formulas (see, e.g., \cite{BOPW21,HP17}). Finally, we give a brief, self‐contained derivation of all self‐adjoint boundary conditions for a Sturm–Liouville operator.

\begin{rem}
For the sake of the reader we observe that Shi and Sun \cite{SS10} have  been already characterized all self‐adjoint boundary conditions for Hamiltonian systems via the Glazman–Krein–Naimark (GKN) theory.
\end{rem}

\subsection{The trace map}\label{subsec:trace-map}

We are now ready to give a concrete   splitting decomposition 
	\[
	\dom( L ^*)=\dom( L )\oplus U
	\]
	where $U$ is a  $2k$-dimensional subspaces.	 

    Let $Z=\Span\Set{z_i|1\le i\le 2n}$  where $z_i$ is a   maximal-domain cutoff of a homogeneous solution near $b$,  such that $z_i=0$ on   $(a,c_0]$ for an interior point $c_0$  and such that $\Set{\big(z^{[1]}_i(b), z_i(b)\big)}_{i=1, \ldots, 2n}$ is a basis of $\R^{2n}$.

\begin{lem}\label{lem:max_dom_construct}
Let $L \in \CFs(H)$ and we assume condition (H1).  Then the following formula holds:
\[
\dom( L ^*)=\dom( L )+\ker  L^*+ Z.
\]
\end{lem}
\begin{proof}
Since $ \dim (\ker L^*+Z)<+\infty$, then  we have $(\ker L^*+Z)^{\omega\omega}=\ker \omega +(\ker L^*+Z)=\dom (L)+\ker L^*+Z$ . Since $[\dom(L)]^\omega=\dom(L^*)$,  we only need to show that $(\ker L^* +Z)^\omega =\dom(L)$. 

By arguing as in the proof of Lemma~\ref{thm:nuovo}, we get that 
\[
(\ker L^*)^\omega=\dom(L)+\ker L^* 
\]
and so 
\[
(\ker L^*)^\omega\cap Z^\omega=D(L)+(\ker L^*\cap Z^\omega).
\]
Let now $v\in Z^\omega$ and we observe that for every $i=1, \ldots, 2n$ it follows that $\omega(v,z_i)=[v,z_i](b)=0$. Since $Z(b)$ is a basis of $\R^{2n}$, then we get that    $v^{[1]}(b)=v(b)=0$ . So, $\ker L^*\cap Z^\omega=\set 0$ by the Cauchy-Lipschitz Theorem, concluding the proof. 
\end{proof}
The next result provides a characterization of the $\dom(L)$ in terms of the $\ker L^*$ and of the boundary values at the regular endpoint.
\begin{cor}\label{lem:vanish_condition}
Let $L \in \CFs(H)$,  let $ f\in \dom( L ^*)$ and we assume condition (H1). Then $f\in \dom( L )$ if and only if  
 $\omega(f,v)=0$ for all $v\in \ker  L ^*$ and   $ f^{[1]}(b)=f(b)=0$ .
\end{cor}
\begin{proof}
By Lemma~\ref{lem:max_dom_construct}, $\dom (L) =(\ker  L ^*)^\omega \cap Z^\omega$. Moreover, we observe that for every $i=1, \ldots, 2n$ we have  $\omega(f,z_i)=[f,z_i](b)$ and   $Z^\omega=\Set{f|f^{[1]}(b)=f(b)=0}$ .
Then we get
  \begin{align}
Z^\omega\cap (\ker L^*)^\omega=\Set{f|f^{[1]}(b)=f(b)=0\textrm{ and } [f,v](b)-[f,v](a^+)=0,\forall v\in \ker  L ^* }\\
=\set{f| f^{[1]}(b)=f(b)=0, [f,v](a^+)=0, \forall v\in \ker  L ^*}
\end{align} 
 This concludes the proof. 
\end{proof}

\begin{lem}\label{thm:decomp-factor-space}
Let $L \in \CFs(H)$ and we assume condition (H1). Let $ Y:=\set{y_1,\cdots, y_{k}}$ be a basis of $\ker L^*$. Then we have that $n \le k \le 2n$. Moreover, there exists a $2k$-dimensional subspace $W$ such that 
\[
 	\dom( L ^*) =\dom( L )\oplus W
 	\]
    where 
\begin{itemize}
\item $W:=\Span(\overline Y)\oplus Z$ and where $\overline Y < Y$ such that $\dim \overline Y=2k-2n$ if $k\neq n$
\item $W:= Z$ for $k = n$.
\end{itemize} 	
\end{lem} 
\begin{rem}
 We observe that the case \( k = 2n \) corresponds to either limit circle endpoints or  regular  conditions, whereas the case \( k = n \) corresponds to the limit point. (Cfr.  \cite{Zet05} and references therein for the endpoints classification of the Sturm-Liouville BVP). 
\end{rem}
\begin{proof}
	By Corollary~\ref{lem:vanish_condition},   $\dom(L)\cap Z=\set 0$ .
    Then by Lemma~\ref{lem:max_dom_construct}, we can choose $\overline Y<  Y$ such that 
    \[
    \dom(L^*)=\dom(L)\oplus(\Span (\overline Y)\oplus Z).
    \]
By Lemma~\ref{thm:nuovo}, $\dim(\Span(\overline Y)\oplus Z)=2\dim(\ker L^*)=2k$. It follows that
  \[
\#\overline Y= \dim (\Span(\overline Y))=\dim(\Span(\overline Y)\oplus Z)-\dim Z=2k-2n.
\] 
This concludes  the proof.

\end{proof}
\subsubsection*{Trace map $\Tr$ and the subspace $W$}\label{subsubsec:trace}
Summing up the above discussion we finally get the decomposition
  \begin{multline}
 	\dom( L ^*) =\dom( L )\oplus W \quad \textrm{ where } \quad W=\Span\Set{y_1, \ldots, y_{2k-2n}, z_1, \ldots, z_{2n}}\\ \textrm{ where } y_j \in \ker L^* \textrm{ and }  z_i=0 \quad \textrm{ on } \quad (a,c_0] \textrm{ and such that } \\ Z(b):=\Set{(z_i^{[1]}(b),z_i(b))|1\le i\le 2n}  \textrm{ is a basis of } \R^{2n}.
 	\end{multline} 
    \begin{itemize}
\item If $k=n$, we consider the symplectic space $(W,\omega|_W)$ and we define  the {\bf trace map}  $\Tr: W\to \R^{2n}$ as:
\begin{itemize}
\item[]   $
		f\longmapsto \Tr(f):=(\omega(f,z_1),\cdots,\omega(f,z_{2n}))
		=([f,z_1](b),\ldots,[f,z_{2n}](b))
		\in  \R^{2n}.
$ 
 \end{itemize}
\item If $k \neq n$, we observe that since $\omega|_Z$ is nondegenerate, then  $Z^{\omega|_W}=Z^\omega\cap W$ is a symplectic subspace of $(W,\omega|_W)$. Then $Z^{\omega|_W}$ has a basis $\set{\widetilde y_1,\cdots,\widetilde y_{2k-2n}}$ such that   $\widetilde y_i-y_i\in Z$ and $ \widetilde y_i^{[1]}(b)=\widetilde y_i(b)=0$  for every  $1\le i\le 2k-2n$.	Then $W$ has the following symplectic  decomposition:
\begin{equation}\label{eq:decomp-W}
W=Z^{\omega|_W}\oplus Z.
\end{equation}
\begin{note}
We set  
\[
[f,g](a^+):=\lim_{t \to a^+}[f,g](t).
\]
\end{note}
In this case, we define the (linear) {\bf trace map}   $\Tr: W\to \R^{2k-2n}\oplus \R^{2n}$  as: 
\begin{multline}
		f\longmapsto \Tr(f):=(\omega(f,\widetilde y_1),\cdots,\omega(f,\widetilde y_{2k-2n}),\omega(f,z_1),\cdots,\omega(f,z_{2n}))\\
		=(-[f,y_1](a^+),\cdots,-[f,y_{2k-2n}](a^+),[f,z_1](b),\ldots, [f,z_{2n}](b))
		\in \R^{2k-2n}\oplus \R^{2n}
\end{multline}
\end{itemize}
Since $\omega|_W$ is nondegenerate,  it  follows that 
\begin{itemize}
\item[] $\Tr:Z^{\omega|_W}\oplus Z \to \R^{2k-2n}\oplus \R^{2n} $ is a bijection  and moreover the decomposition is $\Tr$-invariant. 
\end{itemize}
\begin{rem}\label{rmk:Selection}
In the construction above we defined the trace map, but we did not specify how to select the
\(2k-2n\) functions \(\{y_1,\dots,y_{2k-2n}\}\subset\ker L^*\) so that
\[
W=\Span\{y_1,\dots,y_{2k-2n},\,z_1,\dots,z_{2n}\}.
\]
We choose them to make \(\omega|_W\) nondegenerate. It suffices to ensure that the restriction
\[
\omega\big|_{\Span\{\tilde y_1,\dots,\tilde y_{2k-2n}\}}
\]
is nondegenerate, where   $\widetilde y_i$ are obtained from the selected $y_i$ by subtracting their $Z$ components .
To this end, start with any basis \(\{y_1,\dots,y_k\}\) of \(\ker L^*\) and consider the
\(k\times k\) Gram matrix \(G=(G_{ij})\) of the boundary pairing
\[
G_{ij}\;=\;[y_i,y_j](a),\qquad 1\le i,j\le k.
\]
Reorder the basis so that the principal \((2k-2n)\times(2k-2n)\) block
\(\big([y_i,y_j](a)\big)_{1\le i,j\le 2k-2n}\) is invertible (this is always possible as soon as
  $\rank G=2k-2n$ ). Then take
\[
\{y_1,\dots,y_{2k-2n}\}
\]
to be the first \(2k-2n\) vectors in this reordered basis. By construction,   $\omega$  is nondegenerate on   $\Span\{\widetilde y_1,\dots,\widetilde y_{2k-2n}\}$ , and hence \(\omega|_W\) is
nondegenerate as claimed.
 To justify the rank, project $\ker L^*$ onto the symplectic complement of $Z$ in $W$. Lemma~\ref{lem:max_dom_construct} makes this projection surjective. Its pairing is the negative singular bracket, so its pullback to $\ker L^*$ has rank $2k-2n$. A skew-symmetric matrix of rank $2r$ has an invertible principal $2r$ minor (successive nonzero pivot pairs give this). The homogeneous space itself is isotropic for the full Green form; it is the adjusted space that is nondegenerate.
 \end{rem}

\begin{rem}
We observe that the space
\[
\R^{2k-2n}\oplus \R^{2n}
\]
is the {\sc space of boundary data} and actually  $\R^{2k-2n}$ (resp. $\R^{2n}$) contains the boundary data at $a$ (resp. $b$). 
\end{rem}
Summing up the previous discussion, it is possible to characterize $\dom(L)$ in terms of $\dom(L^*)$ and the trace map $\Tr$.
 We extend $\Tr$ to $\dom(L^*)$ by $\Tr(f):=\Tr(pf)$; the bracket formulas already give this extension.

 \begin{cor}\label{thm:cor-dec} We assume condition (H1) holds. Let
	\begin{enumerate}
    \item $f\in \dom( L ^*)$. Then  $f \in \dom( L )$ if and only if $\Tr(f)=0$
	\item  Let $[f,y_1](a^+),\cdots,[f,y_{2k-2n}](a^+)=0$. Then  $[f,g](a^+)=0$ for each $g\in \dom( L ^*)$.
    \end{enumerate}
\end{cor}
\begin{proof}
	We start by proving item 1.  By the very definition of $\Tr$, we get that  $\Tr(\dom( L))=\set 0$ since $\dom(L) =\ker \omega$. Viceversa, we observe that    $\Tr:W \to \R^{2k-2n}\oplus \R^{2n}$  is a bijection and $\dom (L^*)=\dom(L)\oplus W$. So, we get  that $\Tr(f)=0$ implies $f\in \dom(L)$.

We prove item 2.  Since $[f,y_1](a^+),\cdots,[f,y_{2k-2n}](a^+)=0$, then we have $ f\in \Tr^{-1}(\set 0\times \R^{2n})=\dom(L)\oplus Z$.
By setting $f=u+v$ with $v\in Z$, we get that  $[f,g](a^+)=[v,g](a^+)=0$.

\end{proof}
By using the trace map defined above, the spectral flow formula provided in Theorem~\ref{thm:Sturm_Sf_formula} reduces to  the following. 
\begin{thm}\label{thm:Sturm_Sf_formula-SL} 
Under conditions (H1)-(H2)  and assuming that   $s\mapsto \Lambda_s\in \Lag(W,\omega|_W)$ is gap-continuous,  then  the following spectral flow formula holds: 
  \[
\spfl ( L _{s,\Lambda_s},s\in [0,1])= -\iCLM(\Tr (\Lambda_s), \Tr (V_s), \rho^{\Tr} , s\in  [0,1]),
\]
where $\rho^{\Tr}$ denotes the push-forward symplectic form.
 \end{thm}
 \begin{proof}
The proof of this result is a direct consequence of  Theorem~\ref{thm:Sturm_Sf_formula}, Lemma~\ref{lem:max_dom_construct}, Corollary~\ref{lem:vanish_condition}, Lemma~\ref{thm:decomp-factor-space}, and Corollary~\ref{thm:cor-dec}. 
 \end{proof}

\subsection{Two classical examples}\label{subsec:examples}

We now are ready to show that the spectral flow formula given at Theorem~\ref{thm:Sturm_Sf_formula-SL} is  a generalization of the well-known spectral flow formulas. Actually in the already known 
\begin{enumerate}
	\item Case of {\sc regular SL-operators} on bounded domains
	\item Case of {\sc SL-operators} on $\R^+$.
\end{enumerate} 
it reduces to the already known formulas. 

 The next Sobolev-domain example assumes $P,Q\in C^1([a,b])$, $P$ invertible on $[a,b]$, and continuous zero-order coefficients. These assumptions are specific to this example.

 \subsubsection*{Regular SL-operators on bounded intervals} 

For $s\in[0,1]$, we consider the Sturm-Liouville operator $l_s$ given at Equation~\eqref{eq:sturm-liouville-operator} and we let  
  \[
X=\Span\set{x_{1},\cdots,x_{2n}}\qquad \textrm{ and } \qquad Y=\Span\set{y_{1},\cdots,y_{2n}}
\] 
where 
\begin{itemize}
	\item  $x_i$ are smooth functions defined on $[a,b]$ and vanishing on the interval  $[(a+b)/2, b]$ .
	\item  $y_i$ are smooth functions defined on $[a,b]$ and vanishing on the interval  $[a,(a+b)/2]$.
    \end{itemize}
We set
\[
U=X \oplus Y
\] 
and   choose the functions so that the following families  $\mathcal E_a$ and $\mathcal E_b$   are bases of  $\R^{2n}$  : 
	 \begin{multline}
		\mathcal E_a=\Set{\big(P(a)x_{i}'(a)+ Q(a)x_{i}(a),x_{i}(a)\big)|i=1, \ldots, 2n}\quad \textrm{ and } \\[3pt]
		\mathcal E_b=\Set{\big(P(b)y_{i}'(b)+ Q(b)y_{i}(b),y_{i}(b)\big)|i=1, \ldots, 2n}.
\end{multline}
We denote by $O: U  \to \R^{2n}\oplus \R^{2n}$ be the  coordinate map defined by 
\begin{equation}\label{eq:explicit-coord-map}
O(u)=\big(P(a)u'(a)+ Q(a)u(a),u(a),P(b)u'(b)+ Q(b)u(b),u(b) \big)
\end{equation}
and we consider the decomposition $D(L^*)=D(L)\oplus U$ and let $p: D(L^*) \to U$ be the projection onto the second factor.  

Denoting by $M_s$ the monodromy matrix of the Hamiltonian system induced by the Sturm-Liouville operator $l_s$, then we get that 
	  \[
O(p(V_s))=\Graph(M_s)\qquad \textrm{ for every } s \in[0,1] \qquad \textrm{ and }\qquad V_s=\ker L_s^*.
\] 
This equality is  a direct consequence of the definition of the monodromy matrix $M_s$ being   the projection of the solution space onto the space of boundary conditions.
Let $s\mapsto\Lambda_s \in \Lag(U, \rho)$ be a continuous Lagrangian path. Then we get that 
  \begin{multline*}
\dom(L_{s,\Lambda_s})=\bigl\{u\in W^{2,2}((a,b),\R^n):\\
(P(a)u'(a)+Q(a)u(a),u(a),P(b)u'(b)+Q(b)u(b),u(b))\in O(\Lambda_s)\bigr\}.
\end{multline*} 
and the following spectral flow formula holds 
\begin{equation}\label{eq:sf-quasi-nota}
	\spfl( L _{s,\Lambda_s}, s\in[0,1])=-\iCLM(O(\Lambda_s),\Graph(M_s), -\Omega \oplus \Omega, s\in[0,1])
\end{equation}
where $\Omega$ denotes the standard symplectic form of $\R^{2n}$. 
\begin{ex} \label{ex:standard_regular_spfl}
The spectral flow formula provided at Equation~\eqref{eq:sf-quasi-nota} coincides with the well-known spectral flow formula for the case of regular Sturm-Liouville operator. Actually, the standard framework is the following. Let $s\mapsto Y_s$ be a path  in $\Lag(\R^{2n}\oplus\R^{2n}, -\Omega\oplus\Omega)$ and not in the functional space $\dom(L^*)$. The path $s \mapsto Y_s$ induces a path  $s\mapsto L _{s,Y_s}$ of self-adjoint  operators with domain
\begin{multline}
\dom(L _{s,Y_s}):=\left\{u\in W^{2,2}((a,b) ,\R^n)|\right.\\ \left.(P(a)u'(a)+Q(a)u(a),u(a),P(b)u'(b)+Q(b)u(b),u(b))\in Y_s\right\}
\end{multline}
and  the spectral flow formula reduces to the usual one
\begin{equation}\label{eq:sf-classical}
	\spfl( L _{s,Y_s}, s\in[0,1])=-\iCLM(Y_s,\Graph(M_s),-\Omega \oplus \Omega, s\in[0,1]).
\end{equation}
\end{ex}
By comparing the Equation~\eqref{eq:sf-quasi-nota} and Equation~\eqref{eq:sf-classical}, the only difference is the presence of the mapping $O$. This is a direct consequence of the fact that in former case the  paths are in the functional space $\dom(L^*)$ whilst in the latter case are directly in $\R^{2n}\oplus \R^{2n}$. So, the two formulas coincide once definining $O(\Lambda_s)=Y_s$ for every $s \in[0,1]$.
\begin{rem}
The last remark is that in Theorem~\ref{thm:Sturm_Sf_formula} the symplectic form appearing in the computation of the Maslov index is $\rho^O$ and not just $-\Omega\oplus\Omega$. This is because the aforementioned theorem works also in  the singular case in which the explicit  coordinate map given at Equation~\eqref{eq:explicit-coord-map} is not well-defined anymore. 
\end{rem}

\subsubsection*{SL-operators on the half-line}
  For this example assume $P,Q\in C^1([0,\infty))$, with $P,P^{-1},P',Q,Q'$ bounded, $P(t)\ge p_0\Id$ for some $p_0>0$, and $R_s=R+C_s$ bounded symmetric with $s\mapsto C_s$ continuous in the uniform norm. Thus
\begin{equation}\label{eq:assumptions-coeff-2}
 \dom(L_s)=H^2_0(0,\infty;\mathbb R^n),\qquad
 \dom(L_s^*)=H^2(0,\infty;\mathbb R^n).
\end{equation} 
  Here $H^2_0$ imposes both $u(0)=0$ and $u'(0)=0$. To justify the maximal domain identity, the uniformly elliptic energy estimate with cutoffs gives $u'\in L^2$ from $u,L_su\in L^2$; expanding the differential expression and using the bounds on $P^{-1},P',Q,Q',R_s$ then gives $u''\in L^2$. Conversely these bounds imply $L_su\in L^2$ for $u\in H^2$. At infinity $u,u'$ and the quasi-derivative tend to zero, so the Green bracket vanishes there. Taking the Green annihilator of the maximal domain therefore imposes exactly the two traces at zero, proving the minimal domain identity. Approximation by cutoffs and smooth functions gives the same operator core.
Then  $s\mapsto L_s$ is a gap-continuous path of  closed symmetric minimal operators;  self-adjoint operators  arise only after a Lagrangian boundary condition is imposed at the regular endpoint . 

      Assume in addition that every minimal operator $L_s$ is Fredholm,  and  let us consider  $U=\Span\{x_1,\ldots ,x_{2n}\}$ where 
	\begin{itemize}
		\item  $x_i$ are smooth functions defined on $[0,+\infty)$ and vanishing  on $(1,+\infty)$.
		\item  $\mathcal E_0$ is the   family, chosen to form a basis of $\R^{2n}$,  defined by 
		\begin{equation}
			\mathcal E_0=\Set{\big(P(0)x_{i}'(0)+ Q(0)x_{i}(0),x_{i}(0)\big)|i=1, \ldots, 2n}.
		\end{equation}
	\end{itemize}
	Let us now consider the decomposition $\dom (L ^*)=\dom (L) \oplus U$ where $(U, \rho)$ is the symplectic space with respect to the symplectic form  defined by $\rho=\omega|_U$. We define the coordinate map 
    \[
    O: U\longrightarrow \R^{2n} \quad \textrm{ defined by  } \quad O(u)=\big(P(0)u'(0)+ Q(0)u(0),u(0)\big)
    \]
    and let $p: D(L^*) \to U$ be the projection onto the second factor. We let  $s \mapsto \Lambda_s \in \Lag(U,\rho)$.	Then the following spectral flow formula holds
	  \begin{equation}\label{eq:sf-quasi-nota-2}
	\spfl( L _{s,\Lambda_s}, s\in[0,1])=-\iCLM(O(\Lambda_s),O(p(V_s)), -\Omega , s\in[0,1])
	\end{equation} 
where $V_s:=\ker(L_s^*)$ is the linear subspace of all  $L^2$ solutions of the system on the half-line $[0,+\infty)$. So, in particular,   $O(p(V_s))$ is just the boundary value of these solutions   and is the $L^2$ Cauchy-data space  at $0$.  Its identification with an exponentially stable subspace requires an exponential-dichotomy assumption.

\begin{ex}\label{ex:standard_halfline_spfl}
The spectral flow formula provided at
Equation~\eqref{eq:sf-quasi-nota-2} coincides with the
well-known spectral flow formula for Sturm--Liouville
operators on the half-line, as proved in \cite{HP17}.
The framework in that paper starts with a path
$s\mapsto Z_s$ in $\Lag(\R^{2n},\Omega)$, rather than
in the functional space $\dom(L^*)$.
The path $s\mapsto Z_s$ induces a path
$s\mapsto L_{s,Z_s}$ of self-adjoint operators with domain
\begin{equation}
\dom(L_{s,Z_s})
:=
\left\{
u\in W^{2,2}([0,+\infty),\R^n)
\,\middle|\,
\big(P(0)u'(0)+Q(0)u(0),u(0)\big)\in Z_s
\right\}.
\end{equation}
The spectral flow formula then reduces to
\[
\spfl(L_{s,Z_s},s\in[0,1])
=
-\iCLM(Z_s,W_s^{st}(0),-\Omega,s\in[0,1]).
\]
Here $W_s^{st}(0)$ denotes the $L^2$ Cauchy-data space
at $0$, defined by
\[
W_s^{st}(0)
:=
\set{
\big(P(0)u'(0)+Q(0)u(0),u(0)\big)
\mid u\in\ker L_s^*
}.
\]
\end{ex}

\begin{note}\label{lem:normalized-trace}
For the moving-endpoint arguments below, choose fixed Darboux
coordinates $\tau_a$ for the singular boundary pairing and
use the quasi-derivative coordinates
\[
j_\sigma u=(u^{[1]}(\sigma),u(\sigma))
\]
at the regular endpoint. Thus
\[
\Tr_\sigma u:=(\tau_a u,j_\sigma u)
\]
identifies the boundary quotient with
$(\R^{2k-2n}\oplus\R^{2n},-\Omega\oplus\Omega)$, and
$\Lambda_D=\R^n\oplus\{0\}$ represents the Dirichlet condition.

The singular coordinates are chosen once and used on every
truncation, by cutoff extension near $a$. Bounded self-adjoint
zero-order perturbations leave the Green form unchanged.
The earlier bracket-coordinate trace retains its transported
form $\rho^{\Tr}$.
\end{note}

\begin{note}\label{def:global-morse}
Throughout the paper, the Morse index of a self-adjoint
operator $T$ is
\[
\iMor(T)=\dim\operatorname{Ran}
\mathbf1_{(-\infty,0)}(T),
\]
with the value $+\infty$ allowed.
By the spectral theorem, this equals the supremum of the
dimensions of negative subspaces of $\dom(T)$ for
$u\mapsto\langle Tu,u\rangle$.
If $T$ is semibounded, it also equals the index of its
associated closed quadratic form, including infinite values.
When the negative spectrum is discrete, this is the usual
eigenvalue count with multiplicity.
\end{note}

 \section{The Morse Index Theorem for Singular Sturm-Liouville}\label{sec:Morse-dirichlet}

Theorem~\ref{thm:Sturm_Sf_formula-SL} provides an abstract and general spectral flow formula for a family of Sturm--Liouville differential operators. It serves as a foundational step toward establishing a {\sc Morse Index Theorem} within this singular framework.

The aim of this section is to apply the spectral flow formula described above to the specific case of a single Sturm--Liouville operator \( l \) defined on the function space   \( \mathscr{C}_0^\infty((a,b), \mathbb{R}^n) \) , where the endpoint \( t = b \) is regular according to hypothesis (H1).

In this context, by restricting the operator \( l \) to the interval \( (a, \sigma) \) with \( \sigma \in (a, b] \), we obtain a one-parameter family of operators \( \sigma \mapsto l_{(a,\sigma)} \). This construction is necessary for the application of the spectral flow formula.

Let $\sigma \in (a,b)$. We define  $l_{(a,\sigma)}$ be the one-sided singular Sturm-Liouville differential operator obtained by restricting $l$ on $ \mathscr C_0^\infty((a,\sigma), \R^n)$. We denote by $L_{(a,\sigma)}$ be its minimal operator associated to $l_{(a,\sigma)}$ and we assume that $L_{(a,\sigma)}$ is singular only at the starting instant $t=a$ (and so the final instant $\sigma$ is a regular endpoint). We consider the Lagrangian subspaces 
\[
\Lambda_0 \in \Lag(\R^{2k-2n}) \quad \textrm{  and } \Lambda _D:=\R^{n} \times (0)\qquad \textrm{ and } \Lambda=\Lambda_0 \oplus \Lambda_D
\]  
where the $\Lambda_D$ is  standard  {\sc Dirichlet Lagrangian}. 

We introduce the following two conditions.
\begin{itemize}
\item[-]{ (H3)} The quadratic form on $\mathscr C^\infty_0((a,b), \R^n)$ associated to the SL-operator $l$   is  {\sc bounded from below} in $L^2$.
	\item[-]{ (H4)} The minimal operator $L_b \in \CFs(L^2((a,b), \R^n))$.
\end{itemize}

Given $\Lambda \in \Lag(\R^{2k-2n}\oplus \R^{2n},-\Omega\oplus \Omega)$, we set 
\[
L_{(a,\sigma),\Lambda}=L_{(a,\sigma)}^*|_{\dom(L_{(a,\sigma)})\oplus\Tr_\sigma^{-1}(\Lambda)}
\]
where $\Tr_\sigma: W_\sigma \to \R^{2k-2n} \oplus \R^{2n}$ is the trace map  defined on the subspace $W_\sigma$ associated to the operator $l_{(a,\sigma)}$ and constructed as in Subsection~\ref{subsec:trace-map}.

\begin{addlem}[Cutting at a regular interior point]
\label{lem:regular-cut-fredholm}
If the minimal operator on $(a,b)$ is Fredholm, then the
minimal operator on $(a,c)$ is Fredholm for every $c\in(a,b)$.
The same conclusion holds after a scalar shift whenever
the shifted full minimal operator is Fredholm.
\end{addlem}
\begin{proof}
The interior trace
$u\mapsto(u^{[1]}(c),u(c))$ is continuous and surjective
on the minimal graph domain. Its kernel has codimension
$2n$, and restriction to this kernel splits the operator
into the direct sum of the minimal operators on $(a,c)$
and $(c,b)$. Finite-codimensional restriction preserves
Fredholmness, and each summand of a Fredholm direct sum
is Fredholm. The argument is unchanged by a scalar shift.
\end{proof}

\paragraph{Forms near the regular endpoint.}
Under (H3), interior invertibility of $P$ implies $P>0$.
On a regular tail $(c,b)$, set
\[
E_P[u]:=\int_c^b\langle Pu',u'\rangle\,dt,
\qquad
A(x,y):=\int_x^y\|P^{-1}(t)\|\,dt.
\]
Weighted Cauchy--Schwarz gives
\[
|u(y)-u(x)|^2
\le
A(x,y)\int_x^y\langle Pu',u'\rangle\,dt.
\]
Since $A(c,b)<\infty$, the weighted energy space has
continuous endpoint traces and embeds compactly into
$C([c,b])$. Moreover,
\[
\|u\|_\infty^2
\le \varepsilon E_P[u]+C_\varepsilon\|u\|_{L^2}^2.
\]
These estimates show that the $Q$ terms and fixed Robin
boundary terms are infinitesimally form bounded relative
to $E_P$. Thus the corresponding forms are closed and
semibounded, with compact resolvent.

The closure of $C_0^\infty(c,b)$ in the weighted energy
norm consists precisely of the functions with zero
endpoint values. Indeed, truncate their derivatives
away from the endpoints, correct the integral by a
fixed compactly supported function, and then smooth.
The preceding estimate gives convergence of the
primitives. In particular, maximal-domain functions
with zero endpoint values belong to this Friedrichs
form domain: writing $w=u^{[1]}$, one has
\[
E_P[u]
=
\int_c^b
\langle P^{-1}(w-Qu),w-Qu\rangle\,dt<\infty,
\]
because $u,w$ are bounded at the regular endpoint.

\begin{addlem}[Semibounded extensions and local form domains]
\label{lem:local-form-domains}
Assume (H1), (H3), and (H4). Every self-adjoint extension
of the minimal operator is semibounded.
For
\[
T_c:=L_{(a,c),\Lambda_0\oplus\Lambda_D},
\qquad
E_c:=\dom(t_{T_c}),
\qquad a<c\le b,
\]
use fixed singular boundary coordinates.
Zero extension identifies
\[
E_c=\{u\in E_b:u=0\text{ a.e. on }(c,b)\},
\qquad
t_{T_c}=t_{T_b}|_{E_c}.
\]
Consequently, these are closed nested form subspaces,
\[
\bigcap_{c>a}E_c=\{0\},
\qquad
\overline{\bigcup_{c<b}E_c}^{\,E_b}=E_b,
\qquad
\iMor(T_b)=\lim_{c\uparrow b}\iMor(T_c),
\]
including infinite indices.
\end{addlem}
\begin{proof}
Let $m$ be a lower bound of the minimal operator $A$
and let $T$ be a self-adjoint extension. Since
$d:=\dim\dom(T)/\dom(A)<\infty$, every negative subspace
of $T-mI$ has dimension at most $d$. Its negative
spectral subspace is therefore finite-dimensional,
so $T$ is bounded below.

Write $E_{F,c}$ for the Friedrichs form domain on $(a,c)$.
Interior elliptic localization and the usual cutoff
approximation at the regular point $c$ give
\[
E_{F,c}
=
\{u\in E_{F,b}:u=0\text{ a.e. on }(c,b)\},
\]
with equality of the forms under zero extension.

The operator domain is a form core, and the minimal
domain has finite codimension in $\dom(T_b)$. Hence
\[
E_b=E_{F,b}+V
\]
for a finite-dimensional space $V\subset\dom(T_b)$.
For a fixed $c<b$, its representatives may be cut off
to the left of $c$ without changing their classes
modulo $E_{F,b}$: the removed parts lie in the regular
Dirichlet form domain described above. They preserve
the singular boundary condition and belong to
$\dom(T_c)$. Applying the same cutoff argument to
$\dom(T_c)$ gives
\[
E_b=E_{F,b}+V,
\qquad
E_c=E_{F,c}+V.
\]
Here $V$ is chosen for the fixed cut $c$.

For $f,g\in E_F$ and $v,w\in V$, the representation theorem
gives
\[
t_T(f+v,g+w)
=
t_F(f,g)+\langle f,A^*w\rangle
+\langle A^*v,g\rangle+\langle A^*v,w\rangle.
\]
Since the representatives are supported to the left of
$c$, this identity proves equality of the forms under
zero extension. The support characterization, closedness,
and nesting follow from the corresponding Friedrichs
statements. The intersection is zero by the support
condition. Finally, compactly supported approximation
of the Friedrichs component proves density of the
union. Approximating any finite-dimensional negative
subspace then proves the index limit.
\end{proof}

\paragraph{A regular mixed-boundary Morse formula.}
Assume $P>0$ on $(c,b)$, that $c$ is an interior regular
point, and that $b$ satisfies (H1). For a fixed
right-endpoint Lagrangian $\Gamma$,
\begin{equation}\label{eq:regular-mixed-count}
\iMor(L_{(c,b),\Lambda_D\oplus\Gamma})
=
\sum_{d\in(c,b)}
\dim\ker L_{(d,b),\Lambda_D\oplus\Gamma}.
\end{equation}
\begin{proof}
Write the condition $\Gamma$ as an endpoint value
constraint together with a symmetric Robin form $B$.
For $u(d)=0$, put
\[
A_d:=\int_d^b\|P^{-1}\|\,dt,
\qquad
E_d:=\int_d^b\langle Pu',u'\rangle\,dt.
\]
The preceding estimates give
\[
t_\Gamma[u]\ge
\bigl(
1-2\|Q\|_\infty A_d
-\|R\|_\infty(b-d)A_d
-\|B\|A_d
\bigr)E_d.
\]
Since $A_d\to0$, the realization is positive for $d$
sufficiently close to $b$.

For interior $d$, the forms have compact resolvent and
depend continuously on $d$, after identifying the
intervals by changes of variable that are the identity
near $b$. At a zero crossing $d_0$, the usual
Cauchy-data crossing calculation gives, on the kernel,
\[
G(u,v)=\langle P(d_0)u'(d_0),v'(d_0)\rangle,
\qquad
H(u,v)=\int_{d_0}^b\langle u,v\rangle\,dt.
\]
Both forms are positive definite. The local eigenvalue
equation has reduced form
\[
(d-d_0)G-\lambda H
+o(|d-d_0|+|\lambda|),
\]
so increasing $d$ removes exactly
$\dim\ker L_{(d_0,b),\Lambda_D\oplus\Gamma}$
negative eigenvalues, including at multiple crossings.
The crossings are isolated and are absent near $b$.
Counting backwards from the positive short interval
proves the formula. A zero mode at $c$ is excluded,
as required by the definition of the Morse index.
\end{proof}

 \begin{thm}{\bf [Morse Index Theorem]}\label{thm:main-oneside}
We assume that the SL-operator $l$ satisfies conditions (H1)-(H3)-(H4) and we assume Dirichlet boundary condition at $b$. If   $\iMor(L_{b,\Lambda})<\infty$ , then the following equality holds: 
\[
\iMor ( L _{b,\Lambda})= \sum_{\sigma \in (a,b)} \dim \ker  L_{(a,\sigma),\Lambda}.
\]
\end{thm}

The next result relates the Morse index of the Friedrich extension operator $L _F$ with the limit on $\sigma$ of the Friedrich operators $L_{(a,\sigma), F}$ and $L_{(\sigma,b), F}$ under the only assumption that the operator $L$ is lower bounded. In particular, neither Fredholmness nor the finiteness of the Morse index is assumed. 

\begin{thm}\label{thm:limit_morse_index}
We assume condition (H3). If $l$ is regular at $a$ and if we  assume  Dirichlet boundary condition at this endpoint, then we get 
\[
 \iMor(L_{F})=\lim_{\sigma\to b^-} \iMor(L_{(a,\sigma),F})=\sum_{t\in (a,b)} \dim\big(\gamma(t)\Lambda_D\cap \Lambda_D\big),
\]
where $\gamma$ is the fundamental matrix solution such that $\gamma(a)=\Id$.

If $l$ is regular at $b$ and we assume Dirichlet boundary condition at this endpoint, then  we have
\[
\iMor(L_{F})=\lim_{\sigma\to a^+} \iMor(L_{(\sigma,b),F})=\sum_{t\in (a,b)} \dim\big(\gamma(t)\Lambda_D\cap \Lambda_D\big),
\]
where $\gamma$ is the fundamental matrix solution such that $\gamma(b)=\Id$.
\end{thm} 
\begin{proof}
We establish the result in the case where $a$ is a regular endpoint, since the other case is entirely analogous.

 For, we start  by observing that $\iMor(L_{b,F})=\iMor(t_{L_b})$ where
     $t_{L_b}$  is the closure of the quadratic form $u\mapsto (l_b,u,u),u\in C_0^\infty((a,b),\R^n)$.
    Let $W$ be   any finite-dimensional  negative subspace of $t_{L_b}$ and let $\set{w_1,\cdots,w_m}$ be a basis of $W$.   For every $\delta>0$ there exist  $\set{w_1',\cdots,w_m'}\subset C_0^\infty((a,b),\R^n)$ such that  the following holds: 
    \[ 
    |t_{L_b}(w_i',w_j')-t_{L_b}(w_i,w_j)|<\delta\qquad 1\le i,j\le m.
    \]
In particular, if  $\delta>0$ is sufficiently small, also $\widetilde{W}=\Span\set{w_1',\cdots,w_m'}$ is  a negative subspace of $t_{L_b}$. By choosing $\sigma \in (a,b)$ such that $\supp w_i'\subset (a,\sigma), 1\le i\le m$, then we get 
  \[
\iMor(L_{[a,\sigma],F})\ge m.
\] 
Moreover $\dom(t_{L_{[a,\sigma]}})< \dom (t_{L_b})$; so, we have
\[
\iMor(L_{[a,\sigma],F})\le \iMor(t_{L_b}).
\]
  The truncated form domains are nested. Since $m$ can be the dimension of any finite-dimensional negative subspace, the preceding inequalities give
\[
\iMor(L_{b,F})=\lim_{\sigma\uparrow b}\iMor(L_{(a,\sigma),F}),
\] 
  including the value $+\infty$. For each interior $\sigma$, the regular mixed-boundary formula~\eqref{eq:regular-mixed-count}, with Dirichlet data at both ends and with time reversed when needed, gives
\[
\iMor(L_{(a,\sigma),F})=\sum_{t\in(a,\sigma)}\dim(\gamma(t)\Lambda_D\cap\Lambda_D).
\] 
  Taking the increasing limit proves the stated sum. The proof with the regular endpoint at $b$ is the same, using cutoffs from the left .
\end{proof}

\begin{rem}\label{rem:limit_morse_index}
In the Theorem \ref{thm:limit_morse_index}, we only assume that $l_b$ is lower bounded. No assumption of Fredholmness of $L_b$ is made.

    It is worth noting that,  if $\iMor(L_b)$ is infinite, then the theorem is still true since also the limit and the sum appearing in the thesis are all infinite.
\end{rem}

\subsection{Preliminary Results on Lagrangian Intersection Theory}

In this subsection, we prove some foundational results from Lagrangian intersection theory that will be used throughout the proofs. These include algebraic and topological properties of Lagrangian subspaces in symplectic vector spaces, as well as basic facts concerning the behavior of their intersections under perturbations.

We let $V_{\sigma,s}= \ker (L_{(a,\sigma)}^*+s\Id)$ and for $M>0$, we set
	\[
	m_{\sigma,s}:=\Tr_\sigma (p_\sigma(V_{\sigma,s})) \qquad (\sigma,s) \in (a,b) \times[0,M] 
	\]
    where $p_\sigma: \dom(L_{(a,\sigma)}^*) \to W$. 
	  The following lemma specifies the hypotheses under which $m_{\sigma,s}$ is Lagrangian throughout the rectangle .

\begin{lem}\label{thm:continuity-2-par-family}
Under assumptions (H1), (H3) \& (H4),  and assuming $L_b+sI$ is Fredholm for every $s\in[0,M]$,  the two-parameters family $(\sigma,s) \mapsto m_{\sigma,s}$ is continuous in  $ \Lag(\R^{2k-2n}\oplus\R^{2n},-\Omega\oplus \Omega)$.
\end{lem}
\begin{proof}
 Let $\psi_s(\sigma)$ be the  fundamental (matrix) solution  of the Hamiltonian system defined by the Sturm-Liouville operator $l_b+s I$ such that $\psi_{s}(b)=\Id$ and we let 
 \[
 G _{\sigma,s}:=\begin{pmatrix}\Id_{2k-2n}&0\\0&\psi_s (\sigma)\end{pmatrix}.
 \]
Since  the bijection 
	\[
	V_{b,s}\ni v_s \longmapsto v_s|_{[a,\sigma]} \in V_{\sigma,s}
	\] 
we get  that
	\[
	m_{\sigma,s}=G _{\sigma,s}\, m_{b,s} .
	\]
	By Lemma~\ref{lem:conti_ker}, the map $s\mapsto m_{b,s}$ is continuous. Moreover, 
	since   $(\sigma,s)\mapsto\psi_s(\sigma)$  is continuous (being $t=b$ a regular endpoint)  so the map  $(\sigma, s)\mapsto G _{\sigma, s}$ is continuous too. In conclusionn we get that the map  $(\sigma, s) \mapsto m_{\sigma, s}$ is  continuous.
 Fredholmness of the truncated minimal operators follows from Supplementary Lemma~\ref{lem:regular-cut-fredholm}; uniqueness gives (H0). The additional shift hypothesis holds automatically in the finite-index Morse theorem: a semibounded Fredholm self-adjoint realization of finite negative index has essential spectrum in $[\delta,\infty)$ for some $\delta>0$, or empty essential spectrum. All its nonnegative shifts, and their finite-codimensional minimal restrictions, are therefore Fredholm.
 \end{proof}

\begin{note}
From now on, we  drop the index $s$ in the case of $s=0$. So, in particular $m_{\sigma,0}=m_\sigma$ and $V_{\sigma,0}=V_\sigma$. 
\end{note}
We are going to prove that the path $\sigma \mapsto m(\sigma)$ is a plus curve with respect to  the Lagrangian subspace $\Lambda_0 \oplus  \Lambda_D$ meaning that all crossing instants are positive. A direct consequence of this fact is that the local contribution to the Maslov index of a crossing instant is provided by the its multiplicity. 

\begin{lem}\label{lem:positive_path}
Under   assumptions (H1), (H3), (H4),  we get that 
the path $\sigma \mapsto 	m(\sigma)$ is positive with respect to the Lagrangian subspace $\Lambda_0 \oplus \Lambda_D$. 
\end{lem}

\begin{proof}
	 By definition, we have 
  \[
m_\sigma=\left\{\begin{pmatrix}x\\\psi(\sigma)\,y\end{pmatrix}\left|\begin{pmatrix}x\\y\end{pmatrix}\in m_b\right.\right\}.
\] 
 and 
  
\mbox{
$\mathcal J=\begin{pmatrix}-J_{2k-2n}&0\\0&J_{2n} \end{pmatrix}$
}
 . Given a crossing instant  $\sigma \in (a,b)$ and setting $M_\sigma=\begin{pmatrix}\Id&0\\0&\psi(\sigma) \end{pmatrix}$, we get  that the crossing form 
\[
C_\sigma:= -\mathcal J\dfrac{d}{d\sigma}G _\sigma\, G _\sigma^{-1}|_{m_\sigma\cap (\Lambda_0\oplus \Lambda_D)}
\] 
is a positive definite quadratic form.  Actually, by a direct calculation, we get 
  \[
 -\mathcal J\left(\dfrac{d\,G _\sigma}{d\sigma}\right) G _\sigma^{-1}=\begin{pmatrix}0_{2k-2n} & 0\\ 0 & -J_{2n}\dfrac{d\, \psi}{d\sigma}\psi(\sigma)^{-1}
 \end{pmatrix}.
\] 
Let $u=\begin{pmatrix}u_a\\u_b\end{pmatrix}\in m_\sigma\cap (\Lambda_0\oplus \Lambda_D)$. By the uniqueness theorem for linear ODEs, we get that the following holds:
	  \[
	m_\sigma\cap(\R^{2k-2n}\oplus (0))=\Tr_\sigma(p_\sigma(V_\sigma))\cap (\R^{2k-2n}\oplus (0)) =(0).
	\] 
	Then $u_b\in \Lambda_D$  and $u\neq 0$ if and only if $u_b\neq 0$. For concluding the proof we only need to prove that 
\[
-J_{2n}\left(\dfrac{d\, \psi}{d\sigma}\right) \psi(\sigma)^{-1}\Big|_{\Lambda_D}
\]
is positive definite. To do so, we just observe that   \[-J_{2n}\left(\dfrac{d\, \psi}{d\sigma}\right) \psi(\sigma)^{-1}=\begin{pmatrix}P^{-1}&-P^{-1}Q\\-Q^TP^{-1}&Q^TP^{-1}Q-R\end{pmatrix}(\sigma).\] . Since $P^{-1}(\sigma)$ is positive definite, $-J_{2n}\dfrac{d\, \psi}{d\sigma}\psi(\sigma)^{-1}$ is a positive definite quadratic form on $\Lambda_D$.
Finally,  $C_\sigma$ is a positive quadratic form on $m_\sigma\cap(\Lambda_0\oplus\Lambda_D)$ concluding the proof.
\end{proof}
\begin{prop} \label{lem:maslov_plus}
Under assumptions (H1), (H3) \& (H4)  and for every   $a<c<d<b$ ,  we get
	  \[
	\iCLM(\Lambda_0 \oplus \Lambda_D,m_\sigma, -\Omega\oplus \Omega, \sigma \in [c,d])=\sum_{\sigma\in [c,d)} \dim\ \Big((\Lambda_0\oplus \Lambda_D)\cap m_\sigma\Big).
	\] 
\end{prop}
\begin{proof}
The proof follows directly upon observing that the crossing instants are regular, and hence isolated. Since the interval is compact, these crossing instants are finite in number. By Lemma~\ref{lem:positive_path}, each local contribution to the Maslov index is determined by the multiplicity of the crossing, namely, the dimension of the intersection of the path \( \sigma \mapsto m_\sigma \) with the reference Lagrangian \( \Lambda_0 \oplus \Lambda_D \). The desired conclusion is then obtained by summing these finitely many local contributions.
\end{proof}
We now prove that the Morse index of the self-adjoint extensions \( L_{\sigma, \Lambda} \) varies monotonically as a function of the parameter \( \sigma \).
	\begin{lem}\label{thm:Morse-index-regular}
		Let   $\Lambda=\Lambda_0\oplus\Lambda_D$ . Under the assumptions (H1),(H3) \& (H4) and  for every     $a<c\le b$ , the following inequality holds:
        \[
        \iMor (L_{c,\Lambda})\le \iMor (L_{b, \Lambda}).
        \]
        Moreover , if $\iMor(L_{b,\Lambda})<\infty$,  there exists $c_0\in (a,b)$   such that $\iMor(L_{c_0,\Lambda})=0$.	
	\end{lem}
\begin{proof}
 Supplementary Lemma~\ref{lem:local-form-domains} supplies semiboundedness and the closed supported form-domain identifications used below.  Let $t_{L_{c,\Lambda}}$ be the closure of the quadratic form defined by $\langle L_{c,\Lambda} \cdot, \cdot \rangle_{L^2}$ having domain 
\[
\dom(t_{L_{c,\Lambda}})= \overline{\dom(L_{c,\Lambda})}=\overline{\dom(L_c)+\Tr_c^{-1}(\Lambda)}
\]
where the closure is meant in the $\dom(t_{L_{c,\Lambda}})$.  By   Notation~\ref{def:global-morse}, including infinite indices,  we get that 
\[
\iMor(t_{L_{c,\Lambda}})=\iMor(L_{c,\Lambda}).
\]
Moreover 
		\[
		\dom(t_{L_{c,\Lambda}})=\dom(t_{L_c})+\Tr_c^{-1}(\set 0\oplus \Lambda_D) + \Tr_c^{-1}(\Lambda_0\oplus \set 0)=\dom(t_{L_c})+\Tr_c^{-1}(\Lambda_0\oplus \set 0).
		\]
		So, $\dom(t_{ L _{c,\Lambda}})=\dom(t_{L_c})+ W_c$ where  $W_c$ can be   a finite-dimensional  subspace of $\dom(L_c^*)$  such that $\Tr(W_c) =\Lambda_0\oplus \set 0$.

		For $d< c\le b$, we  choose  $W_d$ in such a way that  

        $\dom(t_{L_{c,\Lambda}})=D(t_{L_c})+ W_d$
		and $\dom(t_{L_{b,\Lambda}})=D(t_{L_b})+ W_d$ .

		The natural embedding $\iota: \dom(t_{L_c}) \hookrightarrow D(t_{L_b})$ induces the embedding between the domains
		$\dom( t_{L_{c,\Lambda}})\hookrightarrow  \dom(t_{L_{b,\Lambda}})$ pointwise defined by 
		\[
		u\mapsto  \widetilde u(x):=
        \begin{cases}u(x)& x\in [a,c]\\[3pt] 0  & x\in [c,b].\end{cases} 
		\]
		By this arguments and by invoking   Supplementary Lemma~\ref{lem:local-form-domains} , we immediately get  that 
	$\dom(t_{L_{c,\Lambda}})< \dom(t_{L_{b,\Lambda}})$ and $t_{L_{c,\Lambda}}=t_{L_{b,\Lambda}}|_{\dom (t_{L_{c,\Lambda}})}$.
	In particular, the following inequality holds
		\[
		\iMor (L_{c,\Lambda})\le \iMor (L_{b,\Lambda}).
		\]		
		Since $\bigcap_{a<\sigma<b} \dom(t_{L_{(a,\sigma),\Lambda}})=\set 0$, then by invoking Lemma~\ref{lem:Morse_vanish_abstract} and Lemma~\ref{thm:lemmaC-2},   there exists  $c_0\in (a,b)$   such that $\iMor(L_{c_0,\Lambda})=0$, concluding the proof.		

	\end{proof}

\subsubsection*{Proof of Theorem~\ref{thm:main-oneside}}
  First fix $e\in(a,b)$ and choose a nonconjugate $c\in(a,e)$; such cuts exist by Lemma~\ref{lem:positive_path}. We prove the formula with terminal endpoint $e$. All $L_{(a,\sigma),\Lambda}$ for $c\le\sigma\le e$ share a lower bound by Supplementary Lemma~\ref{lem:local-form-domains}; all their nonnegative shifts are Fredholm by Lemma~\ref{thm:continuity-2-par-family}. Thus  there exists $M>0$ sufficiently large such that 
  \[
\iMor( L _{e,\Lambda})-\iMor( L _{c,\Lambda})=-\iCLM(\Lambda, m_{e,s},-\Omega\oplus \Omega, s \in [0,M])+\iCLM(\Lambda,m_{c,s},-\Omega\oplus \Omega, s \in [0,M]).
\] 
By the stratum homotopy invariance property of Maslov index, we have 
  \begin{multline}\label{eq:importante}
	\iMor( L _{e, \Lambda})-\iMor( L _{c,\Lambda})\\
=	-\iCLM(\Lambda,m_{e,s},-\Omega\oplus \Omega, s \in [0,M])+\iCLM(\Lambda,m_{c,s},-\Omega\oplus \Omega, s \in [0,M])\\	=-\iCLM(\Lambda,m_{\sigma,M},-\Omega\oplus \Omega, \sigma \in [c,e])+\iCLM(\Lambda,m_{\sigma,0},-\Omega\oplus \Omega, \sigma \in [c,e]).
\end{multline} 
	By Equation~\eqref{eq:importante} and by Proposition~\ref{lem:maslov_plus}, we get that  for $M$ large enough,
	  \begin{multline} \label{eq:importante2}
		\iMor( L _{e, \Lambda})-\iMor( L _{c,\Lambda})
		=-\iCLM(\Lambda,m_{\sigma,M},-\Omega\oplus \Omega, \sigma \in [c,e])+\iCLM(\Lambda,m_{\sigma,0},\sigma \in [c,e]) \\
		=-\iCLM(\Lambda,m_{\sigma,M},-\Omega\oplus \Omega, \sigma \in [c,e])+\sum_{\sigma\in (c,e)} \dim \Big ((\Lambda_0\oplus \Lambda_D)\cap m_\sigma\Big).
	\end{multline} 
	Since   $ L _{e,\Lambda}$  is  bounded from below, then for $M>0$  large enough, we get that   $\iMor( L _{e,\Lambda}+M \Id )=0$ .
	By Proposition~\ref{lem:maslov_plus}, then we get that
	  \[
\begin{aligned}
\iCLM(\Lambda_0\oplus\Lambda_D,m_{\sigma,M},\sigma\in[c,e])
&=\sum_{\sigma\in(c,e)}\dim((\Lambda_0\oplus\Lambda_D)\cap m_{\sigma,M})\\
&=\sum_{\sigma\in(c,e)}\dim\ker(L_{(a,\sigma),\Lambda}+MI).
\end{aligned}
\] 
	We also observe that  for $M$ large enough, we can assume that   $\iMor ( L _{e,\Lambda}+(M-1) \Id)=0$ . 
	By  Lemma~\ref{thm:Morse-index-regular}, we get that  
	  \begin{equation}\label{eq:key-morse-index}
	\iMor ( L_{(a,\sigma),\Lambda}+(M-1) \Id)\le \iMor ( L _{e,\Lambda}+(M-1) \Id)=0 \qquad \sigma \in[c,e].
	\end{equation} 
	This implies that     $\dim\ker(L_{(a,\sigma),\Lambda}+MI)=0$ , (if not, the corresponding non-trivial vector in the kernel contributes to the number $\iMor ( L_{(a,\sigma),\Lambda}+(M-1) \Id)$). So, we get that for $M$ sufficiently large,
	  \begin{equation}\label{eq:importante3}
		\iCLM(\Lambda_0 \oplus \Lambda_D,m_{\sigma,M},-\Omega\oplus \Omega,  \sigma \in [c,e])=\sum_{\sigma\in (c,e)} \dim \ker( L_{(a,\sigma),\Lambda}+M \Id)=0.
	\end{equation} 
	By Equation~\eqref{eq:importante2} and Equation~\eqref{eq:importante3} we get that
	  \begin{equation}\label{eq:importante4}
		\iMor( L _{e, \Lambda})-\iMor( L _{c,\Lambda})
		=\sum_{\sigma\in (c,e)} \dim \big ((\Lambda_0\oplus \Lambda_D)\cap m_\sigma\big).
	\end{equation} 
	Passing to the limit for  $c$ converging  to $a$ at Equation~\eqref{eq:importante4}, we get 
	  \begin{equation}\label{eq:fin-1}
		\iMor( L _{e, \Lambda})\ge
		\sum_{\sigma\in (a,e)} \dim \big ((\Lambda_0\oplus \Lambda_D)\cap m_\sigma\big).
	\end{equation} 
	By Lemma~\ref{thm:Morse-index-regular}, there exists  a nonconjugate  $d>a$ such that  $\iMor( L _{d,\Lambda})=0$ and so, we get
	  \begin{equation}\label{eq:fin-2}
		\iMor( L _{e, \Lambda})=
		\sum_{\sigma\in (d,e)} \dim \big ((\Lambda_0\oplus \Lambda_D)\cap m_\sigma\big)\le \sum_{\sigma\in (a,e)} \dim \big ((\Lambda_0\oplus \Lambda_D)\cap m_\sigma\big).
	\end{equation} 
	Summing up the Equation~\eqref{eq:fin-1} and Equation~\eqref{eq:fin-2},  the formula is proved for every interior $e$. Finally Supplementary Lemma~\ref{lem:local-form-domains} gives
\[
\iMor(L_{b,\Lambda})=\lim_{e\uparrow b}\iMor(L_{(a,e),\Lambda}).
\]
Taking the increasing limit of the conjugate-point sums proves  Theorem~\ref{thm:main-oneside}  , without differentiating at the possibly degenerate regular endpoint $b$ . \qed

\section{Morse Index Theorem for General Boundary Conditions}\label{sec:morse-general}

The purpose of this section is twofold. First, we extend Theorem~\ref{thm:main-oneside} to the setting of general Lagrangian boundary conditions. The main result, stated in Theorem~\ref{thm:triple-index-Morse-Friedrich-extensions}, expresses the difference between the Morse indices of two self-adjoint Fredholm operators—obtained as the Friedrichs extensions of the corresponding quadratic forms—in terms of the triple index. 

Second, we establish that a perturbation of a two-sided singular Sturm--Liouville operator \( L \), which is bounded from below, may fail to remain bounded from below when the perturbation is given by a symmetric but not relatively bounded operator. This analysis sheds light on a phenomenon already noted by Rellich, wherein eigenvalues may “vanish” by escaping to \( -\infty \); see Kato~\cite[Ch.~V, §4, Thm.~4.1]{Kat80} for a precise formulation.

The difference between two operators subject to different boundary conditions has been investigated in \cite{HWY20} and \cite{HPWX20}, both for regular Sturm--Liouville operators and for Sturm--Liouville operators defined on the entire real line. In \cite[Theorem~1.2]{BCLS24}, this result has been further generalized to arbitrary symmetric operators. 

Here, by employing an approach analogous to that developed in \cite{HWY20}, we establish a difference formula for general Sturm--Liouville operators subject to distinct boundary conditions.

For the reader’s convenience, the proofs of the foundational results concerning Friedrichs extensions are deferred to Appendix~\ref{appendix:friedrichs}.

\begin{thm}\label{thm:triple-index-Morse-Friedrich-extensions}
Let \( L \) be a one-sided singular the Sturm--Liouville operator defined at Equation~\eqref{eq:sturm-liouville-operator}, and suppose that conditions   \textnormal{(H1)}, \textnormal{(H3)--(H4)} hold and $\iMor(L_\Lambda),\iMor(L_F)<\infty$ . Then we have:
\[
\iMor(L_{\Lambda}) - \iMor(L_{F}) = \itriple\bigl(p(\ker L^{*}),\, p(\dom(L_{\Lambda})),\, p(\dom(L_{F}))\bigr),
\]
where $\Lambda \in \Lag(W, \omega|_W)$, \(\itriple\) denotes the triple index (Cfr. Appendix~\ref{sec:Maslov} and references therein).
\end{thm}

\begin{proof}

Without loss of generality. We can assume that $\dom(L^*)=\dom(L)\oplus W$ such that
  $\ker L^*< W$ , $\Lambda\in \Lag(W,\rho)$, $F:=\dom (L_F)\cap W\in \Lag (W,\rho)$ and $L_\Lambda=L^*|_{\dom(L)\oplus\Lambda}$ . We start to observe that by Lemma~\ref{thm_=Morse}, we get that  $\iMor(L_\Lambda)=\iMor(t_{L_\Lambda})$ and  $\iMor(L_F)=\iMor(t_L)$. 	By Lemma~\ref{thm:nuovo}, we get that $\ker L^*$ is a Lagrangian subspace of $W$ and by invoking Lemma~\ref{thm:general_morse_diff} for $G=\dom(t_L)=\dom(t_{L_F})$, $V=\dom(t_{L_\Lambda})$ and $Q=t_{L_\Lambda}$,  to compute the difference of the Morse indices, it's enough to compute 
	\[
	\dim (\dom(t_{ L })\cap (\dom(t_{ L }))^{t_{ L _\Lambda}}+\ker t_{ L _\Lambda} )/\ker t_{ L _\Lambda}.
	\]
	 By Lemma~\ref{lem:ker_form} and Corollary~\ref{cor:Friedrich_of_SA}, we get that $\ker t_{ L _\Lambda}=\ker  L _\Lambda$.
	 The inclusion of the Friedrichs form domain is closed and of finite codimension by Lemma~\ref{lem:characterize_form_domain}. The full form has a bounded Fredholm representative by Lemma~\ref{thm:lemmaC-2} (applied to the self-adjoint extension), so Lemma~\ref{thm:general_morse_diff} applies.  Then we have
	\begin{align}
		\dom(t_{ L })\cap (\dom(t_{ L }))^{t_{ L _\Lambda}}=(\dom(t_{L}))^{t_L}=\ker  L_F 
	\end{align}

 We get 
	\begin{align*}
	\dim \left[(\dom(t_{ L })\cap (\dom(t_{ L }))^{t_{ L _\Lambda}}+\ker t_{ L _\Lambda} )/\ker t_{ L _\Lambda}\right] =\dim \left[(\ker L_F+\ker L_\Lambda)/\ker L_\Lambda\right]\\
    =\dim \left[(\ker L_F)/(\ker L_F\cap \ker L_\Lambda)\right]=\dim( \ker L^*\cap F)-\dim (\ker L^*\cap F\cap \Lambda) .
	\end{align*}

	Then by Lemma~\ref{lem:index_orth_compl} and Lemma~\ref{thm:general_morse_diff}, we finally get
	\[
	\iMor( L _\Lambda)-\iMor( L_F)= \itriple(\ker  L ^*, \Lambda,F).
	\] 
	This concludes the proof. 
	\end{proof}

\begin{cor}
 Assume (H1), (H3), (H4) and finite Morse indices.  Consider a decomposition $\dom(L^*)=\dom(L)\oplus U$ with $\ker L^*< U$. Let $\Lambda_0,\Lambda_1\in \Lag(U,\rho)$.
Let $F=U\cap \dom(L_F)$ where $L_F$ is the Friedrich extension of $L$.
Then we have
	\[
	\iMor( L _{\Lambda_1})-\iMor( L_{\Lambda_0})= \itriple(\ker L^* , \Lambda_1,\Lambda_0)-\itriple(\Lambda_1,\Lambda_0,F).
	\]
\end{cor}
\begin{proof}
By Theorem~\ref{thm:triple-index-Morse-Friedrich-extensions}, we get that 
  \[
\iMor(L_{\Lambda_i})-\iMor(L_{F})= \itriple(\ker L^*, \Lambda_i,F) \qquad\textrm{ for } \  i=0,1.
\]  
By this identity and by using Lemma \ref{lem:triple_diff_circle_permu} and Proposition \ref{thm:mainli}, we get that
  \begin{align}
\iMor(L_{\Lambda_1})-\iMor(L_{\Lambda_0})&=\itriple(\ker L^*,\Lambda_1,F)-\itriple(\ker L^*,\Lambda_0,F)\\
&=[
\itriple(F,\ker L^*,\Lambda_1)-\dim(F\cap\Lambda_1)]-[\itriple(F,\ker L^*,\Lambda_0)-\dim(F\cap \Lambda_0)]\\
&=s(F,\ker L^*,\Lambda_0,\Lambda_1)-\dim(F\cap\Lambda_1)+\dim(F\cap\Lambda_0)
\end{align} 
By invoking once again Proposition~\ref{thm:mainli}, we get that 
\begin{align*}
s&(F,\ker L^*,\Lambda_0,\Lambda_1)-\dim(F\cap\Lambda_1)+\dim(F\cap\Lambda_0)\\
&=-s(F,\ker L^*,\Lambda_1,\Lambda_0)-\dim(F\cap\Lambda_1)+\dim(F\cap\Lambda_0)\\
&=-\itriple(F,\Lambda_1,\Lambda_0)+\itriple(\ker L^*,\Lambda_1,\Lambda_0)-\dim(F\cap\Lambda_1)+\dim(F\cap\Lambda_0)\\
&=\itriple(\ker L^*,\Lambda_1,\Lambda_0)-\itriple(\Lambda_1,\Lambda_0,F).
\end{align*}
The last equation comes from Lemma \ref{lem:triple_diff_circle_permu}.
\end{proof}
We close this section with an additional discussion about infinite Morse index.

\begin{prop}\label{thm:morse-infinito}
Let $\Lambda=\Lambda_0\oplus \Lambda_D$ and we assume conditions ( H1), ( H3) \& (H4). 
Then $\iMor(L_{\Lambda})=+\infty$ if and only if 
\[
\sum_{d\in(a,b)}\dim \ker L_{(a,d),\Lambda_0\oplus \Lambda_D}=+\infty.
\]
\end{prop}

\begin{proof}
We start by observing that the map  $\Tr_a:\ker L_b^*\to \R^{2k-2n} $ is a surjection.
Then we have
\begin{multline}
\dim\big[ \Tr^{-1}(\Lambda_0\oplus \R^{2n})\cap \ker L_b^*\big]=\dim\big[\Tr_a^{-1}(\Lambda_0)\cap \ker L_b^*\big] =\dim \ker L_b^*-\codim(\Lambda_0)\\
=\dim \ker L_b^*-(2k-2n-\dim \Lambda_0)=k-(2k-2n-(k-n))=n.
\end{multline}
By the existence and uniqueness theorem for solution of linear odes, it follows that $\Tr_b: \ker L_b^*\to \R^{2n}$ is an injection and in particular   $\Gamma:=\Tr_b(\Tr_a^{-1}(\Lambda_0)\cap\ker L_b^*)$  is a $n$-dimensional  subspace of $(\R^{2n},\omega_b)$.

We are now ready to prove that   $\Gamma$  is a Lagrangian subspace of $(\R^{2n},\omega_b)$. We only need to show that it is isotropic.
Let $u_1,u_2\in \Gamma$ with   $u_1=(f_1^{[1]}(b),f_1(b)),u_2=(f_2^{[1]}(b),f_2(b))$ where $f_1,f_2\in \Tr_a^{-1}(\Lambda_0)\cap\ker L_b^*$ .
We have
\[
0=\omega(f_1,f_2)=[f_1,f_2](b)-[f_1,f_2](a+).
\]
Since   $\Tr_a(f_1),\Tr_a(f_2)\in \Lambda_0$  which is a Lagrangian subspace, we have   $[f_1,f_2](a+)=0$.
Then we get $[f_1,f_2](b)=0$ which imply that $\omega_b(u_1,u_2)=0$.
Then $\Gamma$ is isotropic and being $n$-dimensional then  $\Gamma\in \Lag(\R^{2n},\omega_b)$.
With these notation, we have
\[
\dim \ker L_{(a,c),\Lambda_0\oplus \Lambda_D}=\dim \ker L_{(c,b),\Lambda_D\oplus \Gamma}
\]
Then we can conclude that 
\[
\sum_{d\in(c,b)}\dim \ker L_{(a,d),\Lambda_0\oplus \Lambda_D}=\sum_{d\in(c,b)}\dim \ker L_{(d,b),\Lambda_D\oplus \Gamma}.
\]
By   the independently proved regular mixed-boundary formula~\eqref{eq:regular-mixed-count},  we have
\[
\iMor(L_{(c,b),\Lambda_D\oplus \Gamma})=\sum_{d\in(c,b)}\dim \ker L_{(d,b),\Lambda_D\oplus \Gamma}=\sum_{d\in(c,b)}\dim \ker L_{(a,d),\Lambda_0\oplus \Lambda_D}.
\]
  By Proposition~\ref{prop:estimate_morse_self_ajoint}, $L_\Lambda$  has infinite Morse index   if and only if $L_F$ does, since their domains differ by finitely many dimensions . By Theorem~\ref{thm:limit_morse_index}, it is equivalent to 
\[
\lim_{c\to a^+}\iMor(L_{(c,b),\Lambda_D\oplus \Lambda_D})=+\infty
\]
  The two regular realizations with right conditions $\Gamma$ and $\Lambda_D$ differ by at most $n$ domain dimensions on each side of their common restriction. Proposition~\ref{prop:estimate_morse_self_ajoint} therefore bounds their index difference uniformly in $c$. Consequently, 
 \[
\lim_{c\to a}\iMor(L_{(c,b),\Lambda_D\oplus \Gamma})=+\infty.
\]

 All these equivalences are reversible.  That is 
\[
\lim_{c\to a^+}\sum_{d\in(c,b)}\dim \ker L_{(a,d),\Lambda_0\oplus \Lambda_D}=\sum_{d\in(a,b)}\dim \ker L_{(a,d),\Lambda_0\oplus \Lambda_D}=+\infty.
\]

\end{proof}

\section{About a Rellich counterexample to a  spectral flow formula}\label{sec:Rellich }

In this section, we use the abstract index theory to examine Rellich’s classical counterexample as presented in Kato~\cite[Chapter V, Section 4, Remark 4.13]{Kat80}.

It is well known that if \(T\) is a closed, self‑adjoint operator bounded below and \(A\) is a relatively bounded symmetric perturbation with sufficiently small bound, then the operator 
\[
S = T + A
\]
remains self‑adjoint and bounded below (see Kato \cite[Chapter V, Section 4, Theorem 4.1]{Kat80}).  However, this conclusion can fail when the perturbation is not relatively bounded: Rellich’s scalar counterexample shows that, in such cases, certain eigenvalues may “disappear at \(-\infty\).”    Here we determine the count in the singular boundary-condition setting treated below .  The purpose of this section is to address that question and to provide a precise measure of the multiplicity of eigenvalues that escape to \(-\infty\).

We start by recalling a well-known fact about  a continuous path of self-adjoint  Fredholm operators bounded from below,   defined on the same domain and  parametrized by a compact interval.   Under finite Morse index and a uniform lower bound,  the spectral flow   measures  the difference of the Morse indices of the operators at the boundary points of the interval.  More precisely, let   $s\mapsto A_s$ be a gap-continuous  path of self-adjoint Fredholm operators   with finite Morse indices and a common lower bound . Then the spectral flow $\spfl(A_s ,s\in[0,1])$ is the net number of eigenvalue crossing the $0$ when the parameter $s$ runs from $0$ to $1$ whilst 
$\iMor( A_1)-\iMor(A_0)$ measures  the  difference of the negative spectral spaces between the operator at the instant $1$ and the one at $0$. Under the above condition, we get that
\begin{equation}\label{eq:spectral-flow-diff-morse-indices}
\iMor(A_1)-\iMor(A_0)+\spfl (A_s, s\in[0,1])=0.
\end{equation}
Things  drastically change if the  domains of the operators are not anymore constant. This is, for instance, the case of a gap-continuous path of self-adjoint extensions $s\mapsto A_{\Lambda_s}$ of an operator  $A\in \CFs(H)$ for a gap-continuous path $s\mapsto \Lambda_s$ of Lagrangian subspaces of the symplectic space 
 $(U, \rho))$.  In this case, the domain is not anymore fixed  and so, the operators are not anymore relative bounded perturbation of a fixed one. 

 In this new situation,  Equation~\eqref{eq:spectral-flow-diff-morse-indices}   need not hold without uniform control of the lower bound.  The lack of the equality is very much related to the phenomenon observed by Rellich through a counterexample to the equality provided at Equation~\eqref{eq:spectral-flow-diff-morse-indices}  in the case of Sturm-Liouville operators on $[0,1]$ with varying domains.

In the next result we provide a sharp formula measuring the difference of the Morse indices of the operators $L_{\Lambda_1}$ and $ L_{\Lambda_0}$ in terms of the spectral flow of the associated path $s\mapsto  L_{\Lambda_s}$ and of the Maslov index of the Lagrangian path.  Without loss of generality, we will consider the decomposition 
$\dom(L^*)=\dom(L)\oplus W$
with $\ker L^*< W$.

We let 
\[
V:=\ker  L^*\quad \textrm{ and } \quad \Lambda_F=\dom(L_F)\cap W.
\]
\begin{thm}\label{lem:eigenvalue_minus_infty}
Assume conditions ( H1), ( H3)  ,  (H4)  and let $s\mapsto\Lambda_s$ be continuous . If $\iMor(L_{\Lambda_0})<\infty$, then we get:
	\[
	\iMor(L_{\Lambda_1})-\iMor(L_{\Lambda_0})+\spfl (L_{\Lambda_s}, s\in[0,1])=\iCLM(\Lambda_F,\Lambda_s,\rho, s\in[0,1]).
	\]
\end{thm}
\begin{proof}
By Theorem~\ref{thm:triple-index-Morse-Friedrich-extensions}, we get that 
\[
\iMor(L_{\Lambda_s})-\iMor(L_{F})= \itriple(p(V), \Lambda_s,\Lambda_F).
\] 
By invoking Lemma~\ref{lem:triple_diff_circle_permu} and Proposition~\ref{thm:mainli}, we get that
\begin{align}
\iMor(L_{\Lambda_1})-\iMor(L_{\Lambda_0})&=\itriple(p(V),\Lambda_1,\Lambda_F)-\itriple(p(V),\Lambda_0,\Lambda_F)\\
&=[\itriple(\Lambda_F,p(V),\Lambda_1)-\dim(\Lambda_F\cap\Lambda_1)]-[\itriple(\Lambda_F,p(V),\Lambda_0)-\dim(\Lambda_F\cap \Lambda_0)]\\
&=s(\Lambda_F,p(V),\Lambda_0,\Lambda_1)-\dim(\Lambda_F\cap\Lambda_1)+\dim(\Lambda_F\cap\Lambda_0)
\end{align}
where we denoted by $s(\Lambda_F,p(V),\Lambda_0,\Lambda_1)$ the H\"ormander index of the Lagrangian quadruple. (Cf.  Appendix~\ref{sec:Maslov}). 
Now, by using the spectral flow formula proved at Theorem~\ref{thm:Sturm_Sf_formula}: 
\[
\spfl(L_{\Lambda_s},s\in[0,1])=-\iCLM(\Lambda_s,p(V),\rho, s\in[0,1])
\] 
and since of  Definition~\ref{def:hormander}, we get that
\[
\iCLM(\Lambda_s,p(V),\rho, s\in[0,1])-\iCLM(\Lambda_s,\Lambda_F,\rho, s\in[0,1])=s(\Lambda_F,p(V),\Lambda_0,\Lambda_1).
\]
Summing up we conclude that 
  \begin{multline}
\iMor(L_{\Lambda_1})-\iMor(L_{\Lambda_0})+\spfl (L_{\Lambda_s}, s\in[0,1])\\	=s(\Lambda_F,p(V),\Lambda_0,\Lambda_1)-\dim(\Lambda_F\cap\Lambda_1)+\dim(\Lambda_F\cap\Lambda_0)-\iCLM(\Lambda_s,p(V),\rho, s\in[0,1])\\
=-\iCLM(\Lambda_s,\Lambda_F,\rho, s\in[0,1])-\dim(\Lambda_F\cap\Lambda_1)+\dim(\Lambda_F\cap\Lambda_0)\\
=\iCLM(\Lambda_F,\Lambda_s,\rho, s\in[0,1]).
\end{multline} 
This concludes the proof. 
\end{proof} 
The reason behind Theorem~\ref{lem:eigenvalue_minus_infty} is due to the fact  that some  negative eigenvalue disappearing  at $-\infty$. The total number of such a kind of {\em ghost eigenvalues} can be estimated in terms of the  Maslov index. 
\begin{thm}\label{thm:Rellich }
  Assume (H1), ( H3)  , (H4), and let $s\mapsto\Lambda_s$ be continuous. If     $\Lambda_s\cap \Lambda_F=(0)$ for $s\neq 0$, then  there exists  $K>0$ such that for every $M>K$ there exists $\varepsilon >0$ such that the following holds: 
\begin{equation}\label{eq:infinity-eigenvalues}
\#((-\infty,-M)\cap \mathfrak{sp}( L_{\Lambda_u}))= \iCLM(\Lambda_F,\Lambda_s,\rho, s\in[0,\epsilon]) \qquad  u\in (0,\varepsilon),
\end{equation}
where  eigenvalues in  $\mathfrak{sp}$   are counted with multiplicity . 
\end{thm}
\begin{rem}
  The interval size  $\varepsilon$   may depend on $M$: the quantifier order is $\exists K\,\forall M>K\,\exists\varepsilon_M$. On a punctured interval where $\Lambda_s\cap\Lambda_F=0$, additivity and the absence of crossings make the boundary index independent of the upper endpoint.  The   theorem counts escaping eigenvalues; it does not assert monotonicity of their individual branches or a uniform bound on all remaining positive eigenvalues .
\end{rem}
\begin{proof}
 Every extension is semibounded by Supplementary Lemma~\ref{lem:local-form-domains}.  Choose $K>0$   so that $L_{\Lambda_0}+KI$ is strictly positive and invertible. Its minimal restriction is Fredholm by finite domain codimension. Every other extension of this shifted minimal operator has finite negative index: a negative subspace intersects the positive reference domain trivially, so its dimension is bounded by the finite relative domain codimension. The shifted Friedrichs boundary plane is still $\Lambda_F$. These observations justify applying Theorem~\ref{lem:eigenvalue_minus_infty} to the shifted family.
 Since $s\mapsto L_{\Lambda_s}+K\Id$ is continuous, then there exists $\varepsilon>0$ such that $\ker (L_{\Lambda_u}+K\Id) =\set 0$ for every $ u\in [0,\varepsilon]$. By this we get that 
\[
\spfl(L_{\Lambda_r}+K\Id, r\in [0,u])=0\qquad  \textrm{ for every } u\in [0,\varepsilon].
\]
By Theorem~\ref{lem:eigenvalue_minus_infty}, we get
\[
\iMor(L_{\Lambda_u}+K\Id)=\iMor(L_{\Lambda_u}+K\Id)-\iMor(L_{\Lambda_0}+K\Id)=\iCLM(\Lambda_F,\Lambda_s,\rho, s\in [0,u]) .
\]
Since $\Lambda_F\cap \Lambda_s=(0)$ and since $s\neq 0$, we have
\[
\iMor(L_{\Lambda_u}+K\Id)=\iMor(L_{\Lambda_u}+K\Id)-\iMor(L_{\Lambda_0}+K\Id)=\iCLM(\Lambda_F,\Lambda_s,\rho
, s\in[0,\epsilon])\qquad u\in (0,\varepsilon).
\]
So, the number of  eigenvalues of $L_{\Lambda_u}$  below $-K$  is $\iCLM(\Lambda_F,\Lambda_s,\rho, s\in [0,\varepsilon])$ for every  $u\in (0,\varepsilon)$. 

For $M>K$, the operator $L_{\Lambda_0}+M\Id >0$. So,  with the same reason, we can choose $\epsilon$ and get the same equation.  Then, the result  follows.
\end{proof}

\begin{cor}
Assume   (H1), ( H3)  , (H4), continuity of $s\mapsto\Lambda_s$, and $\Lambda_s\cap\Lambda_F=(0)$ for $s\neq0$ . Then there is  $K>0$ and $ a>0$  such that 
\[
\#((-\infty,-K)\cap \mathfrak{sp}(L_{\Lambda_u}))=\iCLM(\Lambda_F,\Lambda_s,\rho,s\in [0,a]) \qquad  \forall \, u\in (0,a].
\]
Let $k=\iCLM(\Lambda_F,\Lambda_s,\rho,s\in [0,a])$.
Then for each $u\in (0,a]$, there are precisely $k$ eigenvalues of $L_{\Lambda_u}$ less than $-K$, namely 
 $\lambda_1(u)\le,\cdots,\le\lambda_k(u)$. Then we have
\[
\lim_{u\to 0^+} \lambda_i(u)=-\infty \qquad  1\le i\le k.
\]
\end{cor}

\begin{proof}
By Theorem~\ref{thm:Rellich }, there is $K>0,a>0$
\[
\#((-\infty,-K)\cap \mathfrak{sp}(L_{\Lambda_u}))=\iCLM(\Lambda_F,\Lambda_s,\rho,s\in [0,a])\qquad  \forall\, u\in (0,a].
\]
For any $M>K$, it is possible to choose $0<a'<a$ such that 
\[
\#((-\infty,-M)\cap \mathfrak{sp}(L_{\Lambda_u}))=\iCLM(\Lambda_F,\Lambda_s,\rho,s\in [0,a]) \qquad \forall \, u\in (0,a'].
\]
So, for  $u\in (0,a')$ we have
\[
(-\infty,-K)\cap \mathfrak{sp}(L_{\Lambda_u}) = (-\infty,-M)\cap \mathfrak{sp}(L_{\Lambda_u}).
\]
By this it follows that  $\lambda_i(u) <-M$ for $ 1\le i\le k$ and for $u\in (0,a')$.
Then we get 
\[
\lim_{u\to 0^+} \lambda_i(u)=-\infty \qquad   1\le i\le k.
\]
This concludes the proof.
\end{proof}

\section{Bessel-type differential operators}\label{sec:Bessel}
The aim of this section is to establish the functional-analytic and spectral properties of Bessel-type differential operators, and to provide sufficient conditions under which the index theory developed below applies.

We consider the unperturbed {\sc  Bessel-type differential operator}  $h_q$ on $\mathscr C_0^{\infty}((0,1],\R^n)$ defined by 
	\[
h_q:=-\dfrac{d^2}{dt^2}+\dfrac{q}{t^2}\quad \textrm{ for }\quad t\in (0,1) \textrm{ and } \quad q<3/4
	\]
and we are interested in finding the solution space of $h_q$. A convenient  way to do so, is to reparametrize  the  differential operator $h_q$ through the following function $r\mapsto q(r)$ where  
\begin{equation}\label{eq:q(r)}
q(r)=\begin{cases}
-1/4 +r^2 & r\ge 0\\[3pt]
-1/4-r^2&  r<0.
\end{cases}
\end{equation}
So, in this case the operator $h$ can be written as 
\[
h_r = -\dfrac{d^2}{dt^2} + \dfrac{q(r)}{t^2}, \qquad t \in (0,1).
\]
and we observe that the function   $q$  is a bijection from $(-\infty,1)$ to $(-\infty,3/4)$:
\begin{center}
\begin{tikzpicture}
\begin{axis}[
    width=10cm, height=6cm,
    xlabel={$r$}, ylabel={$q(r)$},
    axis lines=middle,
    ymin=-3, ymax=3, xmin=-2, xmax=2,
    grid=both,
    domain=-2:2,
    samples=400,
    thick,
    legend style={at={(0.02,0.98)},anchor=north west}
]
\addplot[blue] {x >= 0 ? -0.25 + x^2 : -0.25 - x^2};
\addlegendentry{$q(r)$}
\end{axis}
\end{tikzpicture}
\end{center}
\begin{rem}\label{rem:important}
The  threshold   $q=-1/4$ is the Hardy semiboundedness threshold. The scalar  limit-point/limit-circle   transition at zero occurs at $q=3/4$: the endpoint is limit circle for $q<3/4$ and limit point for $q\ge3/4$ . In particular   $q(r)=-1/4+r^2$ can still be negative for $0<r<1/2$; nonnegative $r$ does not by itself mean a repulsive potential.
 \end{rem}
By a direct computation we get that   the scalar  equation has two linearly independent solutions in   $L^2(0,1)$  given by 
\begin{align}\label{eq:sol_singular}
	&y_{1,r}(t)=\dfrac{1}{2} (t^{1/2-r}+t^{1/2+r})&& y_{2,r}(t)= \dfrac{1}{2\,r}(t^{1/2+r}-t^{1/2-r})  &&\textrm{ for } r\in (0,1)\\[3pt]
	&y_{1,r}(t)=t^{1/2}  && y_{2,r}(t)=t^{1/2}\ln t  && \textrm{ for }r=0\\[3pt]
	& y_{1,r}(t)= t^{1/2}\cos(r \ln t)&&   y_{2,r}(t)=\dfrac{1}{r}t^{1/2}\sin(r \ln t) &&\textrm{ for } r \in (-\infty, 0).
\end{align}
 For the system, multiply these scalar solutions by the coordinate vectors of $\mathbb R^n$, obtaining $2n$ independent solutions.

 By a  direct calculation, it readily follows that  
\[
[y_{1,r},y_{2,r}](0)=-1.
\]

\subsection{Functional properties of the Bessel-type operator}\label{subsec:properties-bessel}

 \paragraph{Finite boundary dimension without Fredholmness.}
For the real symmetric differential expressions used here, finiteness of the boundary quotient does not require Fredholmness. Indeed, complexify the minimal operator $A$. The von Neumann graph-domain decomposition is
\[
 \dom(A^*)=\dom(A)\dotplus\ker(A^*-i\Id)\dotplus\ker(A^*+i\Id).
\]
Each deficiency space is a space of solutions of a $2n$-dimensional first-order system. Evaluation at an interior point is injective by uniqueness, hence each deficiency dimension is at most $2n$. Complex conjugation interchanges the two deficiency spaces, so their dimensions are equal and self-adjoint extensions exist. Consequently the boundary quotient has dimension at most $4n$, and any two self-adjoint realizations differ by finitely many domain dimensions. The same reasoning applies to restrictions to either side of an interior cut. This supplies the finite-codimension comparisons below before Fredholmness has been established; it does not use Lemma~\ref{thm:beta-finit-dimensional} in reverse.

 The aim of this paragraph is to prove the Fredholmness of the Bessel differential operator and to study its stability properties under bounded perturbation. We start with two abstract results. The first result gives a sufficient condition for any self-adjoint extension of a SL-operator to have compact resolvent. The second is an abstract functional analytic result about the stability property of an operator to remaining a  compact resolvent operator once perturbed by a bounded one. 

\begin{lem}\label{lem:conti_ker_operator}

Let $l$ be a Sturm–Liouville operator defined on   $\mathscr C_0^\infty((0,1),\R^n)$   and we assume that the operator $l$ is regular at $t=1$. 

If  $\dim \ker L^* = 2n$, then $L$ is a symmetric Fredholm operators. Furthermore, any self-adjoint extension of $L$ has compact resolvent.
\end{lem}

\begin{proof}
For a Sturm Liouville equation $-\dfrac{d}{dt}(P(t)\dfrac{d}{dt}+Q(t))u+\trasp{Q}\dfrac{d}{dt}u+R(t)u=0$.
Let $M(t)$ be the monodromy matrix of the Hamiltonian system induced by the Sturm-Liouville operator $l$ such that $M(1)=\Id$. Let $M(t)=\begin{bmatrix}U(t)\\V(t)\end{bmatrix}$. Then $V(t)=(v_1(t),\cdots,v_{2n}(t))$ where $\set{v_1,\cdots,v_{2n}}$ is a basis of $\ker L^*$. Now we solve the equation $L v=f$ with $f\in \image(L)$. 
Since that $M(t)$ is a symplectic matrix, then we have
\[
\begin{pmatrix}
v^{[1]}\\
v
\end{pmatrix}(t)
=M(t)c+M(t)\int_1^t M(s)^{-1} J\begin{pmatrix}0\\f\end{pmatrix}ds=M(t)c+M(t)\int_1^t J\trasp{M}(s)\begin{pmatrix}0\\f\end{pmatrix} ds
\]
for some constant vector $c$. Since $v\in \dom (L)$ so $c=0$. We finally get 
\[
v(t)=(L^{-1} f)(t) = \int_1^t V(t) J \trasp{V}(s) f(s)\, ds \qquad \text{for } f \in \operatorname{im} L.
\]
Since $V\in L^2([0,1],\R^{n\times 2n})$, $L^{-1}$ is bounded on $\image (L)< L^2([0,1],\R^n)$.  Since $L$ is closed, $L^{-1}$ is also closed. So, $\image(L)=\dom(L^{-1})$ is closed since $L^{-1}$ is bounded. 
  Also $\ker L=0$ by endpoint uniqueness, and $\operatorname{coker}L\cong\ker L^*$ has dimension $2n$. Thus $L$ is Fredholm.
The same integral formula defines a Hilbert--Schmidt operator $K$ on all of $H$, and variation of constants gives $Kf\in\dom(L^*)$ and $L^*Kf=f$. If $T$ is any self-adjoint extension and $R_z=(T-zI)^{-1}$ for nonreal $z$ after complexification, then
\[
F_z:=R_z-K(I+zR_z)
\]
has range in $\ker L^*$, because $L^*R_z=I+zR_z$. Thus $F_z$ has finite rank and $R_z=K(I+zR_z)+F_z$ is compact, proving the remaining assertion .
\end{proof}
\begin{lem}\label{lem:pertub_compact_resolvent}
Let $A$ be a self-adjoint operator having compact resolvent. For each bounded operator $B$, $A+B$ also has compact resolvent. 
\end{lem}
\begin{proof}
Note that $\|(A+i\lambda \Id)^{-1}\|\le 1/\lambda$ for $\lambda>0$.
Choose $\lambda$ large enough. We can assume that $\|B(A+i\lambda \Id)^{-1}\|<1$.
We have
\[
(A+B+i\lambda \Id)^{-1}= (A+ i\lambda \Id)^{-1}(\Id +B(A+i\lambda I)^{-1} )^{-1},
\]
which is compact.
\end{proof}

\begin{cor}\label{thm:corollary7-4}
Let us consider the Bessel operator 
\[
l_q:=-\dfrac{d^2}{dt^2}+\dfrac{q}{t^2}
\]
acting  on   $C_0^\infty((0,1),\R^n)$  and we denote by $L_q: \dom(L_q) < H \longrightarrow H$ the corresponding minimal operator where $H=L^2([0,1], \R^n)$. 

Assume that $q<3/4$ and let $B\in \Bsa(H)$. Then any self-adjoint extension of $L_q+B$ is a compact resolvent operator.
\end{cor}
\begin{proof}
By Lemma~\ref{lem:conti_ker_operator} and Lemma~\ref{lem:pertub_compact_resolvent}, we only need to compute the  $\dim\ker L_q^*$. Such a computation can be reduced to the scalar case. So, by the  Equation~\eqref{eq:sol_singular}, we get that the kernel dimension in the scalar case is 2. By this, we conclude that $\dim\ker L_q^*= 2n$. 
 A self-adjoint extension $T$ of $L_q+B$ has $T-B$ as a self-adjoint extension of $L_q$, by the bounded perturbation lemma. Thus the compact-resolvent conclusion transfers to every extension, not just to one selected boundary condition.
 \end{proof}
The next step is to study the Fredholmness properties of the operator $L_q$. For 
 $q\ge 3/4$, this will be done by using a result of Naimark  (cf. \cite{Nai68})  which is known in literature as  {\em splitting method}. Here for the sake of the reader, we write it in the case of  Sturm-Liouville operators, we are interested in. 

\begin{lem}[Naimark, splitting method]\label{thm:ess_split}
Let $l$ be the SL-operator defined on 
$C_0^\infty((a,b),\R^n)$. Let $c\in (a,b)$.
Let $l_1,l_2$ be the restriction of $l$ on $C_0^\infty((a,c),\R^n)$ and $C_0^\infty((c,b),\R^n)$ respectively.

Let $L$ be the corresponding minimal operator associated to $l$.
Let $L_1$ $L_2$ be the corresponding minimal operators associated  to $l_1,l_2$, respectively.
The   essential  spectrum of every self-adjoint extension of the operator $L$ is the union of the   essential  spectra of any self-adjoint extensions of $L_1$ and $L_2$
\end{lem}
\begin{proof}
    For the proof we refer the interested reader to \cite[Theorem 24.1.1]{Nai68}
\end{proof}

\begin{prop}\label{thm:Fredholmness}
	Given  $R\in\mathscr  C^0([0,1],\Sym(n))$,  we define $l_R$ to be the differential operator acting on $\mathscr C_0^\infty((0,1), \R^n)$ and given by 
	\[
    l_R:=-\dfrac{d^2}{dt^2}+\dfrac{R(t)}{t^2} \qquad t \in (0,1)
	\] 
    and we assume that $\sigma (R(0))\subset (-1/4,+\infty)$. Denoting by $L_R$ (resp. $L_R^*$) the minimal (resp. maximal) operator  induced by $l_R$, then,  for every $D$ such that $\dom(L_R) < D< \dom(L_R^*)$ and for every $\lambda \in \C$  we get that the    $(L_R^*+\lambda\Id)|_D$  is a Fredholm operator.
\end{prop}

\begin{proof}
The proof of this result readily follows once we prove that 
\[
\sigma_{\mathrm{ess}}(L_R) = \emptyset.
\]
To do so, we apply Lemma~\ref{thm:ess_split} to the interval  $(0,1)=(0,\delta)\cup(\delta,1)$. We denote by $L_R^1$ (resp. $L_R^2$) the minimal operator corresponding to the differential operator $l_R^1$ (resp. $l_R^2$) obtained by restricting $l_R$ on $\mathscr C_0^\infty((0,\delta), \R^n)$ (resp. $\mathscr C_0^\infty((\delta,1), \R^n)$). Let \( \phi \in C^{\infty}_0((0,1), \R^n) \). By Hardy's inequality, we have
\[
\langle L_R\phi, \phi \rangle = \langle \dot\phi, \dot\phi \rangle + \left\langle \dfrac{R(t)}{t^2}\phi, \phi \right\rangle \ge \left\langle \dfrac{1/4 + R(t)}{t^2}\phi, \phi \right\rangle.
\]
We note that \( \min(\sigma(R(t)) \) is continuous on \( [0,1] \). So, there exists \( \delta > 0 \) such that \( \min(\sigma(R(t)) \ge -1/4 + c \) on \( [0,\delta] \), for some \( c > 0 \).
It follows that
\[
\langle L_R\phi, \phi \rangle \ge c\left\|\dfrac{\phi}{t}\right \|^2 \ge \dfrac{c}{\delta^2} \|\phi\|^2
\qquad \textrm{ for any } \qquad \phi \in \mathscr C^{\infty}_0((0,\delta), \R^n).
\]
Let \( L^i_{R,F} \) be the Friedrich extensions of \( L_R^i \) for $i=1,2$. 
It follows that
\[
\sigma_{\mathrm{ess}}(L^1_{R,F})\subset \sigma(L^1_{R,F}) \subset \left[\dfrac{c}{\delta^2}, +\infty\right).
\]
Since \( l_R \) is regular on \( [\delta,1] \), we have \( \sigma_{\mathrm{ess}}(L^2_{R,F}) = \emptyset \). Therefore,
\[
\sigma_{\mathrm{ess}}(L^1_{R,F} \oplus L^2_{R,F}) \subset \left[\dfrac{c}{\delta^2}, +\infty\right).
\]
By Lemma~\ref{thm:ess_split}, we conclude that
  \[
\sigma_{\mathrm{ess}}(L_R) \subset \left[\dfrac{c}{\delta^2},+\infty\right)\qquad  \forall\,0<\delta\le\delta_0.
\] 
Then $\sigma_{\mathrm{ess}}(L_R)=\emptyset$.   Here $c>0$ is fixed and the tail cut  $\delta$  can be chosen arbitrarily small. The Friedrichs realization therefore has compact resolvent. Every closed intermediate realization differs from it by finitely many domain directions (as proved above independently of Fredholmness); Fredholmness after each scalar shift follows by finite-dimensional domain restriction and extension .
\end{proof}
\begin{thm}\label{thm:Fredholmness-Bessel-perturbation}
    Let $b \in   \mathscr C^0([0,1],\Sym(n))$  and let $L_q$ be the minimal operator of the Bessel operator defined at Corollary~\ref{thm:corollary7-4}. Denoting by $L_b$ the operator pointwise defined by $b$, then  for any $q \in \R$ we get that any  self-adjoint extension of $L_q+L_b$ is a Fredholm operator.
\end{thm}

\begin{proof}
    We let $R(t)=q\,\Id+ t^2\, b(t)$. If 
    \begin{itemize}
        \item $q>-1/4$, by invoking Proposition~\ref{thm:Fredholmness}, we get that any self-adjoint extension of $L_q+L_b$ is Fredholm
        \item $q\leq -1/4$, then by invoking Corollary~\ref{thm:corollary7-4} (which works for $q<3/4$), we get that $L_q+L_b$ has a compact resolvent and in particular it is Fredholm.
    \end{itemize}
    This concludes the proof. 
\end{proof}
Let us now consider the Bessel-type operator on $C_0^\infty((1,+\infty),\R^n)$  defined by 
\[
l_R:=-\dfrac{d^2}{dt^2}+\dfrac{R(t)}{t^2} \qquad \textrm{ for } \qquad  t \in [1,+\infty).
\]
We observe that if $\lim_{t \to \infty}t^{-2}R(t)=0$, then $l_R$  is not Fredholm. Nevertheless, since the essential spectrum is contained in $[0,+\infty)$, it is possible to  study its Morse index.

 For the matrix inequalities below, $\liminf R>-\tfrac14 I$ means $\liminf\lambda_{\min}(R)>-\tfrac14$, while $\limsup R<-\tfrac14 I$ means $\limsup\lambda_{\max}(R)<-\tfrac14$. Thus the inequalities have a uniform strict margin on a sufficiently small or distant tail. The same convention applies to the introductory statements.

 \begin{thm}\label{thm:bessel-morse-index}
Let $l_R$ be the operator defined above on $\mathscr C_0^\infty((1,+\infty),\R^n)$. 
\begin{itemize}
\item{\bf Case 1.} If  $\liminf_{t\to +\infty} R(t)>-1/4 $, then any self-adjoint extension of it has finite Morse index.
\item{\bf Case 2.} If $\limsup_{t\to +\infty} R(t)<-1/4  $, then any self-adjoint extension of it has infinite Morse index.
\end{itemize}
\end{thm}
\begin{proof}
Let $M>1$. By  Using splitting method, let $l_1=l|_{C_0^\infty((1,M),\R^n)}$ and $l_2=l|_{C_0^\infty((M,+\infty),\R^n)}$ .
Let $L,L_1,L_2$ be the minimal operators corresponding to $l,l_1,l_2$, respectively.
Let  $L_D,L_{1,D},L_{2,D}$ be the self-adjoint operators associated to  $l,l_1,l_2$ with Dirichlet boundary conditions, respectively.
 Here one may choose separated self-adjoint realizations at infinity; in Case 1 choose the Friedrichs realization. The finite-domain comparison transfers the conclusion to every extension.  We observe that $L_D$ and $L_{1,D}\oplus L_{2,D}$ are both finite-dimensional self-ajoint extensions of $L_1\oplus L_2$. 

By Proposition~\ref{prop:estimate_morse_self_ajoint}, we have
  \[
\iMor(L_D)\le\iMor(L_{1,D}\oplus L_{2,D})+2n,\qquad\iMor(L_{1,D}\oplus L_{2,D})\le\iMor(L_D)+2n.
\] 
Moreover $\iMor(L_{1,D}\oplus L_{2,D})=\iMor(L_{1,D})+\iMor(L_{2,D})$ where $\iMor(L_{1,D})<\infty$. So, we only need to study $L_{2,D}$.
\paragraph{Case 1.} We assume that  $\liminf_{t\to +\infty} R(t)>-1/4 \Id$. Let $M>0$ such that $R(t)>-1/4 $ for $ t\ge M$.
Then, for each $u\in C_0^\infty((M,+\infty),\R^n)$, by Hardy's inequality, we have
\[
\langle l_2u,u\rangle>0
\]
and so $\iMor (L_{2,D})=0$. Then we get $\iMor(L_2)<\infty$ .

\paragraph{Case 2.} We assume that $\limsup_{t\to +\infty}R(t) <-1/4 \Id $.   Choose $q<-1/4$ and $M>1$ such that $R(t)\le qI$  for $t\ge M$.
We now prove that for $q<-1/4$, the operator $-\dfrac{d^2}{dt^2}+\dfrac{q}{t^2}$ on $W^{2,2}((1,+\infty), \R^n)$ with Dirichlet boundary condition has infinite Morse index. 

By Theorem~\ref{thm:limit_morse_index} and Remark \ref{rem:limit_morse_index}, we only need to show that a solution of the equation with Dirichlet boundary condition has infinitely many zeroes.
By Equation~\eqref{eq:sol_singular}, for $q<-1/4$, the SL-equation has the following solution
\[
\dfrac{1}{r}t^{1/2}\sin(r \ln t) \qquad \textrm{ for } \qquad q=-1/4-r^2 
\]
having infinitely many zeros   (replace $\ln t$ by $\ln(t/M)$ for the Dirichlet condition at $M$). The scalar negative subspaces can be approximated by compactly supported smooth functions in its form norm. Multiplying them by a fixed unit vector and using $R(t)\le qI$ transfers arbitrarily large negative subspaces to the matrix operator's minimal core, and hence to every self-adjoint extension.  This concludes the proof.
\end{proof}

\begin{thm}\label{thm:bessel-morse-index-on-0-1}
Let $l_R$ be the operator defined above on   $\mathscr C_0^\infty((0,1),\R^n)$ . 
\begin{itemize}
\item{\bf Case 1.} If  $\liminf_{t\to 0^+} R(t)>-1/4 $, then any self-adjoint extension  has finite Morse index.
\item{\bf Case 2.} If $\limsup_{t\to 0^+} R(t)<-1/4  $, then any self-adjoint extension has infinite Morse index.
\end{itemize}
\end{thm}
\begin{proof}
The proof of the Case 1 is similar to the corresponding of Theorem~\ref{thm:bessel-morse-index} and essentially based on splitting the interval $(0,1)$ into $(0,c]\cup[c,1]$ and then on using Hardy's inequality on $\mathscr C_0^\infty((0,c),\R^n)$.

  For  Case 2   choose $q<-1/4$ and $c>0$ so that $R(t)\le qI$ on  $(0,c)$  . The supercritical scalar operator is not semibounded, so Theorem~\ref{thm:limit_morse_index} cannot be used at this endpoint. Instead, for $f(t)=\sqrt t\,g(\log t)$ with $g\in C_0^\infty(\mathbb R)$, integration by parts gives
\[
\int\bigl(|f'|^2+qt^{-2}|f|^2\bigr)\,dt
=\int\bigl(|g'|^2+(q+\tfrac14)|g|^2\bigr)\,dx.
\]
Choose $g$ with negative right-hand side, for example by smoothing a Dirichlet sine on an interval of length greater than $\pi/\sqrt{-q-1/4}$. Disjoint translates towards $x=-\infty$ give arbitrarily large negative spaces supported in $(0,c)$. Multiplying by a fixed unit vector and comparing with $R$ gives negative spaces in the minimal core. They belong to every extension domain, proving infinite index.
 \end{proof}

\section{Asymptotic solutions of the N-body problem}
\label{sec:N-bp}

The goal of this section is to apply the preceding theory to
asymptotic solutions of the spatial Newtonian $n$-body problem.
We recover the strict-threshold finiteness and infiniteness
criteria for the Morse index considered in
\cite{BHPT20,HOY21}, using comparison with Bessel-type operators.
Related results for homothetic motions are obtained in
\cite{OP25}. We consider total collisions and completely
parabolic or hyperbolic motions.

\subsection{A description of the problem}

The Newtonian $n$-body problem describes the motion of $n\geq2$
point masses $m_i>0$ in $\mathbb R^3$. Let
\[
M=\operatorname{diag}(m_1I_3,\ldots,m_nI_3),
\]
where $I_3$ is the $3\times3$ identity matrix.
Writing $q=(q_1,\ldots,q_n)\in\mathbb R^{3n}$, Newton's equation is
\begin{equation}\label{eq:Newton-n-body}
M\ddot q=\nabla U(q),
\qquad
U(q)=\sum_{1\leq i<j\leq n}
\frac{m_im_j}{|q_i-q_j|}.
\end{equation}
Here $U$ is the negative potential energy and $\nabla$ denotes
the Euclidean gradient.

Working in the center-of-mass frame, we restrict configurations
to the $n^*=3(n-1)$ dimensional space
\[
\mathcal X=
\left\{
q\in\mathbb R^{3n}:
\sum_{i=1}^n m_iq_i=0
\right\}.
\]
The mass inner product and its norm are
\[
\langle x,y\rangle_M=\langle Mx,y\rangle,
\qquad
|x|_M^2=\langle Mx,x\rangle.
\]
The Lagrangian is
\[
L(q,v)=\frac12|v|_M^2+U(q),
\]
with singularities at the collision configurations
\[
\Delta=\bigcup_{1\leq i<j\leq n}\Delta_{ij},
\qquad
\Delta_{ij}=\{q\in\mathcal X:q_i=q_j\}.
\]
Set $\hat{\mathcal X}=\mathcal X\setminus\Delta$.

On a compact time interval $[a,b]$, the action
\[
\mathcal A(q)=\int_a^b L(q(t),\dot q(t))\,dt
\]
is $C^2$ on the open set of collision-free paths in
$W^{1,2}([a,b],\mathcal X)$. Its critical points with fixed
endpoints are classical solutions of
Equation~\eqref{eq:Newton-n-body}.
Although colliding paths may have finite Newtonian action,
we consider here solutions that are collision-free on their
time interval, allowing a total collision at its terminal
endpoint.

\begin{defn}
Let $q$ be a solution of Equation~\eqref{eq:Newton-n-body}
defined for all sufficiently large positive times.
We say that:
\begin{itemize}
\item
$q$ is a \textbf{completely parabolic motion} if
\[
|q_i(t)-q_j(t)|\longrightarrow\infty
\quad(i\neq j),
\qquad
\dot q_i(t)\longrightarrow0
\quad(1\leq i\leq n);
\]
\item
$q$ is a \textbf{completely hyperbolic motion} if all mutual
distances tend to infinity and the limits
$\lim_{t\to\infty}\dot q_i(t)$ exist and are pairwise distinct.
\end{itemize}
We call the corresponding end a parabolic or hyperbolic
infinity, respectively.
\end{defn}

\begin{defn}
A solution
$q\in C^2([0,T),\hat{\mathcal X})$, with $0<T\leq\infty$,
is called an \textbf{asymptotic solution} if:
\begin{itemize}
\item
$T<\infty$ and $q(t)\to0$ as $t\uparrow T$;
\item
or $T=\infty$ and $q$ has a parabolic or hyperbolic infinity.
\end{itemize}
\end{defn}

\begin{rem}
The analogous definitions for past ends follow by reversing
time. Partial collisions and mixed types of final motion
are not considered here.
\end{rem}

Linearizing Equation~\eqref{eq:Newton-n-body} along $q$ gives
\begin{equation}\label{eq:linearized}
M\ddot x(t)=B(t)x(t),
\qquad
B(t)=D^2U(q(t)),
\qquad x(t)\in\mathcal X.
\end{equation}
The ambient Hessian $D^2U(q)$ is a symmetric $3n\times3n$
matrix with $3\times3$ blocks
\begin{equation}\label{eq:Hessian-U}
\begin{aligned}
D_{ij}
&=\frac{m_im_j}{r_{ij}^3}
  \bigl(I_3-3u_{ij}u_{ij}^{T}\bigr),
&& i\neq j,
\\
D_{ii}
&=-\sum_{j\neq i}D_{ij},
&
r_{ij}&=|q_i-q_j|,
\qquad
u_{ij}=\frac{q_i-q_j}{r_{ij}}.
\end{aligned}
\end{equation}

Introduce the polar coordinates
\[
r=\sqrt{I(q)},
\qquad
s=\frac qr,
\qquad
I(q)=\langle Mq,q\rangle,
\]
and the space of normalized configurations
\[
\mathcal E=\{q\in\mathcal X:I(q)=1\}.
\]
A collision-free configuration $q_0$ is a
\textbf{central configuration} if
\[
\nabla U(q_0)+\mu Mq_0=0,
\qquad
\mu=\frac{U(q_0)}{I(q_0)}.
\]
Thus a normalized central configuration
$a\in\mathcal E\cap\hat{\mathcal X}$ satisfies
\[
\nabla U(a)+U(a)Ma=0.
\]
Equivalently, $a$ is a critical point of $U|_{\mathcal E}$.
At such a point,
\[
D^2(U|_{\mathcal E})(a)[\xi,\eta]
=
\left\langle
\bigl(D^2U(a)+U(a)M\bigr)\xi,\eta
\right\rangle,
\qquad \xi,\eta\in T_a\mathcal E.
\]
Its self-adjoint representative for the mass inner product is
\[
\left.
\bigl(M^{-1}D^2U(a)+U(a)I_{\mathcal X}\bigr)
\right|_{T_a\mathcal E}.
\]

For collision and parabolic ends, the results below assume
that $s(t)$ converges to a single collision-free normalized
central configuration. This convergence is an explicit
hypothesis.

By homogeneity, for every $\mu>0$,
\begin{equation}\label{eq:fava}
\frac{I(\mu q)}{U(\mu q)}D^2U(\mu q)
=
\frac{I(q)}{U(q)}D^2U(q).
\end{equation}

\begin{rem}
The blocks in Equation~\eqref{eq:Hessian-U} satisfy
$D_{ij}=D_{ji}$. Moreover, the three translation vectors
\[
\widetilde e_1=(e_1,\ldots,e_1),
\qquad
\widetilde e_2=(e_2,\ldots,e_2),
\qquad
\widetilde e_3=(e_3,\ldots,e_3)
\]
belong to the kernel of the ambient Hessian $D^2U(q)$.
These vectors do not belong to $\mathcal X$; translation
directions have already been removed by the center-of-mass
reduction.
\end{rem}

The conserved energy is
\[
H_0=\frac12|\dot q|_M^2-U(q).
\]
The following estimates determine the asymptotic coefficients
of the variational equation.

\begin{lem}\label{thm:asympt-estimates}
Let $q\in C^2([0,T),\hat{\mathcal X})$ be an asymptotic
solution.
\begin{itemize}
\item[(a)]
Suppose that $q$ has a total collision at $T<\infty$
or a parabolic end at $T=\infty$, and that
$s(t)\to a\in\mathcal E\cap\hat{\mathcal X}$, where $a$
is a normalized central configuration. Then
\[
\lim_{t\to T}\frac{r(t)}{\beta(t)^{2/3}}
=
\left(\frac92U(a)\right)^{1/3},
\qquad
\beta(t)=
\begin{cases}
T-t,&T<\infty,\\
t,&T=\infty.
\end{cases}
\]
\item[(b)]
If $q$ has a hyperbolic end, then $H_0>0$ and
\[
\lim_{t\to\infty}\frac{r(t)}t=\sqrt{2H_0},
\qquad
\lim_{t\to\infty}\dot r(t)=\sqrt{2H_0},
\qquad
\lim_{t\to\infty}r(t)|\dot s(t)|_M=0.
\]
\end{itemize}
\end{lem}

\begin{proof}
For a total collision or a parabolic end, the
Sundman--Sperling angular estimate gives
$r^3|\dot s|_M^2\to0$; see
\cite[Lemma~2.1 and the references therein]{HOY21}.
Since $|s|_M=1$ and $\langle s,\dot s\rangle_M=0$,
the energy identity becomes
\[
r\dot r^{\,2}+r^3|\dot s|_M^2
=
2H_0r+2U(s).
\]
At a total collision, $r\to0$. At a parabolic end,
the defining velocity and separation limits give $H_0=0$.
Consequently, in both cases,
\[
r\dot r^{\,2}\longrightarrow2U(a).
\]
The sign of $\dot r$ is eventually constant, negative at
a future collision and positive at a parabolic infinity.
Hence
\[
\left|\frac{d}{dt}r^{3/2}\right|
\longrightarrow\frac32\sqrt{2U(a)}.
\]
Integrating towards the collision time, or taking the
corresponding limit at infinity, proves (a).

For a hyperbolic end, write
$v_i=\lim_{t\to\infty}\dot q_i(t)$ and $v=(v_1,\ldots,v_n)$.
Then $q(t)/t\to v$ and
\[
|v|_M^2=2H_0>0.
\]
It follows that $r(t)/t\to\sqrt{2H_0}$ and
$s(t)\to s^+:=v/\sqrt{2H_0}$.
Moreover,
\[
\dot r=\frac{\langle q,\dot q\rangle_M}{r}
\longrightarrow\sqrt{2H_0},
\qquad
r\dot s=\dot q-\dot r\,s\longrightarrow0.
\]
This proves (b).
\end{proof}

\begin{rem}
Item~(a) is the radial Sundman--Sperling estimate.
The proof above derives its numerical coefficient directly
from the energy identity in our normalization.
\end{rem}

By homogeneity,
\[
D^2U(q(t))=r(t)^{-3}D^2U(s(t)).
\]
For a total collision or parabolic end with limiting
normalized central configuration $a$, set
\[
\bar B(a)
=
\left.
\frac{2}{9U(a)}M^{-1}D^2U(a)
\right|_{\mathcal X}.
\]
Lemma~\ref{thm:asympt-estimates} gives
\[
\beta(t)^2
\left.M^{-1}D^2U(q(t))\right|_{\mathcal X}
\longrightarrow\bar B(a).
\]
Thus the \textbf{asymptotic variational equation} is
\begin{equation}\label{eq:asymptotic-linearize}
\ddot x(t)=\frac1{\beta(t)^2}\bar B(a)x(t).
\end{equation}
This is the limiting model: the coefficient of the actual
variational equation is
$\beta(t)^{-2}\bigl(\bar B(a)+o(1)\bigr)$.

All inequalities involving $\bar B(a)$ are understood for
the mass inner product on $\mathcal X$. The condition
\begin{equation}\label{eq:BS-and-Bessel}
\bar B(a)>-\frac14I_{\mathcal X}
\end{equation}
is equivalent to
\[
\left.
\bigl(M^{-1}D^2U(a)+U(a)I_{\mathcal X}\bigr)
\right|_{T_a\mathcal E}
>
-\frac18U(a)I_{T_a\mathcal E}.
\]
Indeed, homogeneity gives
$M^{-1}D^2U(a)a=2U(a)a$, and
$\mathcal X=\mathbb Ra\oplus_M T_a\mathcal E$.
This is precisely the strict [BS]-condition in
\cite{BHPT20}.

\begin{rem}
As discussed in Remark~\ref{rem:important}, the constant
$-1/4$ is the Hardy semiboundedness threshold for the
inverse-square model.

For a hyperbolic end, the limiting configuration
$s^+\in\mathcal E\cap\hat{\mathcal X}$ need not be central.
In this case,
\[
t^3\left.M^{-1}D^2U(q(t))\right|_{\mathcal X}
\longrightarrow
B_h(s^+)
:=
\left.
(2H_0)^{-3/2}M^{-1}D^2U(s^+)
\right|_{\mathcal X}.
\]
The corresponding limiting operators are therefore
\[
h_{+,k}=
\begin{cases}
-\dfrac{d^2}{dt^2}
+\dfrac1{\beta(t)^2}\bar B(a),
& k=2
\quad\text{(total collision or parabolic infinity)},\\[8pt]
-\dfrac{d^2}{dt^2}
+\dfrac1{t^3}B_h(s^+),
& k=3
\quad\text{(hyperbolic infinity)}.
\end{cases}
\]
\end{rem}

The Morse index $\iMor(q)$ is the supremum of the dimensions
of subspaces on which the second variation
\[
Q_q[\xi]
=
\int_0^T
\left(
|\dot\xi(t)|_M^2
+
\langle D^2U(q(t))\xi(t),\xi(t)\rangle
\right)\,dt,
\qquad
\xi\in C_c^\infty((0,T),\mathcal X),
\]
is negative definite. In the semibounded case, this is the
index of its Friedrichs realization, with Dirichlet
condition at $0$.

To specify the fundamental solution, use the mass-normalized
coordinates
\[
z=M^{1/2}x,
\qquad
\mathcal Z=M^{1/2}\mathcal X,
\qquad
V(t)=
\left.
M^{-1/2}D^2U(q(t))M^{-1/2}
\right|_{\mathcal Z}.
\]
The space $\mathcal Z$ has dimension $n^*=3(n-1)$.
After choosing a Euclidean orthonormal basis of $\mathcal Z$,
let $\gamma$ be the fundamental solution of
\[
\frac{d}{dt}
\begin{pmatrix}p\\z\end{pmatrix}
=
\begin{pmatrix}0&V(t)\\I_{n^*}&0\end{pmatrix}
\begin{pmatrix}p\\z\end{pmatrix},
\qquad
\gamma(0)=I_{2n^*},
\]
and set
$\Lambda_D=\mathbb R^{n^*}\oplus\{0\}$.

We obtain the following Morse-index criterion.

\begin{thm}\label{thm:Morse-asymptotic}
Let $q\in C^2([0,T),\hat{\mathcal X})$ have a total
collision at $T<\infty$ or a parabolic end at $T=\infty$.
Assume that $s(t)$ converges to a collision-free normalized
central configuration $a$, and let $\bar B(a)$ be as above.

If
\[
\bar B(a)>-\frac14I_{\mathcal X},
\]
then $\iMor(q)$ is finite and
\[
\iMor(q)
=
\sum_{t\in(0,T)}
\dim\bigl(\gamma(t)\Lambda_D\cap\Lambda_D\bigr).
\]
If
\[
\lambda_{\min}(\bar B(a))<-\frac14,
\]
then $\iMor(q)=\infty$.

If $q\in C^2([0,\infty),\hat{\mathcal X})$ has a
hyperbolic end, its Morse index is finite.
\end{thm}

\begin{proof}
In mass-normalized coordinates,
\[
Q_q[\xi]
=
\int_0^T
\bigl(|\dot z|^2+\langle V(t)z,z\rangle\bigr)\,dt,
\qquad z=M^{1/2}\xi.
\]
Homogeneity and Lemma~\ref{thm:asympt-estimates} imply
\[
\beta(t)^2V(t)\longrightarrow\widehat B(a),
\qquad
\widehat B(a)=M^{1/2}\bar B(a)M^{-1/2}
\quad\text{on }\mathcal Z.
\]
The matrix $\widehat B(a)$ is Euclidean self-adjoint and
has the same eigenvalues as $\bar B(a)$.

Suppose first that
$\widehat B(a)>-\frac14I_{\mathcal Z}$.
Choose $\delta>0$ sufficiently small that, near the
singular end,
\[
V(t)\geq
\left(-\frac14+\delta\right)\beta(t)^{-2}I_{\mathcal Z}.
\]
Hardy's inequality makes the Dirichlet form on this tail
nonnegative. Imposing a Dirichlet condition at the cut
has finite codimension, while the remaining compact
regular interval has finite Morse index. Hence the full
index is finite.

The form is also semibounded: at a collision this follows
from the same Hardy estimate and a smooth localization,
whose error is bounded in $L^2$; at a parabolic end
the potential is bounded. Theorem~\ref{thm:limit_morse_index}
therefore gives
\[
\iMor(q)
=
\sum_{t\in(0,T)}
\dim\bigl(\gamma(t)\Lambda_D\cap\Lambda_D\bigr).
\]

Suppose next that $\widehat B(a)$ has an eigenvalue
$\lambda<-1/4$. Choose a unit eigenvector $v$ and
$\eta\in(\lambda,-1/4)$. Sufficiently near the end,
\[
\langle V(t)v,v\rangle\leq\eta\,\beta(t)^{-2}.
\]
For a variation $z(t)=f(\beta(t))v$, the quadratic form
is consequently bounded above by the scalar form with
coefficient $\eta$. Setting
$\tau=\log\beta$ and
$f(\beta)=\sqrt{\beta}\,g(\tau)$ gives, for compactly
supported smooth $g$,
\[
\int
\left(|f'|^2+\eta\beta^{-2}|f|^2\right)\,d\beta
=
\int
\left(|g'|^2+\left(\eta+\frac14\right)|g|^2\right)\,d\tau.
\]
Since $\eta+1/4<0$, a smooth test function supported
on a sufficiently long interval makes the last integral
negative. Translating this function to disjoint logarithmic
intervals towards the singular end produces arbitrarily
many variations with disjoint supports and negative
quadratic form. Their span is negative definite, proving
that $\iMor(q)=\infty$.

Finally, at a hyperbolic end,
Lemma~\ref{thm:asympt-estimates} and homogeneity give
\[
V(t)=O(t^{-3}),
\qquad
t^2V(t)\longrightarrow0.
\]
The preceding Hardy comparison and compact-interval
argument therefore imply finite Morse index.
\end{proof}

\begin{rem}
The radial identity
\[
\bar B(a)a=\frac49a
\]
shows why the infinite-index criterion concerns the least
eigenvalue rather than the inequality
$\bar B(a)<-\frac14I_{\mathcal X}$.
No conclusion is asserted here in the borderline case
$\lambda_{\min}(\bar B(a))=-1/4$.

The finite or infinite character of the index persists
under self-adjoint boundary changes of finite boundary
codimension. The displayed conjugate-point formula is
the Friedrichs/Dirichlet formula. A triple-index correction
requires the hypotheses of
Theorem~\ref{thm:triple-index-Morse-Friedrich-extensions};
in particular, Fredholmness at zero is not assumed here
for parabolic or hyperbolic ends.
\end{rem}

\appendix

\section{A Concise Review of the Maslov-Type Indices}\label{sec:Maslov}

The purpose of this section is to present the functional analytic and symplectic preliminaries underlying the {\sc spectral flow} and the {\sc Maslov index}. Our primary references are \cite{CLM94,Dui76,RS93,LZ00}, from which we adopt some notation and definitions.

 We denote by $(V, \omega)$ be the standard symplectic space and  
let $\mathscr P([a,b]; V)$ the space of continuous maps 
\[
f: J \to \Set{\textrm{pairs of Lagrangian subspaces in } V}
\]
equipped with the compact-open topology and we recall the following definition. 
\begin{defn}\label{def:Maslov-index}
The {\em CLM-index\/} is the unique integer valued function  
\[
 \iCLM: \mathscr P((a,b) ; \R^{2n}) \to \Z
\]
which satisfies  the following properties: 
\begin{multicols}{2}
\begin{itemize}
\item[(I)] {\sc  Affine Scale Invariance}
\item[(II)] {\sc Deformation Invariance relative to the Endpoints}
\item[(III)] {\sc Path Additivity}
\item[(IV)] {\sc Symplectic Additivity}
\item[(V)] {\sc Symplectic Invariance}
\item[(VI)]{\sc Normalization.}
\end{itemize}
\end{multicols}
\end{defn}
We refer the reader to  \cite[Theorem 1.1]{CLM94}. Following authors in \cite[Section 3]{LZ00} and references therein, let us now introduce the notion of crossing form that gives an efficient  way for computing the intersection indices   in the Lagrangian Grassmannian context.  

Let $\ell$ be a $\mathscr C^1$-curve of Lagrangian subspaces 
such that 
$\ell(0)= L$ and $\dot \ell(0)=\widehat L$. Now, if  $W$ is a fixed Lagrangian subspace transversal to $L$. For  $v \in L$ and  small enough $t$, let $w(t) \in W$ be such that $v+w(t) \in \ell(t)$.  Then the  form 
\begin{equation}\label{eq:forma-Q}
 Q(L, \widehat L)[v]= \dfrac{d}{dt}\Big\vert_{t=0} \omega \big(v, w(t)\big)
\end{equation}
is independent on the choice of $W$. 
\begin{defn}\label{def:crossing-form}
Let $t \mapsto \ell(t)=(\ell_1(t), \ell_2(t))$ be a map in  	 $\mathscr P([a,b]; V)$. For $t \in [a,b]$, the crossing form is a quadratic form defined by 
\begin{equation}\label{eq:crossings}
\Gamma(\ell_1, \ell_2, t)= Q(\ell_1(t), \dot \ell_1(t))- 	Q(\ell_2(t), \dot \ell_2(t))\Big\vert_{\ell_1(t)\cap \ell_2(t)}
\end{equation}
 A {\em crossing instant\/} for the curve $t \mapsto \ell(t)$ is an instant $t \in [a,b]$  such that $\ell_1(t)\cap \ell_2(t)\neq \{0\}$ non-trivially. A crossing is termed {\em regular\/} if the $\Gamma(\ell_1, \ell_2, t)$ is nondegenerate. 
 \end{defn}
 We observe that  if $t$ is a crossing instant, then $
 \Gamma(\ell_1, \ell_2,t)= - \Gamma (\ell_2, \ell_1, t).$
 If  $\ell$ is {\em regular\/} meaning that  it has only regular crossings, then the $\iCLM$-index can be computed through the crossing forms, as follows 
\begin{equation}\label{eq:iclm-crossings}
 \iCLM\big(\ell_1(t), \ell_2(t), \Omega,  t \in [a,b]\big) = \coiMor\big(\Gamma(\ell_2, \ell_1, a)\big)+ 
\sum_{a<t<b} 
 \sgn\big(\Gamma(\ell_2, \ell_1, t)\big)- \iMor\big(\Gamma(\ell_2, \ell_1,b)\big)
\end{equation}
where the summation runs over all crossings $t \in (a,b)$ and $\coiMor, \iMor$  are the dimensions  of  the positive and negative spectral spaces, respectively and $\sgn\= 
\coiMor-\iMor$ is the  signature. 
(We refer the interested reader to \cite{LZ00} ). 

Let $L_0$ be a distinguished Lagrangian and we assume that $\ell_1(t)\equiv L_0$ for every $t \in J$. In this case we get that the crossing form at the instant $t$ provided in Equation~\eqref{eq:crossings} actually reduce to 
\begin{equation}\label{eq:forma-crossing}
 \Gamma\big(\ell_2(t), L_0, t \big)= Q|_{\ell_2(t)\cap L_0}
\end{equation}
and hence   
\begin{equation}\label{eq:iclm-crossings-2}
 \iCLM\big(L_0, \ell_2(t), \Omega,  t \in [a,b]\big) = \coiMor\big(\Gamma(\ell_2, L_0, a)\big)+ 
\sum_{a<t<b} 
 \sgn\big(\Gamma(\ell_2, L_0, t)\big)- \iMor\big(\Gamma(\ell_2, L_0,b)\big)
\end{equation}

\subsection{On the Triple and H\"{o}rmander Index}

A crucial ingredient that, in some sense, measures the difference in the relative Maslov index with respect to two different Lagrangian subspaces is the {\sc H\"{o}rmander index}. This index is also related to the difference of the triple index and to an interesting generalization recently developed by the last author and his collaborators in \cite{ZWZ18}. 

We begin with the following definition of the {\sc Hörmander index}.

\begin{defn}\label{def:hormander}(\cite[Definition 3.9]{ZWZ18})
Let $\lambda, \mu \in \mathscr C^0\big((a,b) , \Lag(V,\omega)\big)$ such that 
\[
\lambda(a)=\lambda_1, \quad \lambda(b)=\lambda_2 \quad  \textrm{ and } \quad \mu(a)=\mu_1, \quad \mu(b)= \mu_2.
\]
Then the {\em H\"ormander index\/} is the integer given by 
\begin{multline}
s(\lambda_1, \lambda_2; \mu_1, \mu_2)
\= 
\iCLM\big(\mu_2, \lambda(t); t \in J\big) - 
\iCLM\big(\mu_1, \lambda(t); t \in J\big) \\
=
\iCLM\big(\mu(t), \lambda_2; t \in J\big)- \iCLM\big(\mu(t), \lambda_1; t \in J\big).
\end{multline}
Compare \cite[Equation (17), pag. 736]{ZWZ18} once observing that we observe that $\iCLM(\lambda,\mu)$ corresponds to $\textrm{Mas}\{\mu,\lambda\}$ in the notation of \cite{ZWZ18}. 
\end{defn}
\paragraph{Properties of the H\"ormander index.}
We briefly recall some well-useful properties of the H\"ormander index.
\begin{itemize}
\item 	$s(\lambda_1,\lambda_2; \mu_1, \mu_2) = -s(\lambda_1,\lambda_2; \mu_2, \mu_1)$
\item $s(\lambda_1,\lambda_2; \mu_1, \mu_2) = 
-s(\mu_1, \mu_2;\lambda_1,\lambda_2) + 
\sum_{j,k \in \{1,2\}}(-1)^{j+k+1}\dim (\lambda_j \cap \mu_k)$.
\item If $\lambda_j\cap \mu_k =\{0\}$ then $s(\lambda_1,\lambda_2; \mu_1, \mu_2) = 
-s(\mu_1, \mu_2;\lambda_1,\lambda_2)$.
\end{itemize}

The H\"ormander index is computable as the difference of the two Maslov  indices each one involving  three different Lagrangian subspaces. This index is defined in terms of the local chart representation of the atlas of the Lagrangian Grassmannian manifold.
\begin{defn}\label{def:Q-Dui76}
$\alpha,\beta,\gamma \in \Lag(V,\omega)$ and we assume that $\alpha \cap \beta=\gamma \cap \beta=(0)$. Then we define the quadratic form $Q(\alpha,\beta;\gamma): \alpha\to \R$ as follows
  \[
Q(\alpha,\beta;\gamma)[u]=\omega(u,Cu) \qquad \textrm{ where } \qquad C:\alpha \to \beta \quad\textrm{ and }\quad \gamma=\Set{u+C\,u| u\in \alpha}.
\] 
\end{defn}
\begin{defn}\label{def:kashi}
Let $\alpha,\beta,\gamma \in \Lag(V,\omega)$, $\varepsilon \= \alpha \cap \beta + \beta \cap \gamma$ and let $\pi\=\pi_\varepsilon $ be the projection in the symplectic reduction of $V$ mod $\varepsilon $.   
  We term {\em triple index\/} the integer defined by
\begin{multline}\label{eq:triple}
\itriple(\alpha, \beta, \gamma)\= \coindex \big[Q(\pi \alpha, \pi \beta; \pi \gamma)\big]	+\dim (\alpha \cap \gamma) -\dim (\alpha\cap \beta \cap \gamma)\\[3pt]
\leq n-\dim (\alpha \cap \beta)-\dim( \beta \cap \gamma) + \dim (\alpha \cap \beta \cap \gamma).
\end{multline}
\end{defn}
\begin{rem}\label{rem:molto-utile-stima}
Definition \ref{def:kashi} is well-posed and we refer the interested reader to \cite[Lemma 2.4]{Dui76} and  \cite[Corollary 3.12 \& Lemma 3.13]{ZWZ18} for further details).  It is worth noticing that $Q(\pi \alpha, \pi \beta; \pi \gamma)$ is a quadratic form on $\pi\alpha$. Being the reduced space $V_\varepsilon $  a $2(n-\dim \varepsilon )$ dimensional subspace, it follows that inertial indices of  $Q(\pi \alpha, \pi \beta; \pi \gamma)$ are integers between $\{0, \ldots,n-\dim \varepsilon \}$.
\end{rem}
\begin{rem}
		We also observe that	for arbitrary Lagrangian subspaces $\alpha,\beta,\gamma$, the quadratic form  $Q(\alpha,\beta,\gamma)$ is well-defined and it is a quadratic form on $\alpha\cap(\beta+\gamma)$.
	 Explicitly, if $z=p+q$ and $z'=p'+q'$ with $p,p'\in\beta$, $q,q'\in\gamma$ and $z,z'\in\alpha$, set $Q(z,z')=\omega(p,q')$. Changing the decomposition adds a vector in $\beta\cap\gamma$ and does not change the value; isotropy of $\alpha,\beta,\gamma$ gives symmetry. In the transverse graph case $z=-Cz+(z+Cz)$, hence $Q(z,z)=\omega(z,Cz)$ as above.
 Furthermore, we have 
    \[
    \coindex \big[Q(\alpha,\beta,\gamma)\big]=\coindex \big[Q(\pi\alpha,\pi\beta,\pi\gamma)\big].
    \]
    So, we can also define the triple index as 
	\[
	\itriple(\alpha, \beta, \gamma)\= \coindex \big[Q( \alpha,  \beta;  \gamma)	\big]+\dim (\alpha \cap \gamma) -\dim (\alpha\cap \beta \cap \gamma).
	\]
Authors in \cite[Lemma 3.2]{ZWZ18}  give a useful property for calculating such a quadratic form.
\begin{equation}\label{eq:invariance_Q}
\coindex \big[Q(\alpha,\beta,\gamma)\big]=\coindex \big[Q(\beta,\gamma,\alpha)\big]= \coindex\big[ Q(\gamma,\alpha,\beta)].
\end{equation}
\end{rem}
  If  $(\alpha,\beta)$ is a Lagrangian decomposition   and $\beta\cap\gamma=0$,  then $\pi$   is  the identity and   $\alpha\cap\beta\cap\gamma=0$, but $\alpha\cap\gamma$ need not vanish. Thus  the triple index is   the coindex of the graph form plus its nullity . It is possible to prove (cf.  \cite[proof of the Lemma 3.13]{ZWZ18}) that 
\begin{equation}\label{eq:kernel-dim-q-form}
	\dim(\alpha \cap \gamma) -\dim(\alpha \cap \beta \cap \gamma)= \nullity \big[Q(\pi \alpha, \pi \beta; \pi \gamma)\big],
\end{equation}
where we denoted by $\nullity Q$ the nullity (namely the kernel dimension of the quadratic form $Q$).
By summing up Equation \eqref{eq:triple} and Equation \eqref{eq:kernel-dim-q-form}, we finally get 
\begin{equation}\label{eq:triple-coindex-extended}
\itriple(\alpha, \beta, \gamma)= \noo{+}\big[Q(\pi \alpha, \pi \beta; \pi \gamma)\big]
\end{equation}
 where $\noo{+} Q$ denotes the so-called {\em extended coindex\/} or {\em generalized coindex\/} (namely the coindex plus the nullity) of the quadratic form $Q$. (Cf.  \cite[Lemma 2.4]{Dui76} for further details).
\begin{lem}\label{thm:properties}
Let $\lambda \in \mathscr C^1\big((a,b) , \Lag(V,\omega)\big)$. Then, for every $\mu \in \Lag(V, \omega)$, we have 
\begin{enumerate}
	\item[\textrm{ \bf{(I)} }] $s\big(\lambda(a), \lambda(b); \lambda(a), \mu \big)= -\itriple\big(\lambda(b), \lambda(a),\mu\big)\leq 0$, 
	\item[\textrm{ \bf{(II)} }] $s\big(\lambda(a), \lambda(b); \lambda(b), \mu \big)= \itriple\big(\lambda(a), \lambda(b),\mu\big)\geq 0$.
\end{enumerate}
\end{lem}
\begin{proof}
	For the proof, we refer the interested reader to \cite[Corollary 3.16]{ZWZ18}.
\end{proof}
The next result, which is the main result of \cite{ZWZ18}, allows us to reduce the computation of the H\"ormander index to the computation of the triple index. 
\begin{prop}{\sc \cite[Theorem 1.1]{ZWZ18}\/}\label{thm:mainli} 
Let $(V,\omega)$ be a $2n$-dimensional symplectic space and let  $\lambda_1, \lambda_2, \mu_1, \mu_2 \in \Lag(V,\omega)$. Under the above notation, we get 
\begin{equation}
s(\lambda_1, \lambda_2,\mu_1,\mu_2)=\itriple(\lambda_1,\lambda_2,\mu_2)- 	\itriple(\lambda_1,\lambda_2,\mu_1)\\
=\itriple(\lambda_1,\mu_1,\mu_2)- \itriple(\lambda_2,\mu_1,\mu_2)
\end{equation}
\end{prop}
\begin{rem}
We emphasize that no transversality conditions are assumed on the  Lagrangian subspaces in Proposition \ref{thm:mainli} 
\end{rem}

\begin{lem} \label{lem:coisotropic_Lagrangian}
	Let $(V, \omega)$ be a finite dimensional symplectic space and let $W$ be a coisotropic subspace. Then $W$ is a	
	Lagrangian subspace if and only if it contains a unique Lagrangian subspace of $V$.
\end{lem}
\begin{proof}
	Assume that $\Lambda\subsetneq W$ is a Lagrangian subspace of $V$.
	Let $v\in W\setminus \Lambda$. Let $l$ be the one dimension subspace spanned by $v$.
	Then we have
	\[
	((l^\omega\cap \Lambda) + l)^\omega=(l^{\omega\omega}+\Lambda^\omega)\cap l^\omega =(l+\Lambda)\cap l^\omega.
	\]
	Since one dimension subspace of $V$ is isotropic, we have $l< l^\omega$.
	It follows that 
	\[
	((l^\omega\cap \Lambda) + l)^\omega=(l+\Lambda)\cap l^\omega =l+\Lambda\cap l^\omega.
	\]
	So,  $W$ contains a subspace $(l^\omega\cap \Lambda)+l$ which is different with $\Lambda$ and it is a Lagrangian subspace of $V$.
	So if $\Lambda< W$ is the unique Lagrangian subspace contained in $W$,  we have $W=\Lambda$ and $W$ is a Lagrangian subspace of $V$.

\end{proof}

\begin{lem}\label{lem:triple_diff_circle_permu}
Let $\alpha,\beta,\gamma \in \Lag(V,\omega)$. Then the following identities hold:
\[
\itriple(\alpha,\beta,\gamma )-\itriple(\beta,\gamma ,\alpha)=\dim(\alpha\cap \gamma ) -\dim(\beta\cap \alpha)\]
or equivalently 
\[
\itriple(\alpha,\beta,\gamma )-\dim (\alpha\cap \gamma )=\itriple(\beta,\gamma ,\alpha) -\dim(\beta\cap \alpha)
\]
\end{lem}
\begin{proof}
By the definition of triple index
\begin{align*}
\itriple(\alpha,\beta,\gamma )=\coiMor \big[  Q(\alpha,\beta;\gamma )\big]+\dim(\alpha\cap \gamma )-\dim(\alpha\cap \beta\cap \gamma )\\
\itriple(\beta,\gamma ,\alpha)=\coiMor \big[  Q(\beta,\gamma ;\alpha)\big]+\dim(\alpha\cap \beta)-\dim(\alpha\cap \beta\cap \gamma ).
\end{align*}
Since $\coiMor[Q]$ is invariant under circular permutation, the lemma follows.
\end{proof}

\section{A Quick Overview of  the  Spectral Flow}\label{sec:sf}

The aim of this section is to briefly recall the Definition and the main
properties of the spectral
flow for a continuous path of closed self-ad[a,b]oint  Fredholm operators  $\CFsa(H )$ on the Hilbert space $H$.
Our basic reference is \cite{HP17}  and references therein.

Let $H $ be a separable complex Hilbert space and let
$A: D (A) < H  \to H $ be  a  self-ad[a,b]oint 
Fredholm
operator. By the Spectral decomposition Theorem (cf., for instance,
\cite[Chapter III,
Theorem 6.17]{Kat80}), there is an orthogonal decomposition $
 H = E_-(A)\oplus E_0(A) \oplus E_+(A),$
that reduces the operator $A$
and has the property that
\[
 \sigma(A) \cap (-\infty,0)=\sigma\big(A_{E_-(A)}\big), \quad
 \sigma(A) \cap \{0\}=\sigma\big(A_{E_0(A)}\big),\quad
 \sigma(A) \cap (0,+\infty)=\sigma\big(A_{E_+(A)}\big).
\]
\begin{defn}\label{def:essential}
Let $A \in \CFsa(H)$. We  term $A$ {\em essentially
positive\/}
if $\sigma_{ess}(A)\subset (0,+\infty)$, {\em essentially negative\/} if
$\sigma_{ess}(A)\subset (-\infty,0)$ and finally
{\em strongly indefinite\/} respectively if $\sigma_{ess}(A) \cap (-\infty,
0)\not=
\emptyset$ and $\sigma_{ess}(A) \cap ( 0,+\infty)\not=\emptyset$.
\end{defn}
\noindent
If $\dim E_-(A)<\infty$,
we define its {\em Morse index\/}
as the integer denoted by $\iindex{A}$ and defined as $
 \iindex{A} \= \dim E_-(A).$
Given $A \in\cfsa(H )$, for  $a,b \notin
\sigma(A)$ we set
\[
\mathcal P_{[a,b]}(A)\=\Real\left(\dfrac{1}{2\pi\, i}\int_\gamma
(\lambda-A)^{-1} d\, \lambda\right)
\]
where $\gamma$ is the circle of radius $\frac{b-a}{2}$ around the point
$\frac{a+b}{2}$. We recall that if
$K\subset \sigma(A)$ consists of  isolated eigenvalues of finite type then
  $
 \Imm \mathcal P_{K}(A)= E_{K}(A)\= \bigoplus_{\lambda \in K}\ker
(\lambda -A);
$
 (cf. \cite[Section XV.2]{GGK90}, for instance) and $0$ either belongs in the
resolvent set of $A$ or it is an isolated eigenvalue of finite multiplicity.
 The next result allows us to  define the spectral flow for
gap
continuous paths in  $\cfsa(H )$.
\begin{prop}\label{thm:cor2.3}
 Let $A_0 \in \cfsa(H )$ be fixed.
 \begin{enumerate}
  \item[(i)] There exists a positive real number $a \notin \sigma(A_0)$ and an
open neighborhood $\mathscr N \subset  \cfsa(H )$ of $A_0$ in the gap
topology such that $\pm a \notin
\sigma(A)$ for all $A \in  \mathscr N$ and the map
 \[
  \mathscr N \ni A \longmapsto \mathcal P_{[-a,a]}(A) \in \Bou(H )
 \]
is continuous and the pro[a,b]ection $\mathcal P_{[-a,a]}(A)$ has constant finite
rank for all $t \in \mathscr N$.
 \item[(ii)] If $\mathscr N$ is a neighborhood as in (i) and $-a \leq c \leq d
\leq a$ are such that $c,d \notin
 \sigma(A)$ for all $A \in \mathscr N$, then $A \mapsto \mathcal P_{[c,d]}(A)$
is
continuous on $\mathscr N$.
 Moreover the rank of $\mathcal P_{[c,d]}(A) \in \mathscr N$ is finite and
constant.
 \end{enumerate}
\end{prop}
\begin{proof}
For the proof of this result we refer the interested reader to
\cite[Proposition 2.10]{BLP05}.
\end{proof}
Let $ A :[a,b] \to \cfsa(H )$ be a gap-continuous  path.  As
consequence
of Proposition \ref{thm:cor2.3}, for every $t \in [a,b]$ there exists $a>0$ and
an open
connected neighborhood $\mathscr N_{t,a} \subset \cfsa(H )$ of
$\mathcal
A(t)$
such that $\pm a \notin \sigma(A)$ for all $A\in \mathscr N_{t,a}$ and the map
$\mathscr N_{t,a} \in A \longmapsto \mathcal P_{[-a,a]}(A) \in \mathcal
B$
is continuous and hence $ \rk\left(\mathcal P_{[-a,a]}(A)\right)$ does not
depends on $A \in \mathscr N_{t,a}$. Let us consider the open covering
of the interval $[a,b]$ given by the
pre-images of the neighborhoods $\mathcal
N_{t,a}$ through $ A $ and, by choosing a sufficiently fine partition of
the interval $[a,b]$ having diameter less than the Lebesgue number
of the covering, we can find  $a=:t_0 < t_1 < \dots < t_n:=b$,
operators $T_i \in \cfsa(H )$ and
positive real numbers $a_i $, $i=1, \dots , n$ in such a way the restriction of
the path $ A $ on the
interval $[t_{i-1}, t_i]$ lies in the neighborhood $\mathscr N_{t_i, a_i}$ and
hence the
$\dim E_{[-a_i, a_i]}( A _t)$ is constant for $t \in [t_{i-1},t_i]$,
$i=1, \dots,n$.
\begin{defn}\label{def:spectral-flow-unb}
The \emph{spectral flow of $ A $} (on the interval $[a,b]$) is defined by
  \[
 \spfl(A,[a,b])\=\sum_{i=1}^n\big(\dim E_{[0,a_i]}(A_{t_i})-\dim E_{[0,a_i]}(A_{t_{i-1}})\big)\in\Z.
\] 
\end{defn}
The spectral flow as given in Definition \ref{def:spectral-flow-unb} is
well-defined
(in the sense that it is independent either on the partition or on the $a_i$)
and only depends on
the continuous path $ A $. Here We list one of the useful properties of the spectral flow.
\begin{itemize}
 \item[]  {\sc (Path Additivity)\/} If $ A _1,\mathcal
A_2: [a,b] \to
 \cfsa(H )$ are two continuous path such that
$ A _1(b)= A _2(a)$, then
 $
  \spfl( A _1 * A _2) = \spfl( A _1)+\spfl( A _2).
 $
\end{itemize}
 In the relative-compact-perturbation discussion below, an $A$-compact symmetric perturbation is a compact map from $W=\dom(A)$, with its graph norm, to $H$, symmetric on $W$. Continuity of perturbations means operator-norm continuity in $\Bou(W,H)$. Their sums with $A$ are taken on $W$.

 As already observed, the spectral flow, in general,  depends on the whole path
and not
[a,b]ust on the ends. However, if the path has a special form, it actually depends
on the
end-points. More precisely, let  $ A  , B \in \cfsa(H )$
and let $\widetilde{ A }:[a,b] \to \cfsa(H )$ be the path
pointwise defined by $\widetilde{ A }(t)\= A + \widetilde{\mathcal
B}(t)$  where $
\widetilde{ B }$ is any continuous curve of $ A $-compact
operators parametrised on $[a,b]$
such that  $\widetilde{ B }(a)\=0$ and  $ \widetilde{ B }(b)\=
 B $. In this case,
the spectral flow depends of the
path $\widetilde A$, only on the endpoints (cf.  \cite{ZL99} and reference
therein).
\begin{rem}
  With the stated graph-norm topology, $\widetilde A$ is continuous into $\Bou(W,H)$. A symmetric $A$-compact perturbation is infinitesimally $A$-bounded, so its sum with $A$ is self-adjoint on $W$; Fredholmness is preserved by relative compactness. Affine homotopies between such perturbation paths stay in this class, which gives the stated endpoint dependence .
\end{rem}
\begin{defn}\label{def:rel-morse-index}(\cite[Definition 2.8]{ZL99}).
 Let $ A  , B \in \cfsa(H )$ and we assume that $ 
B$ is
 $ A $-compact (in the sense specified above). Then the
{\em  relative Morse index of the pair $ A $, $ A + B $\/}
is
defined by $
  \irel( A ,  A + B )=-\spfl(\widetilde{ A };[a,b])$
where $\widetilde{ A }\= A + \widetilde{ B }(t)$ and where
$
\widetilde{ B }$ is any continuous curve parametrised on $[a,b]$
of $ A $-compact operators such that
$\widetilde{ B }(a)\=0$ and
$ \widetilde{ B }(b)\=  B $.
\end{defn}
\noindent
In the special case in which the Morse index of both operators $ A $ and
$ A + B $ are
finite, then
\begin{equation}\label{eq:miserve}
\irel( A ,  A + B )=\iMor(A  +
B)-\iMor(A).
\end{equation}

Let  $W , H $ be separable Hilbert spaces with a dense and
continuous
inclusion $W  \hookrightarrow H $ and let
$ A :[a,b] \to \cfsa(H )$  having fixed domain $W $. We
assume that $ A $ is
a continuously differentiable path   $A:[a,b]\to\Bou(W,H)$, where $W$ is equipped with a fixed graph norm
 and
we denote by $\dot{ A }_{\lambda_0}$ the derivative of
$ A _\lambda$ with respect to the parameter $\lambda \in [a,b]$ at
$\lambda_0$.
\begin{defn}\label{def:crossing-point}
 An instant $\lambda_0 \in [a,b]$ is called a {\em crossing instant\/} if
$\ker\,  A _{\lambda_0} \neq 0$. The
 crossing form at $\lambda_0$ is the quadratic form defined by
\begin{equation}
 \Gamma( A , \lambda_0): \ker  A _{\lambda_0} \to \R, \quad
\Gamma( A , \lambda_0)[u] =
\langle \dot{ A }_{\lambda_0}\, u, u\rangle_H .
\end{equation}
Moreover a  crossing $\lambda_0$ is called {\em regular\/}, if $\Gamma(\mathcal
A, \lambda_0)$ is nondegenerate.
\end{defn}
We recall that there exists $\varepsilon >0$ such that   $ A  +\delta \,
\Id_H $ has only regular crossings
  for almost every $\delta \in (-\varepsilon, \varepsilon)$. In the special case in which all crossings are regular, then the spectral flow
can be easily computed through  the
crossing forms. More precisely the following result  holds.
\begin{prop}\label{thm:spectral-flow-formula}
 If $ A :[a,b] \to \cfsa(W , H )$ has only regular
crossings then they are in a finite
 number and
 \[
  \spfl( A , [a,b]) = -\mathrm{n_-}{\left[\Gamma( A ,a)\right]}+
\sum_{t_0 \in (a,b)}
  \sgn\left[\Gamma( A , t_0)\right]+
  \coiindex{\Gamma( A ,b)}
 \]
where the sum runs over all the crossing instants.
\end{prop}
\begin{proof}
 The proof of this result follows by arguing as in \cite{RS95}. This conclude
the proof.
\end{proof}

\section{Friedrichs Extensions: Definition and Basic Results}\label{appendix:friedrichs}

Friedrich extensions provide a canonical way to associate a self-adjoint operator to a densely defined, symmetric, semi-bounded from below operator in a Hilbert space. This construction is particularly relevant in the spectral theory of differential operators, where it often yields the physically or geometrically meaningful self-adjoint realization of a given symmetric operator.

In this section, we recall the definition of the Friedrich extension and collect some basic properties that are crucial for the analysis of self-adjoint realizations of differential operators. We emphasize that the Friedrich extension is characterized from a variational viewpoint as the   largest  self-adjoint extension (in the   usual quadratic-form order ) among all self-adjoint extensions bounded below by the same constant. This property makes the Friedrich extension particularly useful in applications to partial differential equations, quantum mechanics, and index theory.

Let $(H, \langle \cdot, \cdot \rangle$ be a real Hilbert space  and let \( A: \dom(A) < H \to H \) be a densely defined, symmetric operator on a Hilbert space \( H \), satisfying the semiboundedness condition:
\[
\langle A u, u \rangle \ge m \|u\|^2, \quad \forall \, u \in \dom(A),
\]
for some constant \( m \in \mathbb{R} \). Without loss of generality, one may assume \( m > 0 \) by replacing \( A \) with \( A + cI \) for a suitable shift \( c > -m \).

We define the associated quadratic form \( q \) by
\[
q(u, v) := \langle A u, v \rangle, \quad u, v \in \dom(q) := \dom(A).
\]
This form is symmetric and semibounded from below by \( m \). Introduce the form inner product:
\[
\langle u, v \rangle_q := q(u, v) + (1 - m) \langle u, v \rangle,
\]
which induces the form norm:
\[
\|u\|_q := \sqrt{q(u, u) + (1 - m)\|u\|^2}.
\]
Let \( H_q \) denote the completion of \( \dom(A) \) with respect to \( \|\cdot\|_q \). The form \( q \) then extends uniquely to a closed, densely defined, symmetric, and semibounded quadratic form $t_A$ on \( H _q \).

By the Kato's first representation theorem, there exists a unique self-adjoint operator \( A_F \) on \( H  \), called the \emph{Friedrichs extension} of \( A \), such that:
\begin{itemize}
    \item[-] \( \dom(A_F) < \dom(t_A) \),
    \item[-] For all \( u \in \dom(A_F) \), there exists \( f \in H  \) such that
      \[
    t_A(u, v) = \langle f, v \rangle, \qquad \forall\, v \in \dom(t_A),
    \] 
    and \( A_F u := f \).
\end{itemize}
Hence, the domain of the Friedrichs extension is given by:
  \[
\dom(A_F) = \left\{ u \in \dom(t_A) : \exists\, f \in H  \text{ such that } t_A(u, v) = \langle f, v \rangle, \ \forall\, v \in \dom(t_A) \right\}.
\] 
\begin{rem}
The key feature of the Friedrichs extension is that its domain is determined entirely by the closed quadratic form. A function   \( u \in \dom(t_A) \)  lies in \( \dom(A_F) \) if and only if the map \( v \mapsto t_A(u, v) \) is continuous with respect to the Hilbert space norm on \( H  \). In other words, there exists a constant \( C > 0 \) such that
  \[
|t_A(u, v)| \le C \|v\|, \quad \forall v \in \dom(t_A).
\] 
By the Riesz representation theorem, this defines a unique element \( f \in H  \), and we set \( A_F u := f \).

This variational characterization ensures that the Friedrichs extension is the   largest  self-adjoint extension in the sense of form ordering. It typically corresponds to imposing Dirichlet-type boundary conditions or energy minimization constraints.
\end{rem}
More explicitly by using the Kato's  representation theorem for closed, semibounded quadratic forms, we get that  there exists a unique self-adjoint and semibounded operator \( A_F \) on \( H  \) such that:
\begin{itemize}
\item[-] $\dom\left( A_F^{1/2} \right) = \dom(t_A) = H _q$
\item[-] $\langle A_F^{1/2} \phi,\ A_F^{1/2} \psi \rangle = t_A(\phi, \psi)$ for all  $\phi, \psi \in H _q$. 
\end{itemize}
	\begin{lem}\label{thm:lemmaC-2}
  Let $A\in\CFs(H)$ be symmetric and semibounded below. Endow $E:=\dom(t_A)$ with the inner product
\[
\langle u,v\rangle_E:=t_A(u,v)+c\langle u,v\rangle_H,
\]
where $c>0$ is chosen so that $A_F+c\Id\ge\Id$. Then there exists a bounded self-adjoint Fredholm operator $B_A:E\to E$ such that
\[
t_A(u,v)=\langle B_Au,v\rangle_E,\qquad u,v\in E.
\]
 \end{lem}
\begin{proof}
  The form $t_A$ is bounded on $E$, hence the Riesz theorem gives a bounded self-adjoint representative $B_A$. Let $S=(A_F+c\Id)^{1/2}:E\to H$; $S$ is an isomorphism. Under this identification,
\[
S B_A S^{-1}=A_F(A_F+c\Id)^{-1}=\Id-c(A_F+c\Id)^{-1}.
\]
Because  $A_F$ is   a finite-dimensional domain extension of the Fredholm operator $A$, it is Fredholm. Therefore $A_F(A_F+c\Id)^{-1}$ is Fredholm, and so is $B_A$ .
\end{proof}

The main conclusions of the preceding discussion are encapsulated in the following result. 
\begin{thm}{\bf [Friedrich extension]}
 Let \( A: \dom(A) < H \to H \) be a densely defined, symmetric  and non-negative operator on a Hilbert space \( H \). 
	There  is a  unique self-adjoint operator $A_F: \dom(A_F)< H \to H $ such that
	$\dom(A_F)< \dom(t_{A})$ and
	\[
	t_A(u,v)=\langle A_F u,v\rangle_H\qquad \forall\, u\in \dom(A_F) \quad  v\in \dom(t_{A }).
	\] 
	Furthermore,  $A_F$ is a self-adjoint extension of $A$.
\end{thm}
Such a self-adjoint extension of \( A \) is called the \textbf{Friedrichs extension of \( A \)}.

\begin{cor}\label{cor:Friedrich_of_SA}
Let \( A :\dom(A)< H \to H \) be a   semibounded  self-adjoint operator with dense domain \( \dom(A) \). Then the Friedrichs extension of \( A \) coincides with \( A \) itself.
\end{cor}
\begin{rem}
The Friedrichs extension provides a canonical self-adjoint realization of a symmetric and semibounded operator. In particular, if a densely defined symmetric operator \( A \) is essentially self-adjoint and semibounded from below, then its closure is self-adjoint, and the Friedrichs extension coincides with this closure. Thus, in the case of semibounded and essentially self-adjoint operators, the Friedrichs extension reproduces the unique self-adjoint extension.
\end{rem}
The following result is straightforward and the proof is left to the reader. 
\begin{lem}\label{lem:ker_form}
Under the above notation, we get 
\[
\ker t_A= \ker A_F.
\]
\end{lem}
\begin{proof}
    The proof directly follows by a straightforward computation. 
\end{proof}

\begin{lem}\label{lem:form_order}
	Let  $H$ be a Hilbert space and let $U< V < H$ be two  dense subspaces of $H$. We assume that $A_V:V< H\to H$ is a semi-bounded closed and symmetric operator. Let   $A_U=A_V|_U$ . Then we have
	\[
	t_{A_U}=t_{A_V}|_{\dom(t_{A_U})}.
	\]
\end{lem}
\begin{proof}
Without loss of generality, we may assume that $t_{A_V}(u) \ge \|u\|_H^2$.

Note that $\langle A_U\,u, u \rangle_H = t_{A_V}(u)$ for all   $u\in U$ . Therefore,   $\dom(t_{A_U})$  is the closure of $U$  inside $\dom(t_{A_V})$  with respect to the norm $t_{A_U}(\cdot)^{1/2}$, and we have
\[
t_{A_U} = t_{A_V}|_{\dom(t_{A_U})},
\]
concluding the proof. 
\end{proof}

\begin{rem}
By Lemma~\ref{lem:form_order} and by the very definition of the Friedrich extension, we get that 
\[
A^*|_{\dom(A_F)}=A_F.
\]
So, $A_F$ is characterized by its domain. 
\end{rem}
We are now ready to characterize the domain of the Friedrich extension $A_F$ of a closed, Fredholm symmetric and semi-bounded operator in terms of the adjoint domain $\dom(A^*)$ and the form domain $\dom(t_A)$. We state the result below.

\begin{lem}\label{lem:characterize_Dirichlet}
Let $A \in \CFs(H)$ and we assume that $A$ is semi-bounded from below. 
Then, we have 
	\[
	\dom(A^*)\cap \dom(t_A)=\dom( A_F).
	\]	
\end{lem}
\begin{proof}
Consider the canonical projection $\gamma: \dom(A^*) \to \beta(A) := \dom(A^*) / \dom(A)$, and observe that $\dom(A_F) < \dom(t_A)$. Therefore, we have
\[
\dom(A_F) < \dom(t_A) \cap \dom(A^*).
\]
Let $\mu := \gamma(\dom(t_A) \cap \dom(A^*))$. Then $\mu \supset \gamma(\dom(A_F))$. Since $\gamma(\dom(A_F))$ is a Lagrangian subspace of $\beta(A)$, it follows that $\mu$ is a coisotropic subspace of $\beta(A)$. By Lemma~\ref{lem:coisotropic_Lagrangian}, it suffices to show that $\gamma(\dom(A_F))$ is the unique Lagrangian subspace contained in $\mu$.

Let $\Lambda$ be a Lagrangian subspace of $\beta(A)$ such that $\Lambda < \mu$, and set $A_\Lambda := A^*|_{\gamma^{-1}(\Lambda)}$. Then $A_\Lambda$ is a self-adjoint extension of $A$.  The finite-dimensional domain-extension argument in Supplementary Lemma~\ref{lem:local-form-domains} shows that $A_\Lambda$ is semibounded (that argument uses only finite domain codimension).  By Lemma~\ref{lem:form_order}, we have
\[
t_A = t_{A_\Lambda}|_{\dom(t_A)}.
\]
Since $\dom(A_\Lambda) < \dom(t_A)$, it follows that
\[
t_A|_{\dom(A_\Lambda)} = t_{A_\Lambda}|_{\dom(A_\Lambda)}.
\]
Thus,
\[
t_A(x, y) = \langle A_\Lambda x, y \rangle_H, \quad \forall x, y \in \dom(A_\Lambda).
\]
Because $\dom(A) < \dom(A_\Lambda) < \dom(t_A)$ and $\dom(A)$ is dense in the Hilbert space $\dom(t_A)$, it follows that $\dom(A_\Lambda)$ is also dense in $\dom(t_A)$.

Hence, for each $u \in \dom(t_A)$, there exists a sequence $(u_n)_n < \dom(A_\Lambda)$ such that $u_n \to u$ both in $\dom(t_A)$ and in $H$. Therefore,
\[
t_A(x, y) = \langle A_\Lambda x, y \rangle_H, \quad \forall x \in \dom(A_\Lambda) \text{ and } y \in \dom(t_A).
\]
This shows that $A_\Lambda$ is the unique Friedrich extension of $A$. Consequently, $\gamma(\dom(A_\Lambda))$ is the unique Lagrangian subspace contained in $\mu$.

By Lemma~\ref{lem:coisotropic_Lagrangian}, we finally conclude that
\[
\dom(A_F) = \dom(t_A) \cap \dom(A^*),
\]
which completes the proof.
\end{proof}

\begin{ex}
We consider the Sturm--Liouville operator $l = -\dfrac{d^2}{dt^2}$ defined on   $\mathscr \mathscr C_0^\infty((0,1), \mathbb{R}^n)$ . Then, we have
\[
\dom(A_F) = \dom(A^*) \cap \dom(t_A) = W^{2,2}([0,1], \mathbb{R}^n) \cap W^{1,2}_0([0,1], \mathbb{R}^n).
\]
In this case, $\dom(A_F)$ provides the appropriate functional setting for treating the boundary value problem with Dirichlet boundary conditions. 

We refer to the general boundary conditions parametrized by the quotient space $\dom(A_F)/\dom(A)$ as the {\sc general Dirichlet boundary conditions}.
\end{ex}

\section{Auxiliary Results}\label{sec:postponed-proofs}

This section contains a number of auxiliary results used for proving the main result of the manuscript. We collect all the results in this appendix, for  smoothing the reading of the manuscript. 

 \subsection{Miscellaneous Results}

\begin{lem}\label{lem:continu_subspace}
If  $s\mapsto E_s$ and $s\mapsto F_s$ are two gap-continuous  paths in $\Grass(H)$ such that $s\mapsto E_s+F_s$ is also gap-continuous  in $\Grass(H)$. Then $s \mapsto E_s\cap F_s$ is a gap-continuous  path in $\Grass(H)$.

		If $s\mapsto E_s\cap F_s$, $s\mapsto E_s $, $s\mapsto F_s$  are gap  continuous and $E_s+F_s$ is closed for every $s\in [0,1]$, then the path $s\mapsto E_s+F_s$ is gap  continuous in $\Grass(H)$.	
\end{lem}
\begin{proof}
	For the proof of this result we refer to \cite[Corollary A.3.14]{BZ18} and references therein.
\end{proof}
\begin{cor}\label{cor:continu_ker}
Let \(E, F\) be Hilbert spaces. Let \(A_{s}: E \rightarrow F, s \in[0,1]\) be a gap-continuous  path of operators. Assume that im \(A_{s}=F, s \in[0,1]\). Then \(\operatorname{ker} A_{s}\) is a gap-continuous  path in \(\mathcal{S}(E)\).
\end{cor}
\begin{proof}
    Without loss of generality, we can assume \(E, F\) be subspace of \(E \oplus F\). We have \(\mathcal{G}\left(A_{s}\right)+E=\) \(E+\operatorname{im}\left(A_{s}\right)=E \oplus F\). By Lemma~\ref{lem:continu_subspace}, 
    $\ker A_{s}=\mathcal{G}\left(A_{s}\right) \cap E$ is a gap-continuous  path in \(\mathcal{S}(E)\).
\end{proof}

\begin{lem}\label{lem:dim_factor_space}
	Let $H$ be a Hilbert space and for $i=1,2$ we consider the   symmetric and Fredholm operators 
    $T_i: \dom(T_i)< H\to H$.  We assume that 
	\begin{itemize} 
    \item[-] $\dom(T_1)< \dom(T_2)$ 
    \item[-] $T_1=T_2|_{\dom(T_1)}$.
    \end{itemize} 
	Then we have
	\[
	\dim (\dom(T_2)/\dom(T_1)) =\ind(T_2)-\ind(T_1).
	\]
\end{lem}
\begin{proof}
	Let  $j$ be the injection $\dom(T_1)\to \dom(T_2)$.
	Then we have
	\[
\ind(T_1)=\ind(T_2\, j)=\ind(T_2)+\ind(j)=\ind(T_2)-\coker j=\ind(T_2)-\dim (\dom(T_2)/\dom(T_1)).
	\]	
\end{proof}

\begin{lem}\label{lm:ess_spec_same}
	Let $T \in \CFs(H)$  be densely defined, $T_1$ be a self-adjoint extension of $T$ and we assume  that $\dim\big(\dom(T_1) /\dom (T)\big)<+\infty$.  Then, we have 
	\[
	\sigma_{ess}(T_1)=\sigma_{ess}(T)
	\]
	where $\sigma_{ess}(\#)$ denotes  the essential spectrum of the operator $\#$.
\end{lem}
\begin{proof}
  After complexification, fix $\lambda\in\mathbb C$. The inclusion $j:\dom(T)\to\dom(T_1)$ is bounded between graph-norm spaces, injective, and has finite-dimensional cokernel. Since $T-\lambda\Id=(T_1-\lambda\Id)j$, the two-out-of-three property for bounded Fredholm operators shows that $T-\lambda\Id$ is Fredholm if and only if $T_1-\lambda\Id$ is Fredholm. This proves equality of the Fredholm essential spectra. The adjoint realization need not be symmetric, and no such assertion is used.
 \end{proof}

 In the following three form-domain statements and in Lemma~\ref{lem:index_orth_compl}, assume that $A$ is closed, densely defined, symmetric, semibounded and Fredholm, and that (H0) holds. Write $\dom(A^*)=\dom(A)\oplus U$ and take $\Lambda\in\Lag(U,\rho)$. All extensions are semibounded by the finite-codimension argument of Supplementary Lemma~\ref{lem:local-form-domains}.

 \begin{lem}\label{lem:orth_complement_t_L}
Let $A \in \CFs(H)$ and let   $\Lambda\in\Lag(U,\rho)$ . Under the notation above, the following identity holds:
\[
(\dom(t_A))^{t_{A_\Lambda}} = (\dom(A_F) + \dom(A_\Lambda)) \cap \ker A^*.
\]
\end{lem}

\begin{proof}
Since $\dom(A)$ is dense in the Hilbert space $\dom(t_A)$ and $t_{A_\Lambda}$ is a bounded form on $\dom(t_{A_\Lambda})$, we have
\[
(\dom(t_A))^{t_{A_\Lambda}} = (\dom(A))^{t_{A_\Lambda}} = \{ y \in \dom(t_{A_\Lambda}) \mid t_{A_\Lambda}[x, y] = 0,\ \forall x \in \dom(A) \}.
\]
Recalling the definition of the Friedrich extension, for every $x \in \dom(A)$ we have
\[
t_{A_\Lambda}[x, y] = \langle A_\Lambda x, y \rangle_H = \langle A x, y \rangle_H.
\]
It follows that
\[
(\dom(t_A))^{t_{A_\Lambda}} = (\operatorname{im} A)^\perp \cap \dom(t_{A_\Lambda}) = \ker A^* \cap \dom(t_{A_\Lambda}) = \ker A^* \cap \dom(t_{A_\Lambda}) \cap \dom(A^*).
\]
By Lemma~\ref{lem:characterize_form_domain}, we have
\[
\dom(t_{A_\Lambda}) = \dom(A_\Lambda) + \dom(t_A).
\]
Therefore,
\begin{align*}
\ker A^* \cap \dom(t_{A_\Lambda}) \cap \dom(A^*)
&= \ker A^* \cap (\dom(A_\Lambda) + \dom(t_A)) \cap \dom(A^*) \\
&= \ker A^* \cap (\dom(A_\Lambda) + \dom(A^*) \cap \dom(t_A)) \\
&= \ker A^* \cap (\dom(A_\Lambda) + \dom(A_F)).
\end{align*}
Applying Lemma~\ref{lem:characterize_Dirichlet}, the result follows.
\end{proof}
\begin{lem}\label{lem:characterize_form_domain}
Let $A\in \CFs(H)$ and let   $\Lambda\in\Lag(U,\rho)$ . Then we have
	\[
	\dom(A_\Lambda)+\dom(t_A)=\dom(t_{A_{\Lambda}} )
	\]
\end{lem}
\begin{proof}
	Since we have $\dom(A_\Lambda)< \dom(t_{A_\Lambda})$, and $\dom(t_A)< \dom t_{A_\Lambda}$ we only need to show that 
	\[
	\dom(A_\Lambda)+\dom(t_A)\supset \dom(t_{A_{\Lambda}} ).
	\]
	We choose $V$ such that $\dom(A_\Lambda)=V\oplus \dom(A)$ and 
	since $\dom(t_A)$ is closed in the Hilbert space $\dom(t_{A_\Lambda})$ and $V$ is finite dimensional, then $\dom(t_A)+V$ is also closed in $\dom(t_{A_\Lambda})$.
	Then we have
	\[
	\dom(A_\Lambda)+\dom(t_A)\supset \dom(t_A)+V\supset \dom(A)+V=\dom(A_\Lambda) \Rightarrow \dom(t_A)+\dom(A_\Lambda) \supset \overline {\dom(A_\Lambda)}=\dom(t_{A_\Lambda})
	\]
	where we denoted by $\overline{\#}$ the topological closure of subspace in $\dom(t_{A_\Lambda})$.
	\end{proof}

\begin{cor}
Let $A\in \CFs(H)$ and let   $\Lambda\in\Lag(U,\rho)$ . Then we have
	\[
	\dom(A_\Lambda)\cap \dom(t_A)=\dom(A_\Lambda)\cap \dom(A_F).
	\]	
\end{cor}
\begin{proof}
	By Lemma~\ref{lem:characterize_Dirichlet}, it follows  that 
	\[
	\dom( A_\Lambda)\cap \dom(t_A)=\dom(A_\Lambda)\cap \dom(A^*)\cap \dom(t_A)= \dom(A_\Lambda) \cap \dom(A_F).
	\]
	The result then follows.
\end{proof}

\subsection{Morse indices of forms and of Friedrich extensions}
We now compare the Morse index of the operator under Dirichlet boundary conditions with that under a general self-adjoint boundary condition.

Let $A_\Lambda$ be a self-adjoint extension of $A$. The {\sc Morse index of $A_\Lambda$}, denoted by $\iMor(A_\Lambda)$, is defined as the   dimension of its negative spectral subspace, as in Notation~\ref{def:global-morse}. When the negative spectrum is discrete, this is the negative-eigenvalue count with multiplicity . Associated with $A_\Lambda$ is the closed quadratic form $t_{A_\Lambda}$. The {\sc Morse index of this quadratic form} is defined as the maximal dimension of a subspace of $\dom(t_{A_\Lambda})$ on which $t_{A_\Lambda}$ is negative definite.

\begin{lem}\label{lem:estimate_mor_subspace}
Let $q$ be a quadratic form on $V$ and we assume that $\dim(V/V_1) <+\infty$. 
Then   $\iMor(q|_{V_1})\le\iMor(q)\le\iMor(q|_{V_1})+\dim(V/V_1)$, with infinite values allowed
 \end{lem}
\begin{proof}
Let $W$ be   any finite-dimensional  negative subspace of $q$ , and take the supremum of its possible dimensions .
Then $W\cap V_1$ is a negative subspace of $q|_{V_1}$.
We observe that $\dim W/(W\cap V_1) = \dim (V_1+W)/V_1\le \dim (V/V_1)\le +\infty$. This concludes the proof. 
\end{proof}

\begin{prop}\label{prop:estimate_morse_self_ajoint}
Let $A_1,A_2\in \Cl^{sa}(H)$ be two closed and self-adjoint operators. We assume that $k_i:=\dim \dom(A_i)/(\dom(A_1)\cap \dom (A_2))<+\infty$, for $i=1,2$ and
\[
A_1|_{\dom (A_1)\cap\dom (A_2)}=A_2|_{\dom(A_1)\cap \dom(A_2)}
\]
We have   $\iMor(A_1)\le\iMor(A_2)+k_1$ and $\iMor(A_2)\le\iMor(A_1)+k_2$, including infinite indices
 \end{prop}
\begin{proof}
Let $q_i$ be the quadratic form $\langle A_i u,u\rangle$ on $\dom (A_i)$, $i=1,2$.
By Lemma~\ref{lem:estimate_mor_subspace}, we have
  \[
\iMor(q_i)\le\iMor(q_i|_{\dom A_1\cap \dom A_2})+k_i.
\] 
Since $\iMor(A_i)=\iMor(q_i)$, then the  proposition follows.
\end{proof}

\begin{lem}\label{thm:general_morse_diff}
	Let $Q$ be a   bounded symmetric form on a Hilbert  space $V$ , represented by a bounded self-adjoint Fredholm operator, with finite Morse index . Let $G$ be a closed subspace of  finite codimension in  $V$ and let 
	  \[
	G^Q=\set{v\in V| b_Q(w,v)=0, \forall w\in G }
	\] 
    where $b_Q$ is the bilinear form induced by $Q$ through the polarization identity and we set 
	\[
	\ker Q= V^Q .
	\]
  By Lemma~\ref{lem:double_Q_orth_completion},  $G^{QQ}=G+\ker Q$. Then the following formula holds:
	\[
	\iMor(Q|_V)-\iMor(Q|_G)=\iMor(Q|_{G^Q}) +\dim (G\cap G^Q+\ker Q)/\ker Q.
	\]
\end{lem}
\begin{proof}
    We refer the interested reader to \cite[Theorem 3.1]{HWY20} for the proof. 
\end{proof} 

\begin{lem}\label{thm_=Morse}
	Let $A:\dom(A)< H \to H$ be closed a self-adjoint operator with finite Morse index.
	Then, we have
	\[
	\iMor(A)=\iMor(t_A).
	\]
\end{lem}
\begin{proof}
 With the spectral definition of the index, finiteness of the negative spectral subspace implies that $A$ is semibounded. Its negative spectrum is a finite list of eigenvalues. In the list below eigenvalues are listed without repetition and their multiplicities are included in the dimensions of the eigenspaces.
 Let $\lambda_1, \ldots, \lambda_k$ be the negative eigenvalues of $A$, and let $W = \bigoplus_{i=1}^k \ker(A - \lambda_i \Id)$. Then $t_A|_W$ is a negative definite quadratic form. Therefore, we have
\[
\iMor(t_A) \ge \dim W = \iMor(A).
\]

Now, $\dom(t_A)$ is a Hilbert space with inner product given by $t_A(\cdot, \cdot) + c\langle \cdot, \cdot \rangle_H$, for some sufficiently large constant $c > 0$. For brevity, we denote $E := \dom(t_A)$, and let $\|\cdot\|_E$ be the norm induced by this inner product. Then $t_A$ defines a continuous   symmetric bilinear form  on $E$.

Let $V$ be an $n$-dimensional negative subspace for $t_A$, and let $\{v_1, \ldots, v_n\}$ be a basis of $V$. Since $\dom(A)$ is dense in $E = \dom(t_A)$, for every $\delta > 0$, there exist elements $u_1, \ldots, u_n \in \dom(A)$ such that $\|v_i - u_i\|_E < \delta$ for all $i$.

Because $t_A|_V$ is negative definite, the matrix $(t_A(v_i, v_j))$ is negative definite. Moreover, since $t_A$ is continuous on $E$, the matrix $(t_A(u_i, u_j))$ is also negative definite for $\delta$ sufficiently small.

Since   $t_A(u_i,u_j)=\langle Au_i,u_j\rangle_H$ , by the Min--Max Theorem for self-adjoint operators, it follows that the $n$-th eigenvalue of $A$ satisfies
\[
\lambda_n \le \max\left\{ \langle Au, u \rangle_H \,\middle|\, u \in \mathrm{span}\{u_1, \ldots, u_n\}, \ \|u\|_H = 1 \right\} < 0.
\]
Thus, we conclude that $n \le k$, and hence $\iMor(t_A) \le \iMor(A)$. Combining both inequalities, we obtain $\iMor(t_A) = \iMor(A)$, completing the proof.
\end{proof}

We are now in a position to compute the Morse index
\[
\iMor\left(t_{A_\Lambda} \big|_{(\dom(t_A))^{t_{A_\Lambda}}} \right).
\]

\begin{lem}\label{lem:index_orth_compl}
Given the decomposition $\dom(A^*)=\dom(A)\oplus U$
Let $p$  be the canonical projection to U.
We have
\[
\iMor\left(t_{A_\Lambda} \big|_{(\dom(t_A))^{t_{A_\Lambda}}} \right)
= \coiMor\big[ Q(p(\ker A^*), p(\dom(A_\Lambda)); p(\dom(A_F)) \big].
\]
\end{lem}
\begin{proof}
Without loss of generality. We can assume that $\dom(A^*)=\dom(A)\oplus U$ such that
$\ker A^*< U$, $\Lambda< U$,$F< U$,$A_\Lambda=A^*|_{\dom(A)\oplus\Lambda}$ and 
$A_F=A^*|_{\dom(A)\oplus F}$.
Then we only need to prove that 
\[
\iMor\left(t_{A_\Lambda} \big|_{(\dom(t_A))^{t_{A_\Lambda}}} \right)
= \coiMor\big[ Q(\ker A^*, \Lambda; F) \big].
\]
By Lemma \ref{lem:orth_complement_t_L}, 
\[(D(t_A)^{t_{A_\Lambda}})=\ker A^*\cap(\dom (A_{\Lambda})+\dom{A_F})=\ker A^*\cap U\cap (D(A_{\Lambda})+D(A_F))=\ker A^*\cap (\Lambda+F).
\]
Let $z_1, z_2 \in \ker A^* \cap (\Lambda+F)$, and write $z_i = p_i + q_i$ with $p_i \in \Lambda$ and $q_i \in F$.

Then:
\[
t_{A_\Lambda}(z_1, z_2)
= t_{A_\Lambda}(p_1, p_2) + t_{A_\Lambda}(p_1, q_2)
+ t_{A_\Lambda}(q_1, p_2) + t_{A_\Lambda}(q_1, q_2).
\]

Since $A_\Lambda$ is the Friedrich extension of itself, we have:
\[
t_{A_\Lambda}(u, v) = \langle A_\Lambda u, v \rangle_H, \quad \forall\, u \in \dom(A_\Lambda),\ v \in \dom(t_{A_\Lambda}).
\]
Moreover, since $t_{A_\Lambda}|_{\dom(t_A)} = t_A$, it follows that:
\[
t_{A_\Lambda}(u, v) = t_A(u, v) = \langle A_F u, v \rangle_H, \quad \forall\, u, v \in \dom(A_F).
\]

Therefore:
\begin{align*}
t_{A_\Lambda}(z_1, z_2)
&= \langle A_\Lambda p_1, p_2 \rangle
+ \langle A_\Lambda p_1, q_2 \rangle
+ \langle q_1, A_\Lambda p_2 \rangle
+ \langle A_F q_1, q_2 \rangle \\
&= \langle A^* p_1, p_2 \rangle
+ \langle A^* p_1, q_2 \rangle
+ \langle q_1, A^* p_2 \rangle
+ \langle A^* q_1, q_2 \rangle \\
&= \langle A^*(p_1 + q_1), p_2 + q_2 \rangle
- \langle p_1, A^* q_2 \rangle
+ \langle A^* p_1, q_2 \rangle.
\end{align*}

Since $z_1 = p_1 + q_1 \in \ker A^*$, we have $A^*(p_1 + q_1) = 0$. Thus,
\[
t_{A_\Lambda}(z_1, z_2)
= -\langle p_1, A^* q_2 \rangle
+ \langle A^* p_1, q_2 \rangle
= -\omega(p_1, q_2).
\]

Thus, we conclude that
\begin{equation*}    
\iMor\left(t_{A_\Lambda} \big|_{(\dom(t_A))^{t_{A_\Lambda}}} \right)
= \coiMor\big[ Q(\ker A^*, \Lambda;F) \big],
\end{equation*}
as claimed.
\end{proof}

\begin{lem}\label{lem:more_index_zero}
Let   $t$ be a bounded symmetric form on a Hilbert space $E$, represented by a bounded  self-adjoint  Fredholm operator,   and suppose $\iMor(t)<\infty$. Let $E_s\subset E$ be closed subspaces such that $E_r\subset E_s$  for   $a<r<s<b$ and $\bigcap_{s>a}E_s=0$. Then $\iMor(t|_{E_r})\le\iMor(t|_{E_s})$, and $\iMor(t|_{E_{s_0}})=0$ for some $s_0>a$.
\end{lem}
\begin{proof}
Monotonicity is immediate by inclusion of negative subspaces. Choose $s_j\downarrow a$ (or $s_j\to-\infty$ if $a=-\infty$). The closed subspaces $E_{s_j}$ decrease with zero intersection, so Lemma~\ref{lem:Morse_vanish_abstract} gives the vanishing assertion. In the application to truncated operators, Supplementary Lemma~\ref{lem:local-form-domains} supplies precisely the common form, the closed embeddings and the zero intersection; Lemma~\ref{thm:lemmaC-2} supplies its Fredholm representative.
 \end{proof}

\begin{lem}\label{lem:shrink_subspace_image}
		Let $H$ be a Hilbert space, let $(V_n)_{n\in \N}$ be a sequence of closed (nested) subspaces such that 
		\[
        V_{k+1}< V_k \qquad \textrm{ and }\qquad 
        \bigcap_{n\ge 1} V_n =(0)
        \]
		and let $T\in \BF(H)$ be a bounded Fredholm operator.
		Then  for every $n \ge 1$ the subspace $T(V_n)$ is closed  and
		\[
		\bigcap_{n\ge 1} T(V_n) =(0).
		\]
	\end{lem}
	\begin{proof}
		Since $T$ is a bounded and Fredholm operator, then  $T:(\ker T)^\perp \to \image T$ is a  linear isomorphism.
		Let  $P$ be the orthogonal projection onto $(\ker T)^\perp$ and we observe that the following holds: 	$T=T|_{(\ker T)^\perp}P:H \to \image T$. Since $T|_{(\ker T)^\perp}$ is an isomorphism for proving  the result, we only need to show that $\bigcap_{n\ge 1} P(V_n) =\set 0$ and that $P(V_n)$ is closed.

		Let $W=\ker T=\ker P$. Since $T$ is a Fredholm operator, then $\dim W<+\infty$.
		Let $Pu\in \bigcap_{n\ge 1}P(V_n)$.  Then  for every $n\ge 1$, we get that $u\in V_n+ W$.  
        Moreover $\bigcap_{n\ge 1} (V_n\cap W)< \bigcap_{n\ge 1}V_n=(0)$. Now, since $W$ is finite dimensional, there exists $m\in \N$ such that $V_n\cap W=(0)$ for each $n>m$. Without loss of generality, we can assume that $V_n\cap W=(0)$ for $n\ge 1$.
		Since $V_{n+1}< V_n$,  the decomposition of $u$ in $V_n+W$ is unique. This in particular implies that  $u\in \bigcap_{n\ge 1}V_n\oplus W=W$. By this we get that  $Pu=0$ and and thus  $\bigcap_{n\ge 1} P(V_n) =(0)$.

                Now, since $V_n$ is closed and $W$ is finite dimensional, $V_n+W$ is closed and   $P(V_n) = (V_n+W)\cap W^\perp$ is closed, too. This concludes the proof. 
		 	\end{proof}

	\begin{lem}\label{lem:Morse_vanish_abstract}
	Let $(V_n)_{n \in \N}$ be a   decreasing  sequence of closed subspaces of the of Hilbert space $H$ and we assume that $\bigcap_{n \ge 1} V_n=(0)$.
	Let $B\in \BFsa(H)$ such that $\iMor(B)<\infty$ and let $q(\cdot)=\langle B\cdot ,\cdot \rangle$ be the corresponding quadratic form. Then there exists $m>0$ such that 
    \[
    \iMor(q|_{V_m}) =0.
    \]
	\end{lem}
	\begin{proof}
		Let $u\in (B(V_n))^\perp$ and $v\in  V_n$. Then we have 
		\[
		b(u,v)=\langle Bu,v\rangle=\langle u,Bv\rangle=0
		\]
        where $b$ is the bilinear form associated to $q$ through the polarization identity. 	This, in particular, implies that for every $n \ge 1$, the following inequality holds
		\begin{equation}\label{eq:inequality-1}
		\iMor (q) \ge \iMor (q|_{V_n +(B(V_n))^\perp})\ge \iMor (q|_{V_n})+\iMor (q|_{(B(V_n))^\perp}).
		\end{equation}
        For concluding the proof we only need to prove that there exists  $m\in \N$ such that 
		 $\iMor (q)\le \iMor (q|_{(B(V_m))^\perp}) $.

		Let $P_n$ be the orthogonal projection onto $(B(V_n))^\perp$. By Lemma~\ref{lem:shrink_subspace_image}, we get that 
        \[
        \lim_{n\to +\infty}(\Id-P_n)x=0\qquad \textrm{ for each } \quad 
        x\in H.
        \]
		We denote by $W=\Span\{x_1,\cdots,x_k\}$ be the negative spectral space of $q$ and we observe that $ \lim_{n\to +\infty} b(P_n x_i,P_n x_j)=b(x_i,x_j)$.
		So, there exists $m>0$ such that  the matrix   $A_{ij}:=b(P_m x_i,P_m x_j)$   is negative definite and by this we get that 
	\begin{equation}\label{eq:inequality-2}
    \iMor (q|_{(B(V_m))^\perp}) \ge \iMor (q).
    \end{equation}
By adding the inequalities provided at Equation~\eqref{eq:inequality-1} and  Equation~\eqref{eq:inequality-2}, the conclusion  readily follows. 
	\end{proof}
\begin{ex}
We assume that the SL-operator $ L :\dom( L ^*)< L^2((a,b) ,\R^n)\to L^2((a,b) ,\R^n)$    has $P,Q\in C^1([a,b])$, continuous $R$, and  $P$   uniformly positive definite on $[a,b]$ . In particular under this assumptions  the   minimal symmetric operator and its self-adjoint extensions are  bounded from below ; this is not a semiboundedness assertion for the nonsymmetric maximal operator . 
Given the following decomposition  $\dom( L ^*)=\dom(  L )\oplus U$, let $\Lambda$ be any Lagrangian subspace of $U, \rho)$ and let $\Lambda_D< U$ be the Dirichlet Lagrangian. 
	It follows that  
      \[
    \dom( L ^*)=W^{2,2}([a,b],\R^n)\qquad \textrm{ and } \qquad \dom(L_{\Lambda_D})=W^{2,2}([a,b],\R^n)\cap W^{1,2}_0([a,b],\R^n).
    \] 
	Let $E$ be an isotropic subspace of $(U, \rho)$ and we observe that for every $u,v\in \dom( L_E)$, the bilinear form  defined by  $\langle L_E u,v\rangle_{L^2}$  is   symmetric; after integration by parts its boundary term depends on the finite-dimensional endpoint traces and it is bounded in the form norm .
	By taking the  closure of $\dom( L_E)$ in $W^{1,2}([a,b],\R^n)$, we extend this symmetric bilinear form  continuously  and  we denote it by $t_{ L _E}$. So, in particular 
	$t_{ L _E}(\cdot,\cdot)$ is continuous on its domain denoted by  $\dom(t_{ L _E})$.
Moreover, we have
    \[
	\dom (t_{ L _{D}})=W_0^{1,2}([a,b],\R^n) \qquad \textrm{ and }  \dom (t_{ L _E})=\dom(t_{ L _{D}})+E.
	\]
	\end{ex}

	\begin{lem}
	Assume that $\Lambda$ is Lagrangian. We have  
	$\ker t_{ L _\Lambda}=\ker  L _\Lambda$.
	\end{lem}
	\begin{proof}
		$t_{ L _\Lambda}(u,v)=0$ for all $v\in \dom(t_{ L _\Lambda})$ is equivalent to  $t_{ L _\Lambda}(u,v)=0$ for all  $v\in \dom (L _\Lambda)$  since $\dom(L _\Lambda)$ is dense in $\dom( t_{ L _\Lambda})$.
		Similarly, it also equivalent  to $\langle u, L _\Lambda v\rangle_{L^2}=0$.
		So, we have
		 \[\ker t_{ L _\Lambda}=\dom (t_{ L _\Lambda})\cap (\image  L _{\lambda})^\perp=(W_0^{1,2}((a,b) ,\R^n)+\Lambda)\cap \ker  L _\Lambda=\ker  L _\Lambda.
		 \]
	\end{proof}
	\begin{lem}\label{thm:ghost}
		$\iMor(t_{L_\Lambda})=\iMor(L_\Lambda)$
		\end{lem}
	\begin{proof}
		For each eigenvector $u$ of $L_\Lambda$ corresponding to a  negative eigenvalue, we get that  
        \[
        t_{L_\Lambda}(u,u)=(L_\Lambda u, u)<0.
        \]
        By this, it follows that $\iMor(L_\Lambda)\le  \iMor(t_{L_\Lambda})$.

		Conversely, let $V$ be a finite dimensional subspace of  $\dom(t_{L_\Lambda})$ where the form $ t_{L_\Lambda}$  is  negative definite and let $(v_i)_i$ be its basis.
		For each $\varepsilon>0$, there exists $(w_i)_i\subset \dom (L_\Lambda)$ such that 
		$\|v_i-w_i\|_{W^{1,2}}<\varepsilon$. For $\varepsilon$ small enough, $t_{L_\Lambda}|_{\Span\set{w_i}}=(L_\Lambda w_i, w_i)$ is negative definite. By this it follows that $\iMor(L_\Lambda)\ge  \iMor(t_{L_\Lambda})$ concluding the proof. 
	\end{proof}	

\begin{lem}\label{lem:double_Q_orth_completion}
	Let $H$ be a Hilbert space,  $B\in \BFsa(H)$ and let $Q$ be a quadratic form defined by $Q(u)=\langle Bu,u\rangle$. Let $V$ be a closed subspace of $H$. Then we have 
    \[
    V^{QQ}=V+\ker Q.
    \]
\end{lem}
	\begin{proof}
	We have	$V^Q=\set{u\in H|\langle Bu,v\rangle=0 \textrm{ for all } v\in V}=B^{-1}(V^\perp)$. Since $B$ is self-adjoint , we have $\langle Bu,v\rangle =\langle u,Bv\rangle $ and so, $V^Q =(B(V))^\perp$.
	Then we can conclude that 
	\[
	V^{QQ}= B^{-1}((B(V))^{\perp\perp})=B^{-1}(\overline{B(V)}).
	\]
	Since $B$ is a Fredholm operator, then $B(V)$ is closed. Then we have
	\[
V^{QQ}=B^{-1}(B(V))=V+\ker B.
	\]
	\end{proof}

\vskip.5truecm

\noindent

\begin{flushleft} 
Prof. Xijun Hu\\
School of Mathematics, \\
Shandong University\\
Jinan, 250100, P. R. China \\
E-mail: \texttt{xjhu@sdu.edu.cn}
\end{flushleft}

\vskip.5truecm

\begin{flushleft}
Prof. Alessandro Portaluri\\
Università degli Studi di Torino (DISAFA)\\
Largo Paolo Braccini, 2 \\
10095 Grugliasco, Torino (Italy)\\
Website: \texttt{https://portalurialessandro.wordpress.com}\\
E-mail: \texttt{alessandro.portaluri@unito.it}\\
\medskip
Visiting Professor of Mathematics\\
New York University (Abu Dhabi)\\
Saadiyat Marina District - Abu Dhabi (UAE)\\
E-mail: \texttt{ap9453@nyu.edu}\\
\end{flushleft}

\vskip1truecm

\begin{flushleft} 
Prof. Li Wu \\
School of Mathematics  \\
Shandong University\\
Jinan, 250100, P. R. China \\
E-mail: \texttt{vvvli@sdu.edu.cn}
\end{flushleft}

\end{document}